\documentclass[11pt,reqno]{amsart}

\usepackage[latin1]{inputenc}
\usepackage{amsmath,amsthm,amssymb,amscd,mathrsfs}
\usepackage{mathtools,color,esint,bm}
\usepackage[perpage]{footmisc}
\usepackage{subfig}

\definecolor{darkgreen}{rgb}{0.0, 0.6, 0.13}

\usepackage{hyperref}

\newtheorem{thm}{Theorem}[section]
 
 \newtheorem{lem}[thm]{Lemma}
 \newtheorem{prop}[thm]{Proposition}

 \theoremstyle{definition}
 
 \theoremstyle{remark}
 \newtheorem{rem}[thm]{Remark}
 
 \numberwithin{equation}{section}
 
\usepackage{enumerate} 
 
 \newcommand{\les}{\lesssim}

\newcommand{\Cb}{\mathbb C}

\newcommand{\Eb}{\mathbb E}

\newcommand{\Pb}{\mathbb P}

\newcommand{\Rb}{\mathbb R}

\newcommand{\Tb}{\mathbb T}

\newcommand{\Zb}{\mathbb Z}
\newcommand{\Ac}{\mathcal A}

\newcommand{\Cc}{\mathcal C}

\newcommand{\Fc}{\mathcal F}

\newcommand{\Ic}{\mathcal I}

\renewcommand{\Mc}{\mathcal M}
\newcommand{\Nc}{\mathcal N}

\newcommand{\Qc}{\mathcal Q}
\newcommand{\Rc}{\mathcal R}

\newcommand{\Vc}{\mathcal V}

\newcommand{\Xc}{\mathcal X}
\newcommand{\Yc}{\mathcal Y}

\newcommand{\As}{\mathscr A}
\newcommand{\Bs}{\mathscr B}

\newcommand{\Hs}{\mathscr H}

\newcommand{\Ls}{\mathscr L}
\newcommand{\Ms}{\mathscr M}

\newcommand\<{\langle}
\renewcommand\>{\rangle}

\begin{document}
\thanks{This article is dedicated to the memory of Jean Bourgain}
\author[Yu Deng]{Yu Deng$^1$}
\address{$^1$ Department of Mathematics, University of Southern California, Los Angeles,  CA 90089, USA }
\email{yudeng@usc.edu}
\thanks{$^1$Y. D. is funded in part by NSF DMS-1900251.}
\author[Andrea R. Nahmod]{Andrea R. Nahmod$^2$}
\address{$^2$ 
Department of Mathematics,  University of Massachusetts,  Amherst MA 01003}
\email{nahmod@math.umass.edu}
\thanks{$^2$A.N. is funded in part by NSF DMS-1800852 and the Simons Foundation Collaborations Grant on Wave Turbulence (Nahmod's Award ID 651469).}
\author[Haitian Yue]{Haitian Yue$^3$}
\address{$^3$Department of Mathematics, University of Southern California, Los Angeles,  CA 90089, USA}
\email{haitiany@usc.edu}
  \date{}

\title[Gibbs measure for Hartree]{Invariant Gibbs measure and global strong solutions for the Hartree NLS equation in dimension three}

\begin{abstract} In this paper we consider the defocusing Hartree nonlinear Schr\"odinger equations on $\mathbb T^3$ with real valued and even  potential $V$ and Fourier multiplier decaying like $|k|^{-\beta}$. By relying on the method of random averaging operators \cite{DNY1}, we show that there exists  $\frac{1}{2} \ll \beta_0 <1 $ such that for $ \beta > \beta_0 $ we have invariance of the associated Gibbs measure and global existence of strong solutions in its statistical ensemble. In this way we extend Bourgain's seminal result \cite{Bourgain97} which requires $\beta >2$ in this case.
\end{abstract}

\maketitle

\section{Introduction} In this paper we study the invariant Gibbs measure problem for the nonlinear Schr\"{o}dinger (NLS) equation on $\Tb^3$ with Hartree nonlinearity. Such equation takes the form
\begin{equation}\label{plainHartree}
\left\{
\begin{aligned}
(i\partial_t  + \Delta) u& =  (|u|^2* V)u,\\
u(0)&=u_{\mathrm{in}},
\end{aligned}
\right.
\end{equation} where $V$ is a convolution potential. We will assume that it satisfies the following properties:
\begin{itemize}
\item That $V$ is real-valued and even, and so is $\widehat{V}$;
\item That (\ref{plainHartree}) is \emph{defocusing}, i.e. $V\geq 0$;
\item That $V$ acts like $\beta$ antiderivatives, i.e. $\widehat{V}(0)=1$ and $|\widehat{V}(k)|\lesssim\<k\>^{-\beta}$ for some $\beta\geq0$.
\end{itemize} A typical example for such $V$ is the Bessel potential $\<\nabla\>^{-\beta}$; note that when $V$ is the $\delta$ function (and $\beta=0$) we recover the  usual cubic NLS equation. Our main result, see Theorem \ref{main} below, establishes invariance of Gibbs measure for (\ref{plainHartree}) when $\beta<1$ and is close enough to $1$, greatly improving the previous result of Bourgain \cite{Bourgain97} which assumes $\beta>2$.
\subsection{Background} The equation (\ref{plainHartree}) can be viewed as a regularized or tempered version of the cubic NLS equation, and both naturally arise in the limit of quantum many-body problems for interacting bosons (see e.g. \cite{FKSS, Soh} and references therein). An important question, both physically and mathematically, is to study the construction and dynamics of the \emph{Gibbs measure} for (\ref{plainHartree}), which is a Hamiltonian system.
\subsubsection{Gibbs measure construction} The Gibbs measure, which we henceforth denote by $\mathrm{d}\nu$, is formally expressed as
\begin{equation}\label{naivegibbs}  \mathrm{d}\nu=e^{-H[u]}\prod_{x\in\Tb^3}\,\mathrm{d}u(x),  \end{equation} where $H[u]$ is the renormalization of the Hamiltonian
\[\int_{\Tb^3}|\nabla u|^2+\frac{1}{2}|u|^2(V*|u|^2)\,\mathrm{d}x.\] Rigorously making sense of (\ref{naivegibbs}) is closely linked to the construction of the $\Phi_3^4$ measure in quantum field theory, which has attracted a lot of interest since the 70-80's \cite{Nel, Simon, Aiz, Fro, GJ, LRS} and in recent years \cite{BG,BG2, FKSS, Soh}. In the case of (\ref{plainHartree}), the answer actually depends on the value of $\beta$. When\footnote{We will not study the focusing case $V\leq 0$, where the measure can be constructed only when $\beta>2$; see \cite{OOT}.} $\beta>1/2$, the measure $\mathrm{d}\nu$ can be defined as a weighted version of the Gaussian measure $\mathrm{d}\rho$, namely
\begin{equation}\label{weightgibbs}\mathrm{d}\nu \, =\,  e^{-\int_{\Tb^3} \, \frac{1}{2}\, :\mathrel{|u|^2(V*|u|^2)}: \, dx }\, \cdot \, \mathrm{d}\rho,\quad \mathrm{d}\rho\sim e^{-\frac{1}{2}\int_{\Tb^3}|\nabla u|^2}\prod_{x\in\Tb^3}\,\mathrm{d}x,\end{equation} where $:\mathrel{|u|^2(V*|u|^2)}:$ is a suitable renormalization of the nonlinearity  (see (\ref{renoquartic}) for a precise definition), and the Gaussian free field $\mathrm{d}\rho$ is defined as the law of distribution for the random variable\footnote{Actually the law of (\ref{randomdata}) requires another factor, which is $e^{-\|u\|_{L^2}^2}$, in (\ref{naivegibbs}) and (\ref{weightgibbs}), which does not make a big difference because the $L^2$ norm is also conserved under (\ref{plainHartree}).} 
\begin{equation}\label{randomdata}f(\omega)=\sum_{k\in\Zb^3}\frac{g_k(\omega)}{\<k\>}e^{ik\cdot x},\end{equation} with $\{g_k(\omega)\}$ being i.i.d. normalized centered complex Gaussians. On the other hand, if $0<\beta\leq 1/2$, then $\mathrm{d}\nu$ is a weighted version of a \emph{shifted} Gaussian measure, which is singular with respect to $\mathrm{d}\rho_1$. These results were proved recently by Bringmann \cite{Bringmann1} and Oh-Okamoto-Tolomeo \cite{OOT} by adapting the variational methods of Barashkov-Gubinelli \cite{BG}.

We remark that, in either case above, it can be shown that the Gibbs measure $\mathrm{d}\nu$ is supported in $H^{-1/2-}(\Tb^3)$, the same space as $\mathrm{d}\rho_1$. In particular the typical element in the support of $\mathrm{d}\nu$ has infinite mass, which naturally leads to the \emph{renormalizations} in the construction of $\mathrm{d}\nu$ alluded above, see Section \ref{setup} below. From the physical point of view it is also worth mentioning that, in the same way (\ref{plainHartree}) is derived from quantum many-body systems, the Gibbs measure $\mathrm{d}\nu$, with the correct renormalizations, can also be obtained by taking the limit of thermal states of such systems, at least when $V$ is sufficiently regular (see \cite{FKSS, Soh}).
\subsubsection{Gibbs measure dynamics and invariance} Of same importance as the construction of the Gibbs measure is the study of its dynamics and rigorous justification of its invariance under the flow of (\ref{plainHartree}). The question of proving invariance of Gibbs measures for infinite dimensional Hamiltonian systems, with interest from both mathematical and physical aspects, has been extensively studied over the last few decades. In fact, it is the works \cite{LRS, Bourgain94, Bourgain}---which attempted to answer this question in some special cases---that mark the very beginning of the subject of random data PDEs.

The literature is now extensive, so we will only review those related to NLS equations. After the construction of Gibbs measures in \cite{LRS}, the first invariance result was due to Bourgain \cite{Bourgain94}, which applies in one dimension for focusing sub-quintic equations, and for defocusing equations with any power nonlinearity. Bourgain \cite{Bourgain} then extended the defocusing result to two dimensions, but only for the cubic equation; the two-dimensional case with arbitrary (odd) power nonlinearity was recently solved by the authors \cite{DNY1}. For the case of Hartree nonlinearity (\ref{plainHartree}) in three dimensions, Bourgain \cite{Bourgain97} obtained invariance for $\beta>2$. We also mention the works of Tzvetkov \cite{T06, T08} and of Bourgain-Bulut \cite{BB,BB2} which concern the NLS equation inside a disc or ball, the construction of non-unique weak solutions by Oh-Thomann \cite{OT2} following the scheme in \cite{AC, DaDe, BTT}, and the relevant works on wave equations \cite{BT2,OT1,Bringmann1,Bringmann2,OOT}. In particular the recent work of Bringmann \cite{Bringmann2} establishes Gibbs measure invariance for the \emph{wave} equation with the Hartree nonlinearity (\ref{plainHartree}) for \emph{arbitrary} $\beta>0$.  

The main mathematical challenge in proving invariance of Gibbs measure is the low regularity of the support of the measure, especially in two or more dimensions. For example, for the two dimensional NLS equation with power nonlinearity, the support of the Gibbs measure $\mathrm{d}\nu$ lies in the space of distributions $H^{0-}(\Tb^2)$, while the scaling critical space is $H^{1/2}(\Tb^2)$ for the quintic equation, and approaches $H^1(\Tb^2)$ for equations with high power nonlinearities. This gap is a major reason why the two-dimensional quintic and higher cases have remained open for so many years. In the case of (\ref{plainHartree}) a similar gap is present, namely between the support of $\mathrm{d}\nu$ at $H^{-1/2-}(\Tb^3)$ and the scaling critical space $H^{(1-\beta)/2}(\Tb^3)$ which is higher than $H^0(\Tb^3)$ with $\beta<1$.

On the other hand, it is known since the pioneering work of Bourgain \cite{Bourgain} that with random initial data, one can go below the classical scaling critical threshold and obtain almost-sure well-posedness results. In the recent works \cite{DNY1,DNY2} of the authors, an intuitive probabilistic scaling argument was performed. This leads to the notion of the \emph{probabilistic scaling critical index} $s_{pr}:=-1/(p-1)$ which is much lower than the classical scaling critical index $s_{cr}:=(d/2)-2/(p-1)$ in the case of $p$-th power nonlinearity in $d$ dimensions. In \cite{DNY2} we proved that almost-sure local well-posedness indeed holds in $H^s$ in the full probabilistic subcritical range when $s>s_{pr}$, in any dimensions and for any (odd) power nonlinearity.

For the case of (\ref{plainHartree}), a similar argument as in \cite{DNY1,DNY2} yields that the probabilistic scaling critical index for (\ref{plainHartree}) is $s_{pr}=(-1-\beta)/2$ which is lower than $-1/2$, so it is reasonable to think that almost-sure well-posedness would be true. However the situation here is somewhat different from \cite{DNY1,DNY2} due to the asymmetry of the nonlinearity (\ref{plainHartree}) compared to the power one, which leads to interesting modifications of the methods in these previous works, as we will discuss in Section \ref{idea} below.
\subsubsection{Probabilistic methods}\label{previous} The first idea in proving almost-sure well-posedness was due to Bourgain \cite{Bourgain} and to Da Prato-Debussche \cite{DaDe}, the latter in the setting of parabolic SPDEs, which can be described as a \emph{linear-nonlinear decomposition}. Namely, the solution is decomposed into a linear, random evolution (or noise) term, and a nonlinear term that has strictly higher regularity, thanks to the smoothing effect of randomization. If the linear term has regularity close to scaling criticality, then the nonlinear term can usually be bounded sub-critically, hence a fixed point argument applies. However this idea has its limitations in that the nonlinear term may not be smooth \emph{enough}, and in practice it is usually limited to slightly supercritical cases (relative to deterministic scaling) and does not give optimal results.

In \cite{DNY1}, inspired partly by the \emph{regularity structures theory} of Hairer and the \emph{para-controlled calculus} by Gubinelli-Imkeller-Perkowski in the parabolic SPDE setting, we developed the theory of \emph{random averaging operators}. The main idea is to take the high-low interaction, which is usually the worst contribution in the nonlinear term described above, and express them as a para-product type linear operator---called the random averaging operator---applied to the random initial data. Moreover, this linear operator is \emph{independent} from the initial data it applies to, and has a \emph{randomness structure} which includes the information of the solution at lower scales, see Section \ref{idea}. This structure is then shown to be preserved from low to high frequencies by an \emph{induction on scales} argument, and eventually leads to improved almost-sure well-posedness results.
We refer the reader to \cite{ST} for  an example of a recent application of the method of random averaging operators of \cite{DNY1} to weakly dispersive NLS.

In \cite{DNY2}, the random averaging operators is extended to the more general theory of \emph{random tensors}. In this theory, the linear operators are extended to multilinear operators which are represented by tensors, and whole algebraic and analytic theories are then developed for these random tensors. For NLS equations with odd power nonlinearity, this theory leads to the proof of optimal almost-sure well-posedness results, see \cite{DNY2}. We remark that, while the theory of random tensors is more powerful than random averaging operators, the latter has a simpler structure, is less notation-heavy, and is already sufficient in many situations (especially if one is not very close to probabilistic criticality).

Finally, we would like to mention other probabilistic methods, developed in the recent works of Gubinelli-Koch-Oh \cite{GKO}, Bringmann \cite{Bringmann0,Bringmann2}, and Oh-Okamoto-Tolomeo \cite{OOT}. These methods also go beyond the linear-nonlinear decomposition, and are partly inspired by the parabolic theories. They have important similarities and differences compared to our methods in \cite{DNY1,DNY2}, but they mostly apply for wave equations instead of Schr\"{o}dinger equations, so we will not further elaborate here, but refer the reader to the above papers for further explanation.
\subsection{Setup and the main result}\label{setup} We start by fixing the i.i.d. normalized (complex) Gaussian random variables $\{g_k(\omega)\}_{k\in\Zb^3}$, so that $\Eb g_k=0$ and $\Eb|g_k|^2=1$. Let \begin{equation} \label{data0}f(\omega)=\sum_{ k\in\Zb^3}\frac{g_k(\omega)}{\langle k\rangle}e^{ik\cdot x}, 
 \end{equation} it is easy to see that $f(\omega)\in H^{-1/2-}(\Tb^3)$ almost surely. Let $V:\Tb^3\to\Rb$ is a potential such that $V$ is even, nonnegative, and $V_0=1$, $|V_k|\lesssim\<k\>^{-\beta}$ as described above. Here and below we will use $u_k$ to denote the Fourier coefficients of $u$ and use $\widehat{u}$ to represent time Fourier transform only. In this paper we fix $\beta<1$ and sufficiently close{\footnote{This is a specific value but we do not track it below.}} to $1$. Let $N\in 2^{\Zb_{\geq 0}}\cup\{0\}$ be a dyadic scale, define projections $\Pi_N$ such that $(\Pi_Nu)_k=\mathbf{1}_{\<k\>\leq N}\cdot u_k$, and $\Delta_N=\Pi_N-\Pi_{N/2}$, and define
 \begin{equation}\label{defdata}f_N(\omega)=\Pi_Nf(\omega),\qquad F_N(\omega)=\Delta_Nf(\omega)=f_N(\omega)-f_{N/2}(\omega).\end{equation} we introduce the following truncated and renormalized version of (\ref{plainHartree}), with truncated random initial data, namely:
 \begin{equation}\label{Hartreetrunc1}
\begin{cases}
i\partial_t u_N + \Delta u_N =  \Pi_N[(|u_N|^2*V)\cdot u_N]-\sigma_Nu_N-\Cc_Nu_N \\
u_N(0) = \Pi_N u_{\mathrm{in}}.
\end{cases}
\end{equation} Here in (\ref{Hartreetrunc1}) we fix
\begin{equation}\label{renormalize1}\sigma_N=\Eb\fint_{\Tb^3}|f_N(\omega)|^2\,\mathrm{d}x=\sum_{\<k\>\leq N}\frac{1}{\<k\>^2},\end{equation} and $\Cc_N$ is a Fourier multiplier,
\begin{equation}\label{renormalize2}(\Cc_Nu)_k=(\Cc_N)_k\cdot u_k,\quad (\Cc_N)_k:=\sum_{\<\ell\>\leq N}\frac{V_{k-\ell}}{\<\ell\>^2}.
\end{equation} Note that $u_N$ is supported in $\<k\>\leq N$ for all time. The first counterterm in (\ref{Hartreetrunc1}), namely $-\sigma_Nu_N$, corresponds to the standard Wick ordering, where one fixes $k_1=k_2$ in the expression
\begin{equation}\label{cubicexp}[(|u|^2*V)\cdot u]_k=\sum_{k_1-k_2+k_3=k}V_{k_1-k_2}\cdot u_{k_1}\overline{u_{k_2}}u_{k_3},\end{equation} plugs in $u=f_N(\omega)$, and takes expectations. The second term $-\Cc_Nu_N$ corresponds to fixing $k_2=k_3$, which is present due to the asymmetry of the nonlinearity $(|u|^2*V)\cdot u$. Note that $(\Cc_N)_k$ is uniformly bounded, and thus is unnecessary, if $\beta>1$ (in particular this is the case of Bourgain \cite{Bourgain97}); if $\beta<1$ this becomes a divergent term which needs to be subtracted.

The equation (\ref{Hartreetrunc1}) is a finite dimensional Hamiltonian equation with Hamiltonian
\begin{equation}\label{hamiltonian1}H_N[u] :=\int_{\Tb^3}\big(|\nabla u|^2+\frac{1}{2}|u|^2(V*|u|^2)-\sigma_N|u|^2-\frac{1}{2}\Cc_Nu\cdot\overline{u}+\frac{1}{2}\sigma^2_N-\frac{1}{2}\gamma_N\big),
\end{equation} 
where $\gamma_N= \sum_{\langle k\rangle,\langle \ell \rangle\leq N} \frac{V_{k-\ell}}{\langle k\rangle^2\langle \ell\rangle^2}$.
\begin{rem}
In fact, the Hamiltonian $H_N[u]$ can be also expressed as $\int_{\Tb^3}\big(|\nabla u|^2+ :\mathrel{|u|^2(V*|u|^2)}:\big)$, where the suitable renormalized nonlinearity $:\mathrel{|u|^2(V*|u|^2)}:$ is defined
\begin{equation}
\label{renoquartic}
:\mathrel{|u|^2(V*|u|^2)}: = |u|^2(V*|u|^2)-\sigma_N (V*|u|^2)-\sigma_N|u|^2-\Cc_Nu\cdot\overline{u}+\sigma^2_N-\gamma_N.
\end{equation}
Notice that $\int_{\Tb^3}\sigma_N (V*|u|^2) =\int_{\Tb^3}\sigma_N|u|^2$ since $\widehat{V}(0)=1$.
\end{rem}
We can define the corresponding truncated and renormalized Gibbs measure, namely
\begin{equation}\label{Gibbstrunc}\mathrm{d}\eta_N(u)=\frac{1}{Z_N}e^{-H_N[u]-\|u\|_{L^2}^2}\prod_{\<k\>\leq N}\,\mathrm{d}u_k\mathrm{d}\overline{u_k}
\end{equation} where $Z_N>0$ is a normalization constant making $\mathrm{d}\nu_N$ a probability measure. Clearly $\mathrm{d}\nu_N$ is invariant under the finite dimensional flow (\ref{Hartreetrunc1}). Note that we can also write
\begin{equation}\label{Gibbstrunc2}\mathrm{d}\eta_N(u)=\frac{1}{Z_N^*}e^{-H_N^{\mathrm{pot}}[u]}\,\mathrm{d}\rho_N(u),
\end{equation} where $Z_N^*$ is another positive constant, $\mathrm{d}\rho_N$ is the law of distribution for the linear Gaussian random variable $f_N(\omega):=\Pi_Nf(\omega)$, and $H_N^{\mathrm{pot}}[u]$ represents the potential energy
\begin{equation}\label{hamiltonian2}H_N^{\mathrm{pot}}[u]=\int_{\Tb^3}\big(\frac{1}{2}|u|^2(V*|u|^2)-\sigma_N|u|^2-\frac{1}{2}\Cc_Nu\cdot\overline{u}+\frac{1}{2}\sigma^2_N-\frac{1}{2}\gamma_N\big).
\end{equation}

Now, define $\Pi_N^\perp=1-\Pi_N$, let $\Vc_N$ and $\Vc_N^\perp$ be the ranges of the projections $\Pi_N$ and $\Pi_N^\perp$, and define $\mathrm{d}\rho$ and $\mathrm{d}\rho_N^\perp$ be the laws of distribution for $f(\omega)$ and $\Pi_N^\perp f(\omega)$ respectively. Then we have $\mathrm{d}\rho=\mathrm{d}\rho_N\times\mathrm{d}\rho_N^\perp$; moreover we define \[\mathrm{d}\nu_N=\mathrm{d}\eta_N\times\mathrm{d}\rho_N^\perp=G_N(u)\cdot\mathrm{d}\rho,\qquad G_N(u):=\frac{1}{Z_N^*}e^{-H_N^{\mathrm{pot}}[\Pi_Nu]}.\] We have the following result. Recall that in this paper we are fixing $\beta<1$ close enough to $1$, in particular $\beta>1/2$.
\begin{prop}\label{acGibbs} Suppose $\beta>1/2$, then $G_N(u)$ converges to a limit $G(u)$ in $L^q(\mathrm{d}\rho)$ for all $1\leq q<\infty$, and the sequence of measures $\mathrm{d}\nu_N$ converges to a probability measure $\mathrm{d}\nu$ in the sense of total variations. The measure $\mathrm{d}\nu$ is call the \emph{Gibbs measure} associated with the system (\ref{plainHartree}).
\end{prop}
\begin{proof} This is proved in the recent works of Bringmann \cite{Bringmann1} and Oh-Okamoto-Tolomeo \cite{OOT}. Strictly speaking they are dealing with the case of real-valued $u$ (as they are concerned about the wave equation), but the proof can be readily adapted to the complex-valued case here.
\end{proof}
Now we can state our main theorem\footnote{We remark that Bringmann has an unpublished proof for the same result assuming $\beta>3/2$.}.
\begin{thm}\label{main} Let $\beta<1$ be close enough to $1$. There exists a Borel set $\Sigma\subset H^{-1/2-}(\Tb^3)$ such that $\nu(\Sigma)=1$, and the following holds. For any $u_{\mathrm{in}}\in\Sigma$, let $u_N(t)$ be defined by (\ref{Hartreetrunc1}), then 
\[\lim_{N\to\infty}u_N(t)=u(t)\] exists in $C_t^0H_x^{-1/2-}(\Rb\times\Tb^3)$, and $u(t)\in\Sigma$ for all $t\in\Rb$. This $u(t)$ solves (\ref{plainHartree}) with a suitably renormalized nonlinearity, and defines a mapping $\Phi_t:\Sigma\to\Sigma$ for each $t\in\Rb$. These mappings satisfy the group properties $\Phi_{t+s}=\Phi_t\Phi_s$, and keeps the Gibbs measure $\mathrm{d}\nu$ invariant, namely $\nu(E)=\nu(\Phi_t(E))$ for any $t\in\Rb$ and Borel set $E\subset\Sigma$.
\end{thm}
\begin{rem} As in \cite{DNY1,DNY2}, the sequence $\{u_N\}$ can be replaced by other canonical approximation sequences, for example with the sharp truncations $\Pi_N$ on initial data replaced by smooth truncations, or with the projection $\Pi_N$ on the nonlinearity in (\ref{Hartreetrunc1}) omitted. The limit obtained does not depend on the choice of such sequences, and the proof will essentially be the same.
\end{rem}
\subsubsection{Regarding the range of $\beta$} The range of $\beta$ obtained in Theorem \ref{main} is clearly not optimal. In fact, the equation (\ref{plainHartree}) with Gibbs measure data is probabilistically subcritical as long as $\beta>0$, and one should expect the same result at least when $\beta>1/2$ (so the Gibbs measure is absolutely continuous with the Gaussian free field).

The purpose of this paper, however, is to provide an example where the method of random averaging operators \cite{DNY1} is applied so that one can significantly improve the existing probabilistic results ($\beta$ close but smaller than $1$ versus $\beta>2$ in \cite{Bourgain97}), while keeping the presentation relatively short. In order to treat $\beta>1/2$ one would need to adapt the sophisticated theory of random tensors \cite{DNY2} which will considerably increase the length of this work, so we decide to leave this part to a next paper.

As for the case $0<\beta<1/2$, one would need to deal with the mutual singularity between the Gibbs measure and the Gaussian free field (of course, if one studies the local well-posedness problem with Gaussian initial data as in (\ref{data0}), which is of course different from Gibbs, then a modification of the random tensor theory \cite{DNY2} would also likely work for all $\beta>0$). The recent work of Bringmann \cite{Bringmann2} provides a nice example where this issue is solved in the context of wave equations, and it would be interesting to see whether this can be extended to Schr\"{o}dinger equations. Finally, the case $\beta=0$, which is the famous Gibbs measure invariance problem for the three-dimensional cubic NLS equation, still remains an outstanding open problem as of now. It is probabilistically \emph{critical}, which presumably would require completely new techniques to solve.
\subsection{Main ideas}\label{idea} Due to the absolute continuity of the Gibbs measure in Proposition \ref{acGibbs}, in order to prove Theorem \ref{main}, we only need to consider initial data distributed according to $\mathrm{d}\rho$ for (the renormalized version of) (\ref{plainHartree}), and the initial data distributed according to $\mathrm{d}\rho_N$ for (\ref{Hartreetrunc1}). In other words, we may assume $u(0)=f(\omega)$ for (\ref{plainHartree}), and $u_N(0)=f_N(\omega)$ for (\ref{Hartreetrunc1}).
\subsubsection{Random averaging operators}\label{rao0} Let us focus on (\ref{Hartreetrunc1}); for simplicity we will ignore the renormalization terms. The approach of Bourgain and of Da Prato-Debussche corresponds to decomposing
\[u_N(t)=e^{it\Delta}f_N(\omega)+v(t),\] where $f_N$ is as in (\ref{defdata}), and $v(t)$ is the nonlinear evolution. In particular this $v(t)$ contains a trilinear Gaussian term
\[v^*(t)=\int_0^t e^{i(t-t')\Delta}\Pi_N[(|e^{it'\Delta}f_N(\omega)|^2*V)e^{it'\Delta}f_N(\omega)]\,\mathrm{d}t'.\] This term turns out to only have $H^{0-}$ regularity, which is not regular enough for a fixed point argument (note that the classical scaling critical threshold is $H^{(1-\beta)/2}$). Therefore this approach does not work.

Nevertheless, one may observe that the only contribution to $v^*$ that has worst ($H^{0-}$) regularity is when the first two input factors are at \emph{low} frequency and the third factor is at \emph{high} frequency, such as
\[\int_0^t e^{i(t-t')\Delta}\Pi_N[(|e^{it'\Delta}f_{N'}(\omega)|^2*V)e^{it'\Delta}F_N(\omega)]\,\mathrm{d}t'\] for $N'\ll N$ and $F_N$ as in (\ref{defdata}). Moreover this low frequency component $f_{N'}$ may also be replaced by the corresponding nonlinear term at frequency $N'$, so it makes sense to separate the low-low-high interaction term $\psi^N$ defined by
\begin{equation}
\left\{
\begin{aligned}(i\partial_t+\Delta)\psi^N&=\Pi_N[(|u_{N/2}|^2*V)\psi^N],\\
\psi^N(0)&=F_N(\omega)
\end{aligned}
\right.
\end{equation} as the singular part of $y_N:=u_N-u_{N/2}$, so that $y_N-\psi^N$ has higher regularity.

The idea of considering high-low interactions is consistent with the para-controlled calculus in \cite{GIP,Ha,GKO}. However in those works the singular term $\psi^N$ and the regular term $y_N-\psi^N$ are characterized only by their regularity (for example one is constructed via fixed point argument in $H^{0-}$ and the other in $H^{1/2-}$), which, as pointed out in \cite{DNY1}, is not enough in the context of Schr\"{o}dinger equations. Instead, it is crucial that one studies the operator, referred to as the \emph{random averaging operator} in \cite{DNY1}, which maps $z$ to the solution to the equation
\begin{equation}
\left\{
\begin{aligned}(i\partial_t+\Delta)\psi&=\Pi_N[(|u_{N/2}|^2*V)\psi],\\
\psi(0)&=z.
\end{aligned}
\right.
\end{equation} Note that the kernel of this operator, which we denote by $H^N=(H^N)_{kk'}(t)$, is a Borel function of $\{g_{k}(\omega)\}_{\<k\>\leq N/2}$ and is \emph{independent} from $F_N(\omega)$. Moreover, this $H^N$ encodes the whole \emph{randomness structure} of $u_{N/2}$, which is captured in two particular matrix norm bounds for $H^N$. Essentially, they involve the $\ell_k^2\to\ell_{k'}^2$ operator norm and the $\ell_{kk'}^2$ Hilbert-Schmidt norm for fixed time $t$ (or fixed Fourier variable $\lambda$), see Section \ref{defnormsec} for details.

This is the main idea of the random averaging operators in \cite{DNY1}. Basically, it allows one to fully exploit the randomness structure of the solution at all scales, which is necessary for the proof in the setting of Schr\"{o}dinger equations in the lack of any smoothing effect.
\subsubsection{The special term $\rho^N$: a `critical' component} In addition to the ansatz introduced in Section \ref{rao0}, it turns out that an extra term is necessary due to the structure (especially the asymmetry) of the nonlinearity (\ref{plainHartree}). Recall that $(|u|^2*V)u$ can be expressed as in (\ref{cubicexp}); for simplicity we will ignore any resonances (which are cancelled by the renormalizations), i.e. assume $k_2\not\in\{k_1,k_3\}$ in (\ref{cubicexp}). Here, if $|k_1-k_2|\gtrsim N^\varepsilon$ for some small constant $\varepsilon$, then the potential $V_{k_1-k_2}$, which is bounded by $\<k_1-k_2\>^{-\beta}$, will transform into a derivative gain, which allows one to close easily using the random averaging operator ansatz in Section \ref{rao0}.

However, suppose $|k_1-k_2|$ is very small, say $|k_1-k_2|\sim 1$ in (\ref{cubicexp}), then the potential does not lead to any gain of derivatives, and we will see that this particular term in fact exhibits some (probabilistically) ``critical'' feature. To see this, let us define $\Nc$ to be this portion of nonlinearity (and the corresponding multilinear expression),
\begin{equation}\label{non1} \Nc(u,v,w)=\Pi_N[(\Pi_1(u\overline{v})*V)\cdot w],\end{equation} note the $\Pi_1$ projector restricting to $|k_1-k_2|\sim1$. Then, if we define the iteration terms
\[u^{(0)}(t)=e^{it\Delta}F_N(\omega);\quad u^{(m)}(t)=\sum_{m_1+m_2+m_3=m-1}\int_0^t e^{i(t-t')\Delta}\Nc(u^{(m_1)},u^{(m_2)},u^{(m_3)})(t')\,\mathrm{d}t',\] it follows from simple calculations that $u^{(0)}$ has regularity $H^{-1/2-}$, while each $u^{(m)}$, where $m\geq 1$ has \emph{exactly} regularity $H^{1/2-}$. Therefore, although $u^{(1)}$ is indeed more regular than $u^{(0)}$, the higher order iterations are \emph{not} getting smoother despite all input functions (which are $F_N(\omega)$) having the \emph{same} (and high) frequency. This is in contrast with the ``genuinely (probabilistically) subcritical'' situations (for the standard NLS) in \cite{DNY2}, where for fixed positive constants $\varepsilon$ and $c$,  the $m$-th iteration $u^{(m)}$, assuming all input frequencies are the same, will have increasing and positive regularity in $H^{\varepsilon m-c}$ as $m$ grows and becomes large.
Similarly, one may consider the linear operator
\[z\mapsto\int_0^t e^{i(t-t')\Delta}\Nc(z(t'),e^{it'\Delta}F_N(\omega),e^{it'\Delta}F_N(\omega))\,\mathrm{d}t',\] with $\Nc$ as in (\ref{non1}) and in typical subcritical cases the norm of this operator from a suitable $X^{s,b}$ space to itself would be $N^{-\alpha}$ for some $\alpha>0$, see \cite{DNY1,DNY2}. However here (for Hartree) one can check that the corresponding norm is in fact $\sim1$, and may even exhibit a logarithmic divergence if one adds up different scales.

Therefore, it is clear that the contribution $\Nc$ as in (\ref{non1}) needs a special treatment in addition to the ansatz in Section \ref{rao0}. Fortunately, this term does not depend on the value of $\beta$ and was already treated in Bourgain's work \cite{Bourgain97}. In this work, we introduce an extra term $\rho^N$, which corresponds to the term treated in Bourgain \cite{Bourgain97}, by defining $\xi^N$ such that
\begin{equation}
\left\{
\begin{aligned}(i\partial_t+\Delta)\xi^N&=\Pi_N[(|u_{N/2}|^2*V)\xi^N+\widetilde{\Pi}_{N^{\varepsilon}}((|u_N|^2-|u_{N/2}|^2)*V)]\xi^N,\\
\xi^N(0)&=F_N(\omega)
\end{aligned}
\right.
\end{equation} and defining $\rho^N=\xi^N-\psi^N$, where $\widetilde{\Pi}_{N^{\varepsilon}}$ is a smooth truncation at frequency $N^\varepsilon$ for some small $\varepsilon$. This term is then measured at regularity $H^s$ for some $s<1/2$, while the remainder term $z_N:=y_N-\xi^N$, where $y_N=u_N-u_{N/2}$, is measured at regularity $H^{s'}$ for some $s<s'<1/2$. See Section \ref{ansatz1} for the solution ansatz and Proposition \ref{mainprop} for the precise formulations.
\subsubsection{Additional remark} Note that the precise definitions of the equations satisfied by $\psi^N$ and $\xi^N$, see (\ref{lineareqn}) and (\ref{defxin}), involve projection $\Delta_N$ on the right hand sides; this is to make sure that $(\psi^N)_k$ and $(\xi^N)_k$ are exactly supported in $N/2<\<k\>\leq N$, so that one can exploit the cancellation due to the unitarity of the matrices $H^N$ (corresponding to $\psi^N$), as well as the matrices $M^N$ which corresponds to the term $\xi^N$. This unitarity comes from the mass conservation property of the linear equations defining these matrices, and already plays a key role in Bourgain's work \cite{Bourgain97}. See Section \ref{unit} for details.
\section{Preparations}
\subsection{Reduction of the equation}
We start with the system (\ref{Hartreetrunc1}) with initial data $u_N(0)=f_N(\omega)$. Clearly $(u_N)_k$ is supported in $\<k\>\leq N$. If we denote the right hand side of (\ref{Hartreetrunc1}) by $\Pi_N\Nc(u_N)$, then in Fourier space we have
\begin{equation}\label{nonlinearity1}\Nc(u)_k=\Nc^\circ(u)_k+u_k\cdot\bigg(\sum_{\ell}|u_\ell|^2-\alpha_N\bigg)+u_k\cdot\sum_{\ell\neq k,\langle \ell\rangle\leq N} V_{k-\ell}\bigg(|u_\ell|^2-\frac{1}{\<\ell\>^2}\bigg)-\frac{u_k}{\<k\>^2},
\end{equation} 
\begin{equation}\label{nonlinearity2}\Nc^\circ(u)_k=\sum_{\substack{k_1-k_2+k_3=k\\k_2\not\in\{k_1,k_3\}}}V_{k_1-k_2}\cdot u_{k_1}\overline{u}_{k_2}u_{k_3}.
\end{equation} We will extend $\Nc^\circ(u)$, which is a cubic polynomial of $u$, to an $\Rb$-trilinear operator $\Nc^\circ(u,v,w)$ in the standard way. Note that\[\sum_\ell|(u_N)_\ell|^2=\fint_{\Tb^3}|u_N|^2\,\mathrm{d}x\] is conserved under the flow (\ref{Hartreetrunc1}), we may get rid of the second term on the right hand side of (\ref{nonlinearity1}) by a gauge transform
\[u_N\to e^{iB_Nt}u_N,\quad B_N:=\sum_{\<\ell\>\leq N}\frac{|g_\ell|^2-1}{\<\ell\>^2}.\] If we further define the profile $v_N$ by 
\[(v_N)_k(t)=e^{-it|k|^2}e^{iB_Nt}(u_N)_k(t),\] then $v$ will satisfy the integral equation
\begin{multline}\label{reducedeqn}(v_N)_k(t)=(f_N)_k-i\int_0^t \Pi_N\Mc^\circ(v_N,v_N,v_N)_k(s)\,\mathrm{d}s\\-i\int_0^t(v_N)_k(s)\sum_{\ell\neq k,\<\ell\>\leq N}V_{k-\ell}\bigg(|(v_N)_\ell(s)|^2-\frac{1}{\<\ell\>^2}\bigg)\,\mathrm{d}s+i\int_0^t\frac{(v_N)_k(s)}{\<k\>^2}\,\mathrm{d}s
\end{multline} where
\begin{equation}\label{reducedeqn2}\Mc^\circ(u,v,w)_k(s)=\sum_{\substack{k_1-k_2+k_3=k\\k_2\not\in\{k_1,k_3\}}}e^{is\Omega}\cdot V_{k_1-k_2}\cdot u_{k_1}(s)\overline{v_{k_2}}(s)w_{k_3}(s),\quad \Omega:=|k_1|^2-|k_2|^2+|k_3|^2-|k|^2.
\end{equation} Below we will focus on the system (\ref{reducedeqn})--(\ref{reducedeqn2}).
\subsection{Notations and norms} We setup some basic notations and norms needed later in the proof.
\subsubsection{Notations} As denoted above, we will use $v_k$ to denote Fourier coefficients, and $\Fc v_k=\widehat{v}_k=\widehat{v}_k(\lambda)$ denotes the Fourier transform in time. For a finite index set $A$, we will write $k_A=(k_j:j\in A)$ where each $k_j\in\Zb^3$ and denote by $h_{k_A}$ a tensor $h:(\Zb^3)^A\to\Cb$. We may also define tensors involving $\lambda$ variables where $\lambda\in\Rb$.

We fix the parameters, to be used in the proof, as follows. Let $\varepsilon>0$ be sufficiently small absolute constant. Let $\varepsilon_1$ and $\varepsilon_2$ be fixed such that $\varepsilon_2\ll\varepsilon_1\ll\varepsilon$. Let $\beta<1$ be such that $1-\beta\ll\varepsilon_2$, and choose $\delta$ such that $\delta\ll1-\beta$, and $\kappa$ such that $\kappa\gg \delta^{-1}$. We use $\theta$ to denote any generic small positive constant such that $\theta\ll\delta$ (which may be different at different places). Let $b=1/2+\kappa^{-1}$, so $1-b=1/2-\kappa^{-1}$. Finally, let $\tau$ be sufficiently small compared to all the above parameters, denote $J=[-\tau,\tau]$. Fix a smooth cutoff function $\chi(t)$ which equals 1 for $|t|\leq 1$ and equals 0 for $|t|\geq 2$, and define $\chi_\tau(t):=\chi(\tau^{-1}t)$. We use $C$ to denote any large absolute constant, and $C_\theta$ for any large constant depending on $\theta$. If some event happens with probability $\geq 1-C_\theta e^{-A^\theta}$, where $A$ is a large parameter, we say this event happens \emph{$A$-certainly}.
\subsubsection{Norms}\label{defnormsec}If $(B,C)$ is a partition of $A$, namely $B\cap C=\varnothing$ and $B\cup C=A$, we define the norm $\|h\|_{k_B\to k_C}$ such that
\begin{equation}\label{defnorm0}\|h\|_{k_B\to k_C}^2=\sup\bigg\{\sum_{k_C}\bigg|\sum_{k_B}h_{k_A}z_{k_B}\bigg|^2:\sum_{k_B}|z_{k_B}|^2=1\bigg\}.\end{equation} The same notation also applies for tensors involving the $\lambda$ variables. For functions $u=u_k(t)$ and $h=h_{kk'}(t)$, and $0<c<1$, we also define the norms
\begin{equation}\label{defnorm}
\begin{aligned}
\|u\|_{X^c}^2&:=\int_\Rb\<\lambda\>^{2c}\|\widehat{u}_k(\lambda)\|_{k}^2\,\mathrm{d}\lambda,\\
\|h\|_{Y^c}^2&:=\int_\Rb\<\lambda\>^{2c}\|\widehat{h}_{kk'}(\lambda)\|_{k\to k'}^2\,\mathrm{d}\lambda,\\
\|h\|_{Z^c}^2&:=\int_\Rb\<\lambda\>^{2c}\|\widehat{h}_{kk'}(\lambda)\|_{kk'}^2\,\mathrm{d}\lambda.
\end{aligned}
\end{equation} For any interval $I$, define the corresponding localized norms
\begin{equation}\label{deflocalnorm}
\|u\|_{X^{c}(I)}:=\inf\big\{\|v\|_{X^c}:v=u\mathrm{\ on\ }I\big\}\end{equation} and similarly define $Y^c(I)$ and $Z^c(I)$. By abusing notations, we will call the above $v$ an \emph{extension} of $u$, though it's actually an extension of the restriction of $u$ to $I$.
\subsection{Preliminary estimates} Here we record some basic estimates. Most of them are standard, or are in our previous works \cite{DNY1,DNY2}.
\subsubsection{Linear estimates} Define the original and truncated Duhamel operators
\begin{equation}\label{defi}Iv(t)=\int_0^tv(t')\,\mathrm{d}t',\quad \Ic_\chi v(t)=\chi(t)\int_0^t\chi(t')v(t')\,\mathrm{d}t'.
\end{equation} 
\begin{lem}\label{duhamelform0} We have the formula
\begin{equation}\label{duhamelform}\widehat{\Ic_\chi v}(\lambda)=\int_{\Rb}I(\lambda,\lambda')\widehat{v}(\lambda')\,\mathrm{d}\lambda',
\end{equation} where the kernel $I$ satisfies that
\begin{equation}\label{duhamleker}|I|+|\partial_{\lambda,\lambda'}I|\lesssim\bigg(\frac{1}{\langle \lambda\rangle^3}+\frac{1}{\langle\lambda-\lambda'\rangle^3}\bigg)\frac{1}{\langle\lambda'\rangle}\lesssim\frac{1}{\langle\lambda\rangle\langle\lambda-\lambda'\rangle}.
\end{equation}
\end{lem}
\begin{proof} See \cite{DNY0}, Lemma 3.1 whence by a similar proof, one can also prove (\ref{duhamleker}) for $|\partial_{\lambda,\lambda'}I|$.
\end{proof}
\begin{prop}[Short time bounds]\label{shorttime} Let $\varphi$ be any Schwartz function, recall that $\varphi_\tau(t)=\varphi(\tau^{-1}t)$ for $\tau\ll1$. Then for any $u=u_k(t)$ we have
\begin{equation}\label{sttime1}\|\varphi_{\tau}\cdot u\|_{X^{c}}\lesssim {\tau}^{c_1-c}\|u\|_{X^{c_1}}
\end{equation} provided either $0<c\leq c_1<1/2$, or $u_k(0)=0$ and $1/2<c\leq c_1<1$. The same result also holds if $u=u(t)$ is measured in norms other than $\ell^2$, so (\ref{sttime1}) is true with $X$ replaced by $Y$ or $Z$.
\end{prop}
\begin{proof} See \cite{DNY2}, Lemma 4.2.
\end{proof}
\begin{lem}[Suitable extensions]\label{extension}
Suppose $f(x,t)$ is a function defined in $t\in [-\tau, \tau]=J$ with $|\tau|\ll 1$. Define that 
\begin{equation}\label{lem2.3:g}
g(t) =\begin{cases}
f(t) & \text{if } |t|\leq \tau\\
f(\tau) & \text{if } t> \tau\\
f(-\tau) & \text{if } t<-\tau.
\end{cases}
\end{equation} 
For any Schwartz function $\varphi$,  we have 
\begin{equation}\label{lem2.3:bound1}
\|\varphi(t)\cdot g(t)\|_{X^{b}}\lesssim \|f\|_{X^{b_1}(J)} + \|f\|_{L_t^{\infty}L^2_x(J)},
\end{equation}
provided either $0<b< b_1<1/2$ or $1/2<b< b_1<1$.
When $1/2<b< b_1<1$, we have 
\begin{equation}\label{lem2.3:bound2}
\|\varphi(t)\cdot g(t)\|_{X^{b}}\lesssim \|f\|_{X^{b_1}(J)}.
\end{equation}
\end{lem}
\begin{proof} We only need to bound locally-in-time the function $f^*(t)$, which equals $f(0)$ for $t\geq 0$ and $f(t)$ for $t<0$; in fact $g$ is obtained by performing twice the transformation from $f$ to $f^*$, first at center $\tau$ and then at center $-\tau$.

We can decompose $f$ into two parts, $f_1$ which is smooth and equals $f(0)$ near $0$, and $f_2$ such that $f_2(0)=0$. Clearly we only need to consider $f_2$, so that $f^*$ equals $f_2$ multiplied by a smooth truncation of $\mathbf{1}_{[0,+\infty)}$, with $f_2(0)=0$.

We may replace $\mathbf{1}_{[0,+\infty)}$ by the sign function, and then apply Proposition \ref{shorttime}; note that for an even smooth cutoff function $\chi$,
\[\chi(x)\cdot\mathrm{sgn}(x)=\sum_{N\geq 1}\Delta_N(\chi\cdot\mathrm{sgn})(x)\] where $\Delta_N$ are the standard Littlewood-Paley projections. Moreover $\Delta_N(\chi\cdot\mathrm{sgn})(x)$ can be viewed as a rescaled Schwartz function of the same form as in Proposition \ref{shorttime} with $\tau= N^{-1}$ (due to the expression of the Fourier transform of $\mathrm{sgn}$ and simple calculations), so the desired result follows from Proposition \ref{shorttime}.
\end{proof}

\subsubsection{Counting estimates} Here we list some counting estimates and the resulting tensor norm bounds.
\begin{lem}\label{counting}
(1) Let $\mathcal{R}=\mathbb{Z}$ or $\mathbb{Z}[i]$. Then, given $0\neq m\in\mathcal{R}$, and $a_0,b_0\in\mathbb{C}$, the number of choices for $(a,b)\in\mathcal{R}^2$ that satisfy
\begin{equation}m=ab,\,\,|a-a_0|\leq M,\,\,|b-b_0|\leq N
\end{equation}is $O(M^\theta N^\theta)$ with constant depending only on $\theta>0$.

(2) For dyadic numbers $N_1, N_2, N_3, R>0$ and some fixed number $\Omega_0$. \begin{equation}\label{counting:set}
S^{R} =\left\{
\begin{array}{lr}
(k,k_1, k_2, k_3)\in (\mathbb{Z}^3)^4,\quad k_2\notin \{k_1, k_3\}\\
k=k_1-k_2+k_3, \quad |k|\leq N\\
 |k|^2-|k_1|^2+|k_2|^2-|k_3|^2=\Omega_0\\
\frac{N_j}{2}<|k_j|\leq N_j\, (j\in \{1,2,3\}),\quad \frac{R}{2}<\<k_1-k_2\>\leq R 
\end{array} 
\right\},
\end{equation}
and then $S_k^{R}$ is the set of $(k, k_1,k_2,k_3)\in S^{R}$ when $k$ is fixed and etc. We have the following counting estimates
\begin{align}
\label{counting1}\big|S^{R}\big|&\lesssim \min (N_1^3 N_3^3 (N_2\wedge N)^{1+\theta}, N^3N_2^3 (N_1\wedge  N_3)^{1+\theta}, N_2^3 (RN_3)^{2+\theta}, N^3 (RN_1)^{2+\theta});\\
\label{counting2}
\big|S^{R}_k\big|&\lesssim  \min \big(N_2^3 (N_1\wedge N_3)^{1+\theta}, (N_1N_3)^{2+\theta}, (RN_1)^{2+\theta}\big);\\\label{counting3}
\big|S^{R}_{k_3}\big|&\lesssim \min \big( N_1^3 (N_2\wedge N)^{1+\theta}, (N_2 N)^{2+\theta}, (RN_2)^{2+\theta}\big);\\\label{counting4}
\big|S^{R}_{k_2}\big|&\lesssim  \min \big(N^3 (N_1\wedge N_3)^{1+\theta}, (N_1N_3)^{2+\theta}, (RN_3)^{2+\theta}\big);\\\label{counting5}
\big|S^{R}_{kk_1}\big|&\lesssim  \min \big(N_2, N_3, R\big)^{2+\theta};\quad \big|S^{R}_{k_2k_3}\big|\lesssim  \min \big(N, N_1, R\big)^{2+\theta};\\\label{counting6}
\big|S^{R}_{kk_2}\big|&\lesssim  \min \big(N_1, N_3, R\big)^{1+\theta};\quad \big|S^{R}_{k_1k_3}\big|\lesssim   \min \big(N_2, N, R\big)^{1+\theta};\\\label{counting7}
\big|S^{R}_{kk_3}\big|&\lesssim \min \big(N_1, N_2, R\big)^{2+\theta};\quad \big|S^{R}_{k_1k_2}\big|\lesssim  \min \big(N, N_3, R\big)^{2+\theta}.
\end{align}
\end{lem}
\begin{proof} 
(1) It is the same as the part (1) of Lemma 4.3 in \cite{DNY1}.
(2) We consider $|S^{R}|$. First the number of choices of $k_1$ and $k_3$ is $N_1^3 N_3^3$. After fixing the choice of $k_1$ and $k_3$  to count $(k,k_2)$, it is equivalent to count $k_2$ satisfying the restriction $|k_2|^2 +|k_2+c_1|^2 =c_2$ or to count $k$ satisfying the restriction $|k|^2 +|k+c_3|^2 =c_4$ for some fixed numbers $c_1,...,c_4$ and hence we have 
$|S^{R}|\lesssim N_1^3 N_3^3 (N_2\wedge N)^{1+\theta}$. Similarly if we first fix $k$ and $k_2$, we have $|S^{R}|\lesssim  N^3N_2^3 (N_1\wedge  N_3)^{1+\theta}$.
Also if we fix $k_2$ first, then to count $(k, k_1, k_3)$ is equivalent to count $(k_1, k_3)$ with the restriction $(k_2-k_1)\cdot (k_2-k_3) = c$ for some fixed number $c$. By fixing the first two components of $(k_1, k_3)$ and using part (1), we have $|S^{R,M}|\lesssim  N_2^3 (RN_3)^{2+\theta}$. Similarly we also have $|S^{R}|\lesssim   N^3 (RN_1)^{2+\theta})$.
The proofs of (\ref{counting2})--(\ref{counting7}) are similar.

\end{proof}
\subsubsection{Probabilistic and tensor estimates}
\begin{prop}[Proposition 4.11 in \cite{DNY2}]\label{merge0} Consider two tensors $h_{k_{A_1}}^{(1)}$ and $h_{k_{A_2}}^{(2)}$, where $A_1\cap A_2=C$. Let $A_1\Delta A_2=A$, define the semi-product
 \begin{equation}\label{combination}H_{k_A}=\sum_{k_C}h_{k_{A_1}}^{(1)}h_{k_{A_2}}^{(2)}.
 \end{equation} Then, for any partition $(X,Y)$ of $A$, let $X\cap A_1=X_1$, $Y\cap A_1 =Y_1$ etc., we have
 \begin{equation}\label{combinationbd}\|H\|_{k_X\to k_Y}\leq \|h^{(1)}\|_{k_{X_1\cup C}\to k_{Y_1}}\cdot\|h^{(2)}\|_{k_{X_2}\to k_{C\cup Y_2}}.
 \end{equation}
 \end{prop}
 
  \begin{prop}[Proposition 4.12 in \cite{DNY2}]\label{merge} Let $A_j\,(1\leq j\leq m)$ be index sets, such that any index appears in at most two $A_j$'s, and let $h^{(j)}=h_{k_{A_j}}^{(j)}$ be tensors. Let $A=A_1\Delta\cdots\Delta A_m$ be the set of indices that belong to only one $A_j$, and $C=(A_1\cup\cdots \cup A_m)\backslash A$ be the set of indices that belong to two different $A_j$'s. Define the semi-product
 \begin{equation}\label{combination2}H_{k_A}=\sum_{k_C}\prod_{j=1}^mh_{k_{A_j}}^{(j)}.
 \end{equation} Let $(X,Y)$ be a partition of $A$. For $1\leq j\leq m$ let $X_j=X\cap A_j$ and $Y_j=Y\cap A_j$, and define
 \begin{equation}\label{intermedsets}B_j:=\bigcup_{\ell>j}(A_j\cap A_\ell),\quad C_j=\bigcup_{\ell<j}(A_j\cap A_\ell),
 \end{equation} then we have
 \begin{equation}\label{combinationbd2}\|H\|_{k_X\to k_Y}\leq\prod_{j=1}^m\|h^{(j)}\|_{k_{X_j\cup B_j}\to k_{Y_j\cup C_j}}.
 \end{equation} 
 \end{prop} For the proofs of Propositions \ref{merge0} and \ref{merge}, see \cite{DNY2}. In that work the full power of (\ref{combinationbd}) and (\ref{combinationbd2}) is needed, but here we only need some specific cases, mainly those of the following form (where $q\leq r$)
 \begin{equation}\label{mainmerge}\bigg\|\sum_{k_1,\cdots,k_q}H_{k_1\cdots k_r}h_{k_1k_1'}^{(1)}\cdots h_{k_qk_q'}^{(q)}\bigg\|_{k_{A'}\to k_{B'}}\leq \|H\|_{k_{A}\to k_{B}}\prod_{j=1}^q\|h^{(j)}\|_{k_j\to k_j'}\end{equation} where $(k_{A'},k_{B'})$ is a partition of the variables $(k_1',\cdots,k_q',k_{q+1},\cdots k_r)$ and $(k_{A},k_{B})$ is a partition of the variables $(k_1,\cdots,k_r)$ where each $k_j'\,(1\leq j\leq q)$ is replaced by $k_j$ in $(k_{A'},k_{B'})$.

  \begin{prop}[Proposition 4.14 in \cite{DNY2}]\label{trim} Let $A$ be a finite set and $h_{bck_A}=h_{bck_A}(\omega)$ be a tensor, where each $k_j\in \Zb^d$ and $(b,c)\in (\Zb^3)^q$ for some integer $q\geq 2$. Given signs $\zeta_j\in\{\pm\}$, we also assume that $\langle b\rangle,\langle c\rangle\lesssim M$ and $\langle k_j\rangle\lesssim M$ for all $j\in A$, where $M$ is a dyadic number, and that in the support of $h_{bck_A}$ there is no pairing in $k_A$. Define the tensor
  \begin{equation}\label{contract}
  H_{bc}=\sum_{k_{A}}h_{bck_A}\prod_{j\in A}\eta_{k_j}^{\zeta_j},
  \end{equation}where we restrict $k_j\in E$ in (\ref{contract}), $E$ being a finite set such that $\{h_{bck_A}\}$ is independent with $\{\eta_k:k\in E\}$. Then $\tau^{-1}M$-certainly, we have
  \begin{equation}\label{contbd}\|H_{bc}\|_{b\to c}\lesssim\tau^{-\theta} M^\theta\cdot\max_{(B,C)}\|h\|_{bk_B\to ck_C},
  \end{equation} where $(B,C)$ runs over all partitions of $A$. The same results holds is we do not assume $\langle b\rangle,\langle c\rangle\lesssim M$, but instead that (i) $b,c\in\Zb^3$ and $|b-c|\lesssim M$ and $||b|^2-|c|^2|\lesssim M^{\kappa^{3}}$, and (ii) $h_{bck_A}$ can be written as a function of $b-c$, $|b|^2-|c|^2$ and $k_A$.
 \end{prop} For the proof of Proposition \ref{trim} see \cite{DNY2}, Propositions 4.14 and 4.15.
 \begin{prop}[Weighted bounds]
 \label{weighted} Suppose the matrices $h=h_{kk''}$, $h^{(1)}=h_{kk'}^{(1)}$ and $h^{(2)}=h_{k'k''}^{(2)}$ satisfy that
 \[h_{kk''}=\sum_{k'}h_{kk'}^{(1)}h_{k'k''}^{(2)},\] and $h_{kk'}^{(1)}$ is supported in $|k-k'|\lesssim L$, then we have
 \[\bigg\|\bigg(1+\frac{|k-k''|}{L}\bigg)^\kappa h_{kk''}\bigg\|_{\ell_{kk''}^2}\lesssim \|h^{(1)}\|_{k\to k'}\cdot \bigg\|\bigg(1+\frac{|k'-k''|}{L}\bigg)^\kappa h_{k'k''}^{(2)}\bigg\|_{\ell_{k'k''}^2}.\]
 \end{prop} For the proof of Proposition \ref{weighted} see \cite{DNY1}, Proposition 2.5 or \cite{DNY2}, Lemma 4.3 (there are different versions of this bound, but the proofs are the same).
\section{The ansatz}
\subsection{The structure of $y_N$}
Start with the system (\ref{reducedeqn})--(\ref{reducedeqn2}). Let $y_N=v_N-v_{N/2}$, then $y_N$ satisfies the integral equation
\begin{equation}\label{eqnyn}
\begin{aligned}(y_N)_k(t)&=(F_N)_k-i\sum_{\max(N_1,N_2,N_3)=N}\int_0^t\Pi_N\Mc^\circ(y_{N_1},y_{N_2},y_{N_3})_k(s)\,\mathrm{d}s\\&+i\int_0^t\frac{(y_N)_k(s)}{\<k\>^2}\,\mathrm{d}s-i\sum_{\max(N_1,N_2,N_3)\leq N/2}\int_0^t\Delta_N\Mc^\circ(y_{N_1},y_{N_2},y_{N_3})_k(s)\,\mathrm{d}s\\&-i\int_0^t\bigg[(v_N)_k(s)\sum_{\ell\neq k,\<\ell\>\leq N}V_{k-\ell}\bigg(|(v_N)_\ell(s)|^2-\frac{1}{\<\ell\>^2}\bigg)\\&\hspace{3cm}-(v_{N/2})_k(s)\sum_{\ell\neq k,\<\ell\>\leq N/2}V_{k-\ell}\bigg(|(v_{N/2})_\ell(s)|^2-\frac{1}{\<\ell\>^2}\bigg)\bigg]\,\mathrm{d}s.
\end{aligned}
\end{equation} 
\subsubsection{The term $\psi^{N,L}$} For any $L\leq N/2$, consider the linear equation for $\Psi=\Psi_k(t)$:
\begin{equation}\label{lineareqn}\partial_t\Psi_k(t)=-i\Delta_N\Mc^<(v_{L},v_{L},\Psi)_k(t),
\end{equation} where we define, with $\delta\ll 1$,
\begin{equation}\label{lownonlin}\Mc^<(u,v,w)_k(t):=\sum_{\substack{k_1-k_2+k_3=k\\k_2\not\in\{k_1,k_3\}}}e^{it\Omega}\cdot \eta\bigg(\frac{k_1-k_2}{N^{1-\delta}}\bigg)V_{k_1-k_2}\cdot u_{k_1}(t)\overline{v_{k_2}}(t)w_{k_3}(t);
\end{equation}define also $\Mc^>:=\Mc^\circ-\Mc^<$. If (\ref{lineareqn}) has initial data $\Psi_k(0)=\Delta_N\phi_k$, then the solution may be expressed as \begin{equation}\label{rao}\Psi_k(t)=\sum_{k'}H_{kk'}^{N,L}(t)\phi_{k'}.
\end{equation} where $H^{N,L}=H_{kk'}^{N,L}$ is the kernel of a linear operator (or a matrix). Define also
\begin{equation}\label{raoterm}(\psi^{N,L})_k(t)=\sum_{k'}H_{kk'}^{N,L}(t)(F_N)_{k'},
\end{equation} and similarly
\begin{equation}\label{rao2} h^{N,L}:=H^{N,L}-H^{N,L/2},\quad \zeta^{N,L}:=\psi^{N,L}-\psi^{N,L/2};
\end{equation}note that when $L=1$ we will replace $L/2$ by $0$, so for example $(\psi^{N,0})_k(t)=(F_N)_k$. For simplicity denote 
\begin{equation}\label{simply} H^N:=H^{N,N/2} \qquad \text{and} \qquad \psi^N: =\psi^{N,N/2}. \end{equation} Note that each $h^{N,L}$ and $H^{N,L}$ is a Borel function of $(g_k(\omega))_{\<k\>\leq N/2}$, and is thus independent from the Gaussians in $F_N$.
\subsubsection{The terms $\xi^N$ and $\rho^N$} Next, similar to (\ref{lineareqn}), we consider the linear equation \begin{equation}\label{defxin}\partial_t\Xi_k(t)=-i\Delta_N\big[\Mc^<(v_{N/2},v_{N/2},\xi^N)+\Mc^{\ll}(v_{N},v_N,\xi^N)-\Mc^{\ll}(v_{N/2},v_{N/2},\Xi)\big]_k(t),
\end{equation} where $\Mc^{\ll}$ is defined by
\begin{equation}\label{lownonlin1}\Mc^{\ll}(u,v,w)_k(s):=\sum_{\substack{k_1-k_2+k_3=k\\k_2\not\in\{k_1,k_3\}}}e^{is\Omega}\cdot \eta\bigg(\frac{k_1-k_2}{N^{\varepsilon}}\bigg)V_{k_1-k_2}\cdot u_{k_1}(s)\overline{v_{k_2}}(s)w_{k_3}(s).
\end{equation} If the initial data is $\Xi_k(0)=\Delta_N\phi_k$, then we may write the solution as
\begin{equation}\label{defxin2}\Xi_k(t)=\sum_{k'}M_{kk'}^N(t)\phi_{k'},\end{equation} which defines the matrix $M^N=M_{kk'}^N$. We then define $\xi^N$ and $\rho^N$ by
\begin{equation}\label{defxin3}(\xi^N)_k(t):=\sum_{k'}M_{kk'}^N(t)(F_N)_{k'},\quad\rho^N:=\xi^N-\psi^N.\end{equation}
\subsubsection{The ansatz}\label{ansatz1} Now we introduce the ansatz
\begin{equation}\label{ansatz}(y_N)_k(t) \, = \, (\xi^{N})_k(t)\, +\, (z_N)_k(t).
\end{equation} where $z_N$ is a remainder term. We can calculate that $z_N$ solves the equation (recall $y_N= v_N - v_{N/2}$),
\begin{equation}\label{eqnzn}
\begin{aligned}(z_N)_k(t)=&-i\sum_{\max(N_1,N_2,N_3)=N}\int_0^t\Pi_N\Mc^>(y_{N_1},y_{N_2},y_{N_3})_k(s)\,\mathrm{d}s\\&+i\int_0^t\frac{(y_N)_k(s)}{\<k\>^2}\,\mathrm{d}s-i\sum_{\max(N_1,N_2,N_3)\leq N/2}\int_0^t\Delta_N\Mc^\circ(y_{N_1},y_{N_2},y_{N_3})_k(s)\,\mathrm{d}s\\
&-i\sum_{\max(N_1,N_2)=N;N_3\leq N}\int_0^t\Pi_{N}(\Mc^<-\Mc^{\ll})(y_{N_1},y_{N_2},y_{N_3})_k(s)\,\mathrm{d}s\\
&-i\int_0^t\Pi_{N/2}\Mc^<(v_{N/2},v_{N/2},y_N)_k(s)\,\mathrm{d}s-i\int_0^t\Delta_N\Mc^<(v_{N/2},v_{N/2},z_N)_k(s)\,\mathrm{d}s\\
&-i\sum_{\max(N_1,N_2)=N}\Pi_{N/2}\Mc^{\ll}(y_{N_1},y_{N_2},y_N)_k(s)\,\mathrm{d}s-i\sum_{\max(N_1,N_2)=N}\Delta_N\Mc^{\ll}(y_{N_1},y_{N_2},z_N)_k(s)\,\mathrm{d}s\\&-i\int_0^t\bigg[(v_N)_k(s)\sum_{\ell\neq k,\<\ell\>\leq N}V_{k-\ell}\bigg(|(v_N)_\ell(s)|^2-\frac{1}{\<\ell\>^2}\bigg)\\&\hspace{3cm}-(v_{N/2})_k(s)\sum_{\ell\neq k,\<\ell\>\leq N/2}V_{k-\ell}\bigg(|(v_{N/2})_\ell(s)|^2-\frac{1}{\<\ell\>^2}\bigg)\bigg]\,\mathrm{d}s.
\end{aligned}
\end{equation}
\subsection{Unitarity of matrices $H^{N,L}$ and $M^N$}\label{unit} The following properties of $H$ and $M$ will play a fundamental role. This idea goes back to Bourgain \cite{Bourgain97}. Recall that for $L\leq N/2$ the matrix $H^{N,L}$ is defined by (\ref{lineareqn}) and (\ref{rao}). Note that if $\Psi$ solves (\ref{lineareqn}) then $\Psi_k(t)$ is supported in $N/2<\<k\>\leq N$, and we have
\begin{equation}\label{unitary1}
\begin{aligned}\partial_t\sum_k|\Psi_k(t)|^2&=2\cdot\mathrm{Im}\sum_{k}\overline{\Psi_k(t)}\cdot\sum_{\substack{k_1-k_2+k_3=k\\k_2\not\in\{k_1,k_3\}}}e^{it(|k_1|^2-|k_2|^2+|k_3|^2-|k|^2)}\\
&\times\eta\bigg(\frac{k_1-k_2}{N^{1-\delta}}\bigg)V_{k_1-k_2}\cdot(v_L)_{k_1}(t)\overline{(v_L)_{k_2}(t)}\Psi_{k_3}(t).
\end{aligned}
\end{equation} The sum on the right hand side may be replaced by two terms, namely $S_1$ where we only require $k_1\neq k_2$ in the summation, $S_2$ where we require $k_1\neq k_2$ and $k_2=k_3$ in the summation. For $S_1$ by swapping $(k,k_1,k_2,k_3)\mapsto (k_3,k_2,k_1,k)$ we also see that $S_1\in\Rb$ and hence $\mathrm{Im}(S_1)=0$; moreover
\[S_2=\sum_{k\neq k_2}\eta\bigg(\frac{k-k_2}{N^{1-\delta}}\bigg)V_{k-k_2}\overline{\Psi_k(t)}(v_L)_k(t)\cdot\Psi_{k_2}(t)\overline{(v_L)_{k_2}(t)}\] which is also real valued by swapping $(k,k_2)\mapsto(k_2,k)$. This means that $\sum_k|\Psi_k(t)|^2$ is conserved in time. Therefore for each fixed $t$, the matrix $H^{N,L}=H_{kk'}^{N,L}$ is unitary, hence we get the identity
\begin{equation}\label{unitarity}\sum_{k'}H_{k_1k'}^{N,L}\cdot\overline{H_{k_2k'}^{N,L}}=\delta_{k_1k_2}
\end{equation} with $\delta_{k_1k_2}$ being the Kronecker delta. This in particular holds for $L=N/2$. In the same way, the matrix $M^N$ defined by (\ref{defxin}) and (\ref{defxin2}) also satisfies (\ref{unitarity}).
\subsection{The a priori estimates} We now state the main a priori estimate, and prove that this implies Theorem \ref{main}.
\begin{prop}\label{mainprop} Given $0<\tau\ll 1$, and let $J=[-\tau,\tau]$. Recall the parameters defined in Section. For any $M$, consider the following statements, which we call $\mathtt{Local}(M)$:
\begin{enumerate}
\item For the operators $h^{N,L}$, where $L<M$ and $N>L$ is arbitrary, we have
\begin{equation}\label{apriori1}\|h^{N,L}\|_{Y^{1-b}(J)}+\sup_{t\in J}\|h^{N,L}(t)\|_{\ell^2\to \ell^2}\leq L^{-1/2+3\varepsilon_1},\quad \|h^{N,L}\|_{Z^b(J)}\leq N^{1+\delta}L^{-1/2+2\varepsilon_1},
\end{equation} as well as 
\begin{equation}\label{apriori1.5}\bigg\|\bigg(1+\frac{|k-k'|}{\min(L,N^{1-\delta})}\bigg)^{\kappa}h_{kk'}^{N,L}\bigg\|_{Z^b(J)}\leq N^{3/2}.
\end{equation}
\item For the terms $\rho^N$ and $z_N$, where $N\geq M$, we have
\begin{equation}\label{apriori1.8}\|\rho^N\|_{X^b(J)}\leq N^{-1/2+\varepsilon_1+\varepsilon_2},\quad \|z_N\|_{X^b(J)}\leq N^{-1/2+\varepsilon_1}.
\end{equation}
\item For any $L_1,L_2<M$, the operator defined by \begin{equation}\label{linearmapping}(\Ls z)_k(t)=-i\int_0^t\Delta_N\Mc^<(y_{L_1},y_{L_2},z)_k(t')\,\mathrm{d}t'
\end{equation} has an extension, which we still denote by $\Ls$ for simplicity. The kernel $\Ls_{kk'}(t,t')$ has Fourier transform $\widehat{\Ls}_{kk'}(\lambda,\lambda')$, which satisfies \begin{equation}\label{basicmatrix1}\int_{\Rb^2}\<\lambda\>^{2(1-b)}\<\lambda'\>^{-2b}\|\widehat{\Ls}\|_{k\to k'}^2\,\mathrm{d}\lambda\mathrm{d}\lambda'\leq L^{-1+6\varepsilon_1-2\varepsilon_2}
\end{equation} and
\begin{equation}\label{basicmatrix2}\int_{\Rb^2}\<\lambda\>^{2b}\<\lambda'\>^{-2(1-b)}\|\widehat{\Ls}\|_{kk'}^2\,\mathrm{d}\lambda\mathrm{d}\lambda'\leq N^{2+2\delta}L^{-1+4\varepsilon_1-2\varepsilon_2},
\end{equation} where $L=\max(L_1,L_2)$.
\end{enumerate} Now, with the above definition, we have that \[\Pb(\mathtt{Local}(M/2)\wedge\neg(\mathtt{Local}(M)))\leq C_\theta e^{-(\tau^{-1}M)^\theta}\]
\end{prop}
\begin{proof}[Proof of Theorem \ref{main}] By Proposition \ref{mainprop}, in particular we know that $\tau^{-1}$-certainly, the event $\mathtt{Local}(M)$ happens for any $M$. By (\ref{rao}), (\ref{defxin3}) and (\ref{ansatz}) we have
\[y_N=F_N+\sum_{L\leq N/2}\zeta^{N,L}+\rho^N+z_N.\] Exploiting independence between $h^{N,L}$ and $F_N$ and using Proposition \ref{trim} combined with (\ref{apriori1}), we can show that $\|\zeta^{N,L}\|_{X^b(J)}\lesssim N^\delta L^{-1/3}$. Summing over $L$ and noticing that $\zeta^{N,L}$ is supported in $N/2<\<k\>\leq N$, we see that \[\bigg\|\sum_{L\leq N/2}\zeta^{N,L}\bigg\|_{C_t^0H_x^\gamma(J)}\lesssim N^{-\gamma/2}\] for any $\gamma>0$. Using also (\ref{apriori1.8}) we can see that the sequence $\{v_N-f_N\}$ converges in $C_t^0H_x^{0-}(J)$, hence $\{v_N\}$ converges in $C_t^0H_x^{-1/2-}(J)$, and so does the original sequence $\{u_N\}$.

Therefore, the solution $u_N$ to (\ref{Hartreetrunc1}) converges to a unique limit as $N\to\infty$, up to an exceptional set with probability $\geq 1-C_\theta e^{-\tau^{-\theta}}$. This proves the \emph{almost-sure local well-posedness} of (\ref{plainHartree}) with Gibbs measure initial data. Since the truncated Gibbs measure $\mathrm{d}\eta_N$ defined by (\ref{Gibbstrunc}) is invariant under (\ref{Hartreetrunc1}), and the truncated Gibbs measures converge strongly to the Gibbs measure $\mathrm{d}\nu$ as in Proposition \ref{acGibbs}, we can apply the standard local-to-global argument of Bourgain, where the a priori estimates in Proposition \ref{mainprop} allows us to prove the suitable stability bounds needed in the process, in exactly the same way as in \cite{DNY1}. The almost-sure global existence and invariance of Gibbs measure then follows.
\end{proof}
\subsection{A few remarks and simplifications}\label{rem} From now on we will focus on the proof of Proposition \ref{mainprop}, and assume that the bounds involved in $\mathtt{Local}(M/2)$ are already true. The goal is to recover (\ref{apriori1})--(\ref{apriori1.8}), and (\ref{basicmatrix1})--(\ref{basicmatrix2}) for $M$. Before proceeding, we want to remark on a few simplifications that we would like to make in the proof below. These are either standard, or are the same as in \cite{DNY1,DNY2}, and we will not detail out these arguments in the proof below.

(1) In proving these bounds we will use the standard continuity argument, which involves a smallness factor. Here this factor is provided by the short time $\tau\ll 1$. In particular, we can gain a positive power $\tau^\theta$ by using\footnote{In the case $c>1/2$ we also need $u(0)=0$ in Proposition \ref{shorttime}, but as we will only estimate the Duhamel terms of form $u=I(N)$ or $u=\Ic_\chi N$, see (\ref{defi}), we do indeed have $u(0)=0$.} Proposition \ref{shorttime} at the price of changing the $c$ exponent in the $X^c$ (or $Y^c$ or $Z^c$) norm by a little. It can be checked in the proof below that all the estimates allow for some room in $c$, so this is always possible.

(2) In each proof below, we can actually gain an extra power $M^{\delta/10}$ compared to the desired estimate, so any loss which is $M^{C\kappa^{-1}}$ will be acceptable. In fact, in the proof below we will frequently encounter losses of at most $M^{C\kappa^{-1}}$ due to manipulations of the $c$ exponent in various norms as in (1), and due to application of probabilistic bounds such as Proposition \ref{trim} were we lose a small $\theta$ power.

(3) In the course of the proof, we will occasionally need to obtain bounds of quantities of form $\sup_{\lambda}G(\lambda)$, where $\lambda$ ranges in an interval, and for each \emph{fixed} $\lambda$, the quantity $|G(\lambda)|$ can be bounded, apart from a small exceptional set; moreover, here $G$ will be differentiable and $G'(\lambda)$ will satisfy a weaker but unconditional bound. Then we can apply the \emph{meshing argument} in \cite{DNY1,DNY2}, where we divide the interval into a large number of subintervals, approximate $G$ on each small interval by a sample (or an average), control the error term using $G'$, and add up the exceptional sets corresponding to the sample in each interval. In typical cases, where $M$-certainly $|G(\lambda)|\leq M^\theta$ for each fixed $\lambda$, $|I|\leq M^C$ and $|G'(\lambda)|\leq M^C$ unconditionally, we can deduce that $M$- certainly, $\sup_\lambda|G(\lambda)|\leq M^\theta$, because the number of subintervals is $O(M^C)$ so the total probability for the union of exceptional sets is still sufficiently small.
\section{The random averaging operator} 
\subsection{The operator $\Ls$}\label{operbound}\label{Lbound}We start by proving (\ref{basicmatrix1})--(\ref{basicmatrix2}) for $L=M/2$. We need to construct an extension of $\Ls$ defined in (\ref{linearmapping}). This is done first using Lemma \ref{extension} to find extensions of each component of $y_{L_1}$ and $y_{L_2}$ (note that $\max(L_1,L_2)=M/2$), such that these extension terms satisfy (\ref{apriori1})--(\ref{apriori1.8}) with the localized $X^b(J)$ etc. norms replaced by the global $X^b$ etc. norms, at the expense of some slightly worse exponents. The change of value in exponents will play no role in the proof below so we will omit it. Then, by attaching to $\Ls$ a factor $\chi(\tau^{-1}t)$ and using Lemma \ref{shorttime} (see Section \ref{rem}) we can gain a smallness factor $\tau^\theta$ at the price of further worsening the exponents. These operations are standard so we will not repeat them below. 

Note that the extension defined in Lemma \ref{extension} preserves the independence between the matrices $h^{L_j,R_j}$ and $F_{L_j}$ for $R_j\leq L_j/2$.

Recall that $\widehat{\Ls}_{kk'}(\lambda,\lambda')$ is the Fourier transform of the kernel $\Ls_{kk'}(t,t')$ of $\Ls$, we have \begin{equation}\label{fourierkernel}(\widehat{\Ls z})_k(\lambda)=\sum_{k'}\int_\Rb \widehat{\Ls}_{kk'}(\lambda,\lambda')\widehat{z}_{k'}(\lambda')\,\mathrm{d}\lambda',\end{equation} Now we consider the different cases.

(1) Suppose in (\ref{linearmapping}) we replace $y_{L_j}$ by $\rho^{L_j}+z_{L_j}$ for $j\in\{1,2\}$, then in particular we may assume that $\|y_{L_j}\|_{X^b}\lesssim L_j^{-1/2+\varepsilon_1+\varepsilon_2}$ due to (\ref{apriori1.8}). By (\ref{linearmapping}) and (\ref{fourierkernel}) we have
\begin{equation}\label{formulaL}\widehat{\Ls}_{kk'}(\lambda,\lambda')=\sum_{k_1-k_2=k-k'}\int_{\Rb^2} I(\lambda,\Omega+\lambda_1-\lambda_2+\lambda')\cdot V_{k_1-k_2}(\widehat{y_{L_1}})_{k_1}(\lambda_1)\cdot\overline{(\widehat{y_{L_2}})_{k_2}(\lambda_2)}\,\mathrm{d}\lambda_1\mathrm{d}\lambda_2
\end{equation} where $\Omega=|k|^2-|k_1|^2+|k_2|^2-|k'|^2$ and $I=I(\lambda,\mu)$ is as in (\ref{duhamleker}); we will omit the factor $\eta((k_1-k_2)/N^{1-\delta})$ in the definition of $\mathcal M^{<}$ in (\ref{lownonlin}) as it does not play a role. We may also assume that $|k_1-k_2|\sim R\lesssim L$. In the above expression, let $\mu:=\lambda-(\Omega+\lambda_1-\lambda_2+\lambda')$, in particular we have $|I|\lesssim\langle \lambda\rangle^{-1}\langle\mu\rangle^{-1}$ by (\ref{duhamleker}). By a routine argument, in proving (\ref{basicmatrix1}) we may assume $|\lambda_j|\leq L^{100}$ and $|\mu|\les L^{100}$; in fact, if say $|\lambda_1|$ is the maximum of these values and $|\lambda_1|\geq L^{100}$ (the other cases being similar), then we may fix the values of $k_j$, and hence $k-k'$, at a loss of at most $L^{12}$, and reduce to estimating
\[|\widehat{\Ls}_{kk'}(\lambda,\lambda')|\lesssim\int_{\Rb^3}\frac{1}{\<\lambda\>\<\lambda-\lambda_1+\lambda_2-\lambda_3-\Omega\>}\widehat{w_1}(\lambda_1)\overline{\widehat{w_2}(\lambda_2)}\,\mathrm{d}\lambda_1\mathrm{d}\lambda_2,\] with $|\lambda_1|\sim K\geq L^{100}$ and $\|\<\lambda_j\rangle^{b}\widehat{w_j}\|_{L^2}\lesssim 1$ for each $j$. By estimating $w_1$ in the unweighted $L^2$ norm we can gain a power $K^{-1/2}$, and using the $L_{\lambda_2}^1$ integrability of $\widehat{w_2}$ which follows from the weighted $L^2$ norm we can fix the value of $\lambda_2$. In the end this leads to
\[\sup_{k,k'}|\widehat{\Ls}_{kk'}(\lambda,\lambda')|\lesssim\langle\lambda\>^{-1}K^{-1/2}\] and hence \[\big\|\<\lambda\>^{1-b}\<\lambda'\>^{-b}\sup_{k,k'}|\widehat{\Ls}_{kk'}(\lambda,\lambda')|\big\|_{L_{\lambda,\lambda'}^2}\lesssim K^{-1/3}\lesssim L^{-30},\]which is more than enough, because $\|\widehat{\Ls}\|_{k\to k'}=\sup_{k,k'}|\widehat{\Ls}_{kk'}|$ if $\Ls$ is supported where $k-k'$ is constant.

Now we may assume $|\lambda_j|\leq L^{100}$ for $j\in\{1,2\}$ and $|\mu|\leq L^{100}$; we may also assume $|\lambda|+|\lambda'|\leq L^{\kappa^3}$ as otherwise we gain from the weights $\<\lambda\>^{2(1-b)}$ and $\<\lambda'\>^{-2b}$ in (\ref{basicmatrix1}). Similarly, in proving (\ref{basicmatrix2}) we may assume $|\lambda_j|\leq N^{100}$ for $j\in\{1,2\}$, $|\mu|\leq N^{100}$, and $|\lambda|+|\lambda'|\leq N^{100}$ (otherwise we may also fix $(k,k')$ and argue as above). Therefore, in proving (\ref{basicmatrix2}) we may replace the unfavorable exponents $\<\lambda\>^{2b}\<\lambda'\>^{-2(1-b)}$ by the favorable ones $\<\lambda\>^{2(1-b)}\<\lambda'\>^{-2b}$ at a price of $N^{C\kappa^{-1}}$; this will be acceptable since in the proof we will be able to gain a power $N^{-\delta/2}$. We remark that in the proof below (though not here), we may use the $Y^{1-b}$ norm as in (\ref{apriori1}) for the matrices in the decomposition of $y_{L_j}$; using the bounds of $\lambda_j$ as above, we may replace the exponent $1-b$ by $b$ (which then implies $L_{\lambda_j}^1$ integrability) again at a loss of either $L^{C\kappa^{-1}}$ or $N^{C\kappa^{-1}}$ depending on whether we are proving (\ref{basicmatrix1}) or (\ref{basicmatrix2}), which is acceptable. See also Section \ref{rem}.

This then allows us to fix the values of $\lambda_j$ in (\ref{formulaL}) using the $L_{\lambda_j}^1$ integrability coming from the weighted norms; moreover, by using the bound $|I|\lesssim\langle \lambda\rangle^{-1}\langle\mu\rangle^{-1}$, upper bounds for $\lambda$ and $\mu$ as above, and the weights in (\ref{basicmatrix1})--(\ref{basicmatrix2}), we may also fix the values of $\lambda$, $\lambda'$ and $\lfloor\mu\rfloor$, and reduce to estimating the quantity
\begin{equation}\label{quantity0-1}\Qc_{kk'}=\sum_{k_1-k_2=k-k'}h_{kk_1k_2k'}^{\mathrm{b}}(w_1)_{k_1}\overline{(w_2)_{k_2}},
\end{equation} where the tensor (which we call the \emph{base tensor})
\[h^{\mathrm{b}}=h_{kk_1k_2k'}^{\mathrm{b}}=V_{k_1-k_2}\cdot\mathbf{1}_{k_1-k_2+k'=k}\cdot \mathbf{1}_{|k|^2-|k_1|^2+|k_2|^2-|k'|^2=\Omega_0}\] with some value $\Omega_0$ determined by $\lambda_j$, $\lambda$, $\lambda'$ and $\lfloor\mu\rfloor$. Here we also assume $|k_j|\lesssim L_j$ and $|k_1-k_2|\sim R\lesssim L$, and $\|w_j\|_{\ell^2}\lesssim L_j^{-1/2+\varepsilon_1+\varepsilon_2}$. 

Now (\ref{quantity0-1}) is easily estimated by using Proposition \ref{merge} that
\[\|\Qc\|_{k\to k'}\lesssim\|h^{\mathrm{b}}\|_{kk_2\to k_1k'}\cdot \|w_1\|_{k_1}\cdot\|w_2\|_{k_2}\lesssim R\cdot R^{-\beta}\cdot L_1^{-1/2+\varepsilon_1+\varepsilon_2}L_2^{-1/2+\varepsilon_1+\varepsilon_2}\lesssim L^{-1/2+2\varepsilon_1-\varepsilon_2},\] which is enough for (\ref{basicmatrix1}) (namely we multiply this by the factor $\<\lambda\>^{-1}$ coming from $I$, and the weight $\<\lambda\>^{1-b}\<\lambda'\>^{-b}$ in (\ref{basicmatrix1}), then take the $L^2$ norm in $\lambda$ and $\lambda'$ to get (\ref{basicmatrix1}); the same happens below). For the $\|\Qc\|_{kk'}$ norm we have
\[\|\Qc\|_{kk'}\lesssim\|h^{\mathrm{b}}\|_{k_1\to kk_2k'}\cdot \|w_1\|_{k_1}\cdot\|w_2\|_{k_2}\lesssim R^{-\beta}\cdot NR\cdot L_1^{-1/2+\varepsilon_1+\varepsilon_2}L_2^{-1/2+\varepsilon_1+\varepsilon_2}\lesssim NL^{-1/2+2\varepsilon_1-\varepsilon_2},\] which is enough for (\ref{basicmatrix2}). Note that all the bounds for $h^{\mathrm{b}}$ we use here follow from Lemma \ref{counting}.

(2) Suppose $y_{L_1}$ is replaced by $\rho^{L_j}+z_{L_j}$, and $y_{L_2}$ is replaced by $\psi^{L_2}$. We may further decompose $\psi^{L_2}$ into $\zeta^{L_2,R_2}$ for $R_2\leq L_2/2$, (including the case $R_2=0$ by which we mean $\zeta^{L_2,0}=F_{L_2}$) and perform the same arguments as above fixing the $\lambda$ variables, and reduce\footnote{This reduction step actually involves a meshing argument as the estimate for $\Qc$ is probabilistic, see Section \ref{rem}.} to estimating the quantity
\begin{equation}\label{quantity0-2}\Qc_{kk'}=\sum_{k_1-k_2=k-k'}h_{kk_1k_2k'}^{\mathrm{b}}(w_1)_{k_1}\sum_{k_2'}\overline{h_{k_2k_2'}^{(2)}}\cdot\overline{(F_{L_2})_{k_2'}},
\end{equation} where $\|w_1\|_{\ell^2}\lesssim L_1^{-1/2+\varepsilon_1+\varepsilon_2}$, $h^{(2)}$ is independent from $F_{L_2}$, and is either the identity matrix or satisfies $\|h^{(2)}\|_{k_2\to k_2'}\lesssim R_2^{-1/2+3\varepsilon_1}$ and $\|h^{(2)}\|_{k_2k_2'}\lesssim L_2^{1+\delta} R_{2}^{-1/2+2\varepsilon_1}$. We then estimate (\ref{quantity0-2}) by
\begin{multline}\|\Qc\|_{k\to k'}\lesssim L_2^{-1}(\|h^{\mathrm{b}}\|_{kk_1k_2\to k'}+\|h^{\mathrm{b}}\|_{kk_1\to k_2k'})\|w_1\|_{k_1}\|h^{(2)}\|_{k_2\to k_2'}\\\lesssim R^{-\beta}\cdot R\min(L_1,L_2)\cdot L_2^{-1}L_1^{-1/2+\varepsilon_1+\varepsilon_2}\lesssim L^{-1/2+2\varepsilon_1-\varepsilon_2},\end{multline} using Propositions \ref{merge0} and \ref{trim}, which is enough for (\ref{basicmatrix1}). Note that here $h^{\mathrm{b}}$ depends on $k$ and $k'$ only via $k-k'$ and $|k|^2-|k'|^2$, and that $||k|^2-|k'|^2|\leq L^{\kappa^3}$ given the assumptions, so Proposition \ref{trim} is applicable. Similarly for the $\ell_{kk'}^2$ norm we have
\begin{multline}\|\Qc\|_{kk'}\lesssim L_2^{-1}(\|h^{\mathrm{b}}\|_{kk'\to k_1k_2}+\|h^{\mathrm{b}}\|_{kk_1k'\to k_2})\|w_1\|_{k_1}\|h^{(2)}\|_{k_2\to k_2'}\\\lesssim R^{-\beta}\cdot N(\min(L_1,L_2)+\min(L_1,R))\cdot L_2^{-1}L_1^{-1/2+\varepsilon_1+\varepsilon_2}\lesssim NL^{-1/2+2\varepsilon_1-\varepsilon_2},\end{multline} which is enough for (\ref{basicmatrix2}).

(3) Suppose $y_{L_j}$ is replaced by $\psi^{L_j}$ for $j\in\{1,2\}$. In this case we will start from (\ref{linearmapping}) and expand
\[(\psi^{L_j})_{k_j}=\sum_{k_j'}(H^{L_j})_{k_jk_j'}(F_{L_j})_{k_j'}\]for $j\in\{1,2\}$. There are then two cases, namely when $k_1'=k_2'$ or otherwise.

If $k_1'\neq k_2'$, then we can repeat the above argument (including further decomposing $\psi^{L_j}$ into $\zeta^{L_j,R_j}$ using (\ref{rao2}) and (\ref{simply}))  and fix the time Fourier variables, and reduce to estimating a quantity
\begin{equation}\label{quantity0-3}\Qc_{kk'}=\sum_{k_1-k_2=k-k'}h_{kk_1k_2k'}^{\mathrm{b}}\sum_{k_1',k_2'}h_{k_1k_1'}^{(1)}(F_{L_1})_{k_1'}\cdot\overline{h_{k_2k_2'}^{(2)}}\cdot\overline{(F_{L_2})_{k_2'}},
\end{equation} where $h^{(j)}$ is independent from $F_{L_j}$, and is either the identity matrix or satisfies $\|h^{(j)}\|_{k_j\to k_j'}\lesssim R_j^{-1/2+3\varepsilon_1}$ and $\|h^{(j)}\|_{k_jk_j'}\lesssim L_j^{1+\delta} R_{j}^{-1/2+2\varepsilon_1}$. Since $k_1'\neq k_2'$, we can apply Proposition \ref{trim}, either in $(k_1',k_2')$ jointly (if $L_1=L_2$) or first in $k_1'$ then in $k_2'$ (if, say, $L_1\geq 2L_2$) and get that
\begin{multline}\|\Qc\|_{k\to k'}\lesssim (L_1L_2)^{-1}\max(\|h^{\mathrm{b}}\|_{k\to k_1k_2k'},\|h^{\mathrm{b}}\|_{kk_1\to k_2k'},\|h^{\mathrm{b}}\|_{kk_2\to k_1k'},\|h^{\mathrm{b}}\|_{kk_1k_2\to k'})\\\times\|h^{(1)}\|_{k_1\to k_1'}\|h^{(2)}\|_{k_2\to k_2'}\lesssim R^{-\beta}(L_1L_2)^{-1}\cdot R\min(L_1,L_2)\lesssim L^{-2/3},\end{multline} which is enough for (\ref{basicmatrix1}). As for $\ell_{kk'}^2$ norm we have
\[\|\Qc\|_{kk'}\lesssim (L_1L_2)^{-1}\|h^{\mathrm{b}}\|_{kk_1k_2k_3}\cdot \|h^{(1)}\|_{k_1\to k_1'}\|h^{(2)}\|_{k_2\to k_2'}\lesssim (L_1L_2)^{-1}R^{-\beta}\cdot \min(L_1,L_2)^{3/2} NR\] which is enough for (\ref{basicmatrix2}).

Finally assume $k_1'=k_2'$, then $L_1=L_2=L$. In (\ref{linearmapping}) we the summation in $k_1'=k_2'$ gives
\[\sum_{k_1'}\frac{1}{\<k_1'\>^2}(H^L)_{k_1k_2'}(t')\overline{(H^L)_{k_2k_1'}}(t').\]Using the cancellation (\ref{unitarity}) since $k_1\neq k_2$, we can replace the factor $1/\<k_1'\>^2$ in the above expression by $1/\<k_1'\>^2-1/\<k_1\>^2$; then by further decomposing $H^{L_j}$ into $h^{L_j,R_j}$ by (\ref{simply})  and repeating the above arguments, we can reduce to estimating the quantity
\begin{equation}\label{quantity0-4}\Qc_{kk'}=\sum_{k_1-k_2=k-k'}h_{kk_1k_2k'}^{\mathrm{b}}\cdot(\widetilde{h})_{k_1k_2},\quad (\widetilde{h})_{k_1k_2}=\sum_{k_1'}\bigg(\frac{1}{\<k_1'\>^2}-\frac{1}{\<k_1\>^2}\bigg)h_{k_1k_1'}^{(1)}\overline{h_{k_2k_1'}^{(2)}},
\end{equation} where $h^{(j)}$ is either the identity matrix or satisfies $\|h^{(j)}\|_{k_j\to k_j'}\lesssim R_j^{-1/2+3\varepsilon_1}$ and $\|h^{(j)}\|_{k_jk_j'}\lesssim L_j^{1+\delta} R_{j}^{-1/2+2\varepsilon_1}$. Note that we may assume $|k_j-k_j'|\lesssim R_jL^\delta$ using the bound (\ref{apriori1.5}), so in particular we have
\[\bigg|\frac{1}{\<k_1'\>^2}-\frac{1}{\<k_1\>^2}\bigg|\lesssim \frac{R+\min(R_1,R_2)}{L^3}\] up to a loss of $L^{C\delta}$ (which is acceptable as in this case we can gain at least $L^{\varepsilon_2}$). Using these, we estimate, assuming without loss of generality that $R_1\geq R_2$:
\begin{multline}\|\Qc\|_{k\to k'}\lesssim \|h^{\mathrm{b}}\|_{kk_1k_2\to k'}\|\widetilde{h}\|_{k_1k_2}\lesssim\frac{R+R_2}{L^3}\cdot L^{1+\delta}R_1^{-1/2+2\varepsilon_1}R_2^{-1/2+3\varepsilon_1}\cdot R^{-\beta}L\min(R_1,R)\\\lesssim L^{-1/2+3\varepsilon_1-\varepsilon_2},\end{multline}
\[\|\Qc\|_{kk'}\lesssim\|h^{\mathrm{b}}\|_{k_1k_2\to kk'}\|\widetilde{h}\|_{k_1k_2}\lesssim \frac{R+R_2}{L^3}\cdot L^{1+\delta}R_1^{-1/2+2\varepsilon_1}R_2^{-1/2+3\varepsilon_1}\cdot R^{-\beta}NL\lesssim NL^{-2/3}.\] This completes the proof for (\ref{basicmatrix1}) and (\ref{basicmatrix2}). 
\subsection{The matrices $H^{N,L}$ and $h^{N,L}$}\label{linearbound} We now prove (\ref{apriori1})--(\ref{apriori1.5}). Let $\Ls^{N,L}$ be the linear operator defined by
\begin{equation}\label{defofL}z\mapsto -i\int_0^t\Delta_N\Mc^<(v_L,v_L,z)(t')\,\mathrm{d}t',\end{equation} we also extend its kernel in the same way as we do for $\Ls$ in Section \ref{Lbound}. Let $\widetilde{\Ls}^{N,L}=\Ls^{N,L}-\Ls^{N,L/2}$, then by induction hypothesis and the proof in Section \ref{Lbound}, we know that $\widetilde{\Ls}^{N,L}$ also satisfies the estimates (\ref{basicmatrix1})--(\ref{basicmatrix2}). Clearly (\ref{basicmatrix1}) implies that $\|\widetilde{\Ls}^{N,L}\|_{X^b\to X^{1-b}}\lesssim L^{-1/2+3\varepsilon_1-\varepsilon_2}$; moreover it is easy to see that \[\|\Ls^{N,L}z\|_{X^1}\lesssim\|\Mc^<(v_L,v_L,z)\|_{L_{t,x}^2}\lesssim L^{12}\|z\|_{X^0},\] hence $\|\Ls^{N,L}\|_{X^0\to X^1}\lesssim L^{12}$ and he same holds for $\widetilde{\Ls}^{N,L}$. By interpolation we obtain that $\|\widetilde{\Ls}^{N,L}\|_{X^\alpha\to X^\alpha}\lesssim L^{-1/2+3\varepsilon_1}$ for $\alpha\in\{b,1-b\}$ (note that we can always gain a positive power of $\tau$ using Lemma \ref{shorttime}, see Section \ref{rem}). Moreover, consider the kernel $(\Fc\widetilde{\Ls}^{N,L})_{kk'}(\lambda,\lambda')$, then we also have the bound
\[\int_\Rb\<\lambda\>^{2(1-b)}\|\<\lambda'\>^{-b}(\Fc\widetilde{\Ls}^{N,L})_{kk'}(\lambda,\lambda')\|_{k'\lambda'\to k}^2\,\mathrm{d}\lambda\lesssim L^{-1+6\varepsilon_1-2\varepsilon_2}\] which follows from (\ref{basicmatrix1}). If we replace the factor $\<\lambda'\>^{-b}$ by $1$, then a simple argument shows that
\[\|(\Fc\Ls^{N,L})_{kk'}(\lambda,\lambda')\|_{k'\lambda'\to k}\lesssim L^{12}\<\lambda\>^{-1}\] (and the same for $\widetilde{\Ls}^{N,L}$) by using that 
\[|(\Fc\Ls^{N,L}z)_k(\lambda)|\lesssim\<\lambda\>^{-1}\int_\Rb\<\lambda-\mu\>^{-1}|\Fc\Mc^<(v_L,v_L,z)_k(\mu)|\,\mathrm{d}\mu\lesssim\<\lambda\>^{-1}\|\Mc^<(v_L,v_L,z)\|_{L^2}\]and then fixing the Fourier modes of $v_L$. Interpolating again, we get that \begin{equation}\label{aux1}\int_\Rb\<\lambda\>^{2(1-b)}\|\<\lambda'\>^{-(1-b)}(\Fc\widetilde{\Ls}^{N,L})_{kk'}(\lambda,\lambda')\|_{k'\lambda'\to k}^2\,\mathrm{d}\lambda\lesssim L^{-1+6\varepsilon_1}.\end{equation} A similar interpolation gives\begin{equation}\label{aux2}\int_\Rb\<\lambda'\>^{-2b}\|\<\lambda\>^{b}(\Fc\widetilde{\Ls}^{N,L})_{kk'}(\lambda,\lambda')\|_{k'\to k\lambda}^2\,\mathrm{d}\lambda'\lesssim L^{-1+6\varepsilon_1}.\end{equation} Clearly $\Ls^{N,L}$ satisfies (\ref{aux1})--(\ref{aux2}) with right hand sides replaced by 1.

Now let \[\Hs^{N,L}=(1-\Ls^{N,L})^{-1}=\sum_{n=0}^\infty(\Ls^{N,L})^n,\] it is easy to see that $\Hs^{N,L}-1$ satisfies the same bounds (\ref{aux1})--(\ref{aux2}) with right hand sides replaced by 1; for example (\ref{aux1}) follows from iterating the bound
\begin{multline*}\big\|\<\lambda\>^{1-b}\|\<\lambda''\>^{-(1-b)}(\As\Bs)_{kk''}(\lambda,\lambda'')\|_{k''\lambda''\to k}\big\|_{L_\lambda^2}\\\lesssim \big\|\<\lambda\>^{1-b}\|\<\lambda'\>^{-(1-b)}\As_{kk'}(\lambda,\lambda'')\|_{k'\lambda'\to k}\big\|_{L_\lambda^2}\cdot\|\Bs\|_{X^{1-b}\to X^{1-b}}\end{multline*}
provided \begin{equation}\label{product}(\As\Bs)_{kk''}(\lambda,\lambda'')=\sum_{k'}\int_\Rb \As_{kk'}(\lambda,\lambda')\Bs_{k'k''}(\lambda',\lambda'')\,\mathrm{d}\lambda',\end{equation} and (\ref{aux2}) is proved similarly. Defining further
\[\widetilde{\Hs}^{N,L}=\Hs^{N,L}-\Hs^{N,L/2}=\sum_{n=1}^\infty(-1)^{n-1}(\Hs^{N,L}\widetilde{\Ls}^{N,L})^n\Hs^{N,L}.\] By iterating the $X^\alpha\to X^\alpha$ bounds and using also (\ref{basicmatrix2}) for $\widetilde{\Ls}^{N,L}$ we can show that
\begin{equation}\label{aux3}\int_{\Rb^2}\<\lambda\>^{2b}\<\lambda'\>^{-2(1-b)}\|(\Fc\widetilde{\Hs}^{N,L})_{kk'}(\lambda,\lambda')\|_{kk'}^2\,\mathrm{d}\lambda\mathrm{d}\lambda'\lesssim N^{2+2\delta}L^{-1+4\varepsilon_1}.\end{equation} The weighted bound
\begin{equation}\label{aux4}\int_{\Rb^2}\<\lambda\>^{2b}\<\lambda'\>^{-2(1-b)}\bigg\|\bigg(1+\frac{|k-k'|}{\min(L,N^{1-\delta})}\bigg)^\kappa(\Fc\widetilde{\Hs}^{N,L})_{kk'}(\lambda,\lambda')\bigg\|_{kk'}^2\,\mathrm{d}\lambda\mathrm{d}\lambda'\lesssim N^{3}\end{equation} is shown in the same way but using Proposition \ref{weighted}.

In addition, we can also show that
\begin{equation}\label{aux5}\int_{\Rb^2}\<\lambda\>^{2(1-b)}\<\lambda'\>^{-2b}\|(\Fc\widetilde{\Hs}^{N,L})_{k\to k'}(\lambda,\lambda')\|_{kk'}^2\,\mathrm{d}\lambda\mathrm{d}\lambda'\lesssim L^{-1+6\varepsilon_1}.\end{equation} This can be proved using (\ref{aux1})--(\ref{aux2}), by iterating the bounds
\begin{multline}\big\|\<\lambda\>^{1-b}\<\lambda''\>^{-b}\|(\As\Bs)_{kk'}(\lambda,\lambda'')\|_{k'\to k''}\big\|_{L_{\lambda,\lambda''}^2}\\\lesssim\big\|\<\lambda\>^{1-b}\<\lambda'\>^{-b}\|\As_{kk'}(\lambda,\lambda')\|_{k'\to k'}\big\|_{L_{\lambda,\lambda'}^2}\cdot\big\|\<\lambda''\>^{-b}\|\<\lambda'\>^{b}\Bs_{k'k''}(\lambda',\lambda'')\|_{k''\to k'\lambda'}\big\|_{L_{\lambda''}^2}
\end{multline} and similarly 
\begin{multline}\big\|\<\lambda\>^{1-b}\<\lambda''\>^{-b}\|(\As\Bs)_{kk'}(\lambda,\lambda'')\|_{k'\to k''}\big\|_{L_{\lambda,\lambda''}^2}\\\lesssim\big\|\<\lambda\>^{1-b}\|\<\lambda'\>^{-(1-b)}\As_{kk'}(\lambda,\lambda')\|_{k'\lambda'\to k}\big\|_{L_{\lambda}^2}\cdot\big\|\<\lambda'\>^{1-b}\<\lambda''\>^{-b}\|\Bs_{k'k''}(\lambda',\lambda'')\|_{k'\to k''}\big\|_{L_{\lambda',\lambda''}^2}
\end{multline} assuming (\ref{product}).

Now we can finally prove (\ref{apriori1})--(\ref{apriori1.5}). In fact, by definition of $\Hs^{N,L}$ and $\widetilde{\Hs}^{N,L}$, there exists an extension of $h^{N,L}$ such that
\[(\widehat{h^{N,L}})_{kk'}(\lambda)=\int_\Rb (\Fc\widetilde{\Hs}^{N,L})_{kk'}(\lambda,\lambda')\widehat{\chi}(\lambda')\,\mathrm{d}\lambda',\] so the $Y^{1-b}$ and $Z^b$ bounds in (\ref{apriori1}), as well as (\ref{apriori1.5}), can be deduced directly from (\ref{aux3})--(\ref{aux5}). The bound $\sup_t\|h^{N,L}(t)\|_{k\to k'}$ is also easily controlled by $\|\widetilde{\Hs}^{N,L}\|_{X^b\to X^b}$ using the embedding $L_t^\infty L^2\hookrightarrow X^b$. This completes the proof for (\ref{apriori1})--(\ref{apriori1.5}).

\section{Estimates for $\rho^N$}\label{estrhoN} In this section we prove the first bound in (\ref{apriori1.8}) reagrding $\rho^N$, assuming $N=M$. Recall that from (\ref{lineareqn}), (\ref{rao}) and (\ref{defxin}) we deduce that $\rho^N$ satisfies the equation
\begin{multline}\label{eqnrhon}(\rho^N)_k(t)=-i\int_0^t\Delta_N\Mc^<(v_{N/2},v_{N/2},\rho^N)_k(t')\,\mathrm{d}t'\\-i\Delta_N[\Mc^{\ll}(v_N,v_N,\psi^N+\rho^N)-\Mc^{\ll}(v_{N/2},v_{N/2},\psi^N+\rho^N)]_k(t')\,\mathrm{d}t'
\end{multline} with initial data $(\rho^N)_k(0)=0$. Let $\Ls^{N,L}$ be defined as in (\ref{defofL}), and denote $\Ls^N:=\Ls^{N,N/2}$. from Section \ref{linearbound} we know that $(1-\Ls^{N})^{-1}:=\Hs^N$ is well-defined, and has kernel $(\Hs^N)_{kk'}(t,t')$ in physical space and $(\Fc\Hs^N)_{kk'}(\lambda,\lambda')$ in Fourier space. Then (\ref{eqnrhon}) can be reduced to
\begin{equation}\label{eqnrhon2}
(\rho^N)_k(t)=\sum_{k'}\int_0^t (\Hs^N)_{kk'}(t,t')W_{k'}(t')\,\mathrm{d}t'
\end{equation} where
\begin{equation}\label{eqnrhon3}W_k(t)=-i\Delta_N\int_0^t\sum_{w_1,w_2,w_3}\Mc^{\ll}(w_1,w_2,w_3)_k(t')\,\mathrm{d}t'.
\end{equation} Here in (\ref{eqnrhon3}) we assume for $j\in\{1,2\}$ that $w_j\in\{\psi^{N_j},\rho^{N_j},z_{N_j}\}$ where $\max(N_1,N_2)=N$, and that $w_3\in\{\psi^N,\rho^N\}$.

In order to prove the bound for $\rho^N$ in (\ref{apriori1.5}), we will apply a continuity argument, namely assuming (\ref{apriori1.5}) and then improving it with a smallness factor. This can be done as long as we bound \begin{equation}\label{improverhon}\|W\|_{X^{b}(J)}\leq \tau^\theta N^{-1/2+\varepsilon_1+\varepsilon_2},\end{equation} since from Section \ref{linearbound} we know $\Hs^N$ is bounded from $X^b(J)$ to $X^b(J)$. In fact we will prove (\ref{improverhon}) with an extra gain $N^{-\varepsilon_2/2}$ which will allow us to ignore any possible $N^{C\delta}$ loss in the process. The smallness factor $\tau^\theta$ will be provided by Lemma \ref{shorttime} as in Section \ref{rem}, so we will not worry about it below. We divide the right hand side of (\ref{eqnrhon3}) into three terms:
\begin{itemize}
\item Term I: when $w_3=\rho^N$;
\item Term II: when $w_3=\psi^N$ and $z_{N'}\in\{w_1,w_2\}$ for some $N'\geq N/2$; 
\item Term III: when $w_3=\psi^N$ and $w_1,w_2\in\{\psi^N,\rho^N,\psi^{N/2},\rho^{N/2}\}$.
\end{itemize} Note that these are the only possibilities, since if (say) $N_1=N$, $w_1\in\{\psi^N,\rho^N\}$ and $N_2\leq N/2$, then we must have $N_2=N/2$ due to the support condition for $\psi^N$ and $\rho^N$, as well as the restriction $|k_1-k_2|\lesssim N^\varepsilon$ in $\Mc^\ll$. Moreover, the estimate of term I follows from the operator norm bound
\begin{equation}\label{normbound}\big\|I\Delta_N\Mc^{\ll}(y_{N_1},y_{N_2},z)\big\|_{X^b(J)}\lesssim\tau^\theta \max(N_1,N_2)^{-1/3}\|z\|_{X^b(J)}
\end{equation} which is proved by repeating the arguments in Section \ref{operbound} (the proof that works for $\Mc^<$ certainly also works for $\Mc^\ll$). In the next two sections we will deal with terms II and III respectively.
\subsection{Term II} Assume without loss of generality that $w_1=z_{N'}$. There are then two cases to consider, when $w_2\in\{\rho^{N_2},z_{N_2}\}$ or when $w_2=\psi^{N_2}$.
\subsubsection{The case $w_2\in\{\rho^{N_2},z_{N_2}\}$}\label{case2-1} If $w_2\in\{\rho^{N_2},z_{N_2}\}$, then we in particular have $\|w_2\|_{X^{b}(J)}\lesssim N_2^{-1/2+\varepsilon_1+\varepsilon_2}$. By Lemma \ref{extension}, we may fix an extension of $w_1$ and $w_2$ that satisfy the same bounds as they do but with $X^b(J)$ replaced by $X^b$; moreover they satisfy the same measurability conditions as $w_1$ and $w_2$. For simplicity we will still denote them by $w_1$ and $w_2$. The same thing is done for $w_3=\psi^N$, as well as the corresponding matrices.

Now, by (\ref{eqnrhon3}) and Lemma \ref{duhamelform0}, we can find an extension of II, which we still denote by II for simplicity, such that
\begin{multline}\label{fourierduhamel1}\widehat{\mathrm{II}_k}(\lambda)=\sum_{k_1-k_2+k_3=k}\int_{\Rb^3} I(\lambda,\Omega+\lambda_1-\lambda_2+\lambda_3)\cdot(\widehat{w_1})_{k_1}(\lambda_1)\cdot\overline{(\widehat{w_2})_{k_2}(\lambda_2)}\\\times\eta\bigg(\frac{k_1-k_2}{N^\varepsilon}\bigg)V_{k_1-k_2}\sum_{k_3'}(\widehat{H^N})_{k_3k_3'}(\lambda_3)(F_N)_{k_3'}\,\mathrm{d}\lambda_1\mathrm{d}\lambda_2\mathrm{d}\lambda_3
\end{multline} where $\Omega=|k|^2-|k_1|^2+|k_2|^2-|k_3|^2$ and $I=I(\lambda,\mu)$ is as in (\ref{duhamleker}). In the above expression, let $\mu:=\lambda-(\Omega+\lambda_1-\lambda_2+\lambda_3)$, in particular we have $|I|\lesssim\langle \lambda\rangle^{-1}\langle\mu\rangle^{-1}$ by (\ref{duhamleker}). By a routine argument we may assume $|\lambda|\leq N^{100}$ and similarly for $\mu$ and each $\lambda_j$; in fact, if say $|\lambda_1|$ is the maximum of these values and $|\lambda_1|\geq N^{100}$, then we may fix the values of $k$ and all $k_j$ at a loss of at most $N^{12}$, and reduce to estimating (with the value of $\Omega$ fixed)
\[|\widehat{\mathrm{II}}(\lambda)|\lesssim\int_{\Rb^3}\frac{1}{\<\lambda\>\<\lambda-\lambda_1+\lambda_2-\lambda_3-\Omega\>}\big|\widehat{w_1}(\lambda_1)\overline{\widehat{w_2}(\lambda_2)}\widehat{w_3}(\lambda_3)\big|\,\mathrm{d}\lambda_1\mathrm{d}\lambda_2\mathrm{d}\lambda_3,\] with $|\lambda_1|\sim K\geq N^{100}$ and $\|\<\lambda_j\rangle^{b}\widehat{w_j}\|_{L^2}\lesssim 1$ for each $j$. By estimating $w_1$ in the unweighted $L^2$ norm we can gain a power $K^{-1/2}$, and using the $L^1$ integrability of $\widehat{w_j}$ which follows from the weighted $L^2$ norms we can fix the values of $\lambda_j$ for $j\in\{2,3\}$. In the end this leads to
\[|\widehat{\mathrm{II}}(\lambda)|\lesssim\mathbf{1}_{|\lambda|\lesssim K}\langle\lambda\>^{-1}K^{-1/2}\] and hence $\|\<\lambda\>^{b}\widehat{\mathrm{II}}\|_{L^2}\lesssim K^{-1/3}\lesssim N^{-30}$, which is more than enough for (\ref{apriori1.8}). 

Now, with $|\lambda|\leq N^{100}$ etc., we may apply the bounds (\ref{apriori1})--(\ref{apriori1.8}), but for the extensions and global norms, and replace the $Y^{1-b}$ norm (if any) by the $Y^b$ norm at a loss of $N^{C\kappa^{-1}}$ which will be neglected as stated above. Similarly, as $|\lambda|\leq N^{100}$, we also only need to estimate $\mathrm{II}$ in the $X^{1-b}$ instead of $X^b$ norm  again at a loss of $N^{C\kappa^{-1}}$. Then, using $L^1$ integrability in $\lambda_j$ (together with a meshing argument, see Section \ref{rem}) provided by the weighted bounds (\ref{apriori1})--(\ref{apriori1.8}), and the (almost) summability in $\mu$ due to the $\<\mu\>^{-1}$ factor in (\ref{duhamleker}), we may fix the values of $\lambda$, $\lambda_j\,(1\leq j\leq 3)$ and $\lfloor\mu\rfloor$ (and hence the value of $\Omega\in\Zb$) and reduce to estimating the $\ell_k^2$ norm of the following quantity
\begin{equation}\label{quantity}\Qc_k:=\sum_{\substack{k_1-k_2+k_3=k\\|k|^2-|k_1|^2+|k_2|^2-|k_3|^2=\Omega_0}}\eta\bigg(\frac{k_1-k_2}{N^\varepsilon}\bigg)V_{k_1-k_2}\cdot(\widehat{w_1})_{k_1}\overline{(\widehat{w_2})_{k_2}}\cdot\sum_{k_3'}H_{k_3k_3'}(F_N)_{k_3'}.
\end{equation} Here in (\ref{quantity}) we assume that $|k_1|\leq N$, $|k_2|\leq N_2$, $|k_1-k_2|\lesssim N^\varepsilon$ and $N/2< \<k_3\>,\<k_3'\>\leq N$, and $\Omega_0\in\Zb$ is fixed, and the inputs satisfy that
\[\|\widehat{w_1}\|_{\ell^2}\lesssim N^{-1/2+\varepsilon_1},\quad \|\widehat{w_2}\|_{\ell^2}\lesssim N_2^{-1/2+\varepsilon_1+\varepsilon_2},\quad \|H\|_{k_3\to k_3'}\lesssim 1.
\] To estimate $\Qc$, we may assume $|k_1-k_2|\sim R\lesssim N^{\varepsilon}$, and define the base tensor
\[h^{\mathrm{b}}=h_{kk_1k_2k_3}^{\mathrm{b}}=\eta\bigg(\frac{k_1-k_2}{N^\varepsilon}\bigg)V_{k_1-k_2}\cdot\mathbf{1}_{k_1-k_2+k_3=k}\cdot \mathbf{1}_{|k|^2-|k_1|^2+|k_2|^2-|k_3|^2=\Omega_0},\] with also the restrictions on $k_j$ as above. Then we have
\[\Qc_k=\sum_{k_1,k_2,k_3,k_3'}h_{kk_1k_2k_3}^{\mathrm{b}}\cdot(\widehat{w_1})_{k_1}\overline{(\widehat{w_2})_{k_2}}\cdot H_{k_3k_3'}(F_N)_{k_3'}\] and hence
\[\|\Qc\|_{\ell^2}\lesssim N^{-1/2+\varepsilon_1}N_2^{-1/2+\varepsilon_1+\varepsilon_2}\bigg\|\sum_{k_3,k_3'}h_{kk_1k_2k_3}^{\mathrm{b}}H_{k_3k_3'}(F_N)_{k_3'}\bigg\|_{kk_2\to k_1}.\] By Lemma \ref{trim} and the independence between $H_{k_3k_3'}$ and $(F_N)_{k_3'}$, we get that
\begin{multline*}\bigg\|\sum_{k_3,k_3'}h_{kk_1k_2k_3}^{\mathrm{b}}H_{k_3k_3'}(F_N)_{k_3'}\bigg\|_{kk_2\to k_1}\lesssim N^{\delta}\cdot \max\bigg[N^{-1}\bigg\|\sum_{k_3}h_{kk_1k_2k_3}^{\mathrm{b}}H_{k_3k_3'}\bigg\|_{kk_2k_3'\to k_1},\bigg\|\sum_{k_3}h_{kk_1k_2k_3}^{\mathrm{b}}H_{k_3k_3'}\bigg\|_{kk_2\to k_1k_3'}\bigg]\\
\lesssim N^{\delta-1}\|H\|_{k_3\to k_3'}\cdot\max\big(\big\|h_{kk_1k_2k_3}^{\mathrm{b}}\big\|_{kk_2k_3\to k_1},\|h_{kk_1k_2k_3}^{\mathrm{b}}\big\|_{kk_2\to k_1k_3}\big)\end{multline*} $N$-certainly. By the definition of $h^{\mathrm{b}}$, and using Schur's bound and counting estimates in Lemma \ref{counting}, and noticing that $|k_2|\leq N_2$ and $|k_1-k_2|\lesssim R$, we can bound\[\max\big(\big\|h_{kk_1k_2k_3}^{\mathrm{b}}\big\|_{kk_2k_3\to k_1},\|h_{kk_1k_2k_3}^{\mathrm{b}}\big\|_{kk_2\to k_1k_3}\big)\lesssim N^\delta R^{-\beta}\cdot N\cdot\min(N_2,R).\] Since also $\|H\|_{k_3\to k_3'}\lesssim 1$, we conclude that
\begin{equation}\label{rhoest1}\|\Qc\|_{\ell^2}\lesssim N^{-1/2+\varepsilon_1}N_2^{-1/2+\varepsilon_1+\varepsilon_2}\cdot N^{2\delta}R^{-\beta}\min(N_2,R)\lesssim N^{-1/2+\varepsilon_1+\varepsilon_2/2},\end{equation} which is enough for (\ref{apriori1.8}). This concludes the proof for term II when $w_2\in\{\rho^{N_2},z_{N_2}\}$. Note that the above argument also works for the case when $w_1=\rho^N$ and $w_2=\rho^{N_2}$, because here we must have $N_2\geq N/2$ due to the support condition of $\rho^N$ and the assumption $|k_1-k_2|\lesssim N^\varepsilon$, and the above arguments give the same (in fact better) estimates.
\subsubsection{The case $w_2=\psi^{N_2}$}\label{case2-2} In this case, by repeating the first part of the arguments in Section \ref{case2-1}, we can reduce to estimating the $\ell_k^2$ norm of the quantity
\begin{equation}\label{quantity2}\Qc_k:=\sum_{\substack{k_1-k_2+k_3=k\\|k|^2-|k_1|^2+|k_2|^2-|k_3|^2=\Omega_0}}\eta\bigg(\frac{k_1-k_2}{N^\varepsilon}\bigg)V_{k_1-k_2}\cdot(\widehat{w_1})_{k_1}\cdot\sum_{k_2'}\overline{H_{k_2k_2'}^{(2)}(F_{N_2})_{k_2'}}\cdot\sum_{k_3'}H_{k_3k_3'}^{(3)}(F_N)_{k_3'}.
\end{equation} Here in (\ref{quantity2}) we assume that $|k_1|\leq N$, $|k_2|\leq N_2$, $|k_1-k_2|\sim R\lesssim N^\varepsilon$, $N_2/2< \<k_2\>,\<k_2'\>\leq N_2$ and $N/2< \<k_3\>,\<k_3'\>\leq N$, and $\Omega_0\in\Zb$ is fixed, and the inputs satisfy that
\[\|\widehat{w_1}\|_{\ell^2}\lesssim N^{-1/2+\varepsilon_1},\quad \|H^{(j)}\|_{k_j\to k_j'}\lesssim 1\,(j=2,3).
\] Moreover, this $H^{(j)}$ is such that either $H^{(j)}=\mathrm{Id}$ or $\|H^{(j)}\|_{k_jk_j'}\lesssim N_j^{1+\delta}$ with $N_3=N$. The sum in (\ref{quantity2}) can be decomposed into a term where $k_2'\neq k_3'$ and a term where $k_2'=k_3'$.

\emph{Case 1: $k_2'\neq k_3'$}. Let $h_{kk_1k_2k_3}^{\mathrm{b}}$ be defined as above, it suffices to estimate the $\ell_{k_1}^2\to\ell_k^2$ norm of the tensor
\[(k,k_1)\mapsto\sum_{k_2,k_3}h_{kk_1k_2k_3}^{\mathrm{b}}\sum_{k_2',k_3'}\overline{H_{k_2k_2'}^{(2)}(F_{N_2})_{k_2'}}\cdot H_{k_3k_3'}^{(3)}(F_N)_{k_3'}.\] by using the $\ell^2$ norm of $w_1$. If $N_3=N$, then the tensors $h^{\mathrm{b}}$, $H^{(2)}$ and $H^{(3)}$ are independent from $(F_N)_{k_2'}$ and $(F_N)_{k_3'}$, and $k_2'\neq k_3'$, so we can apply Lemma \ref{trim}; if $N_2\leq N/2$, then $h^{\mathrm{b}}$, $H^{(2)}$ and $H^{(3)}$ and $(F_{N_2})_{k_2'}$ are all independent from $(F_N)_{k_3'}$, and moreover $h^{\mathrm{b}}$ and $H^{(2)}$ are independent from $(F_{N_2})_{k_2'}$, so we can apply Lemma \ref{trim} iteratively, first for the sum in $(k_3,k_3')$, and then for the sum in $(k_2,k_2')$. In either case, by applying Lemma \ref{trim} and combining it with Lemma \ref{merge} and estimating $H^{(j)}$ in the $k_j\to k_j'$ norm, we obtain $N$-certainly that the desired $\ell_{k_1}^2\to\ell_k^2$ norm of the tensor is bounded by
\[N^\delta N_2^{-1}N^{-1}\cdot\max\big(\|h^{\mathrm{b}}\|_{kk_2k_3\to k_1},\|h^{\mathrm{b}}\|_{kk_2\to k_1k_3},\|h^{\mathrm{b}}\|_{kk_3\to k_1k_2},\|h^{\mathrm{b}}\|_{k\to k_1k_2k_3}\big).\] Using the fact that $|k_2|\leq N_2$ and $|k_3|\leq N$ in the support of $h^{\mathrm{b}}$, and Lemma \ref{counting} as above, we can show that
\[\max\big(\|h^{\mathrm{b}}\|_{kk_2k_3\to k_1},\|h^{\mathrm{b}}\|_{kk_2\to k_1k_3},\|h^{\mathrm{b}}\|_{kk_3\to k_1k_2},\|h^{\mathrm{b}}\|_{k\to k_1k_2k_3}\big)\lesssim R^{-\beta}N^\delta\cdot NN_2,\] hence
\begin{equation}\label{rhoest2}\|\Qc\|_{\ell^2}\lesssim N^{-1/2+\varepsilon_1}\cdot R^{-\beta}N^{2\delta}\lesssim N^{-1/2+\varepsilon_1+\varepsilon_2/2},\end{equation} which is enough for (\ref{apriori1.8}).

\emph{Case 2: $k_2'=k_3'$}. In this case we must have $N_2=N$, and we can reduce (\ref{quantity2}) to the expression\footnote{Here we are simplifying by replacing $|g_{k_2'}|^2$ by $1$ (we will do the same below). This is because $\Eb(|g_k|^2-1)=0$, so any large deviation estimate satisfied by linear combinations of $g_k$, which is the only thing we rely on, will hold also for linear combinations of $|g_k|^2-1$, so the contribution of $|g_{k_2'}|^2-1$ can always be treated in the same way as the $k_2'\neq k_3'$ case.}
\begin{equation}\label{quantity3}\Qc_k=\sum_{\substack{k_1-k_2+k_3=k\\|k|^2-|k_1|^2+|k_2|^2-|k_3|^2=\Omega_0}}\eta\bigg(\frac{k_1-k_2}{N^\varepsilon}\bigg)V_{k_1-k_2}\cdot(\widehat{w_1})_{k_1}\cdot(\widetilde{H})_{k_2k_3}
\end{equation} where
\[(\widetilde{H})_{k_2k_3}=\sum_{k_2'}\frac{1}{\<k_2'\>^2}\overline{H_{k_2k_2'}^{(2)}}\cdot H_{k_3k_2'}^{(3)}.\] As $k_2\neq k_3$ in (\ref{quantity2}) due to the definition of $\Mc^{\ll}$, we know that either $H^{(2)}$ or $H^{(3)}$ must not be identity, hence we have $\|\widetilde{H}\|_{\ell_{k_2k_3}^2}\lesssim N^{-1+\delta}$. By (\ref{quantity3}) we then simply estimate
\begin{equation}\label{rhoest3}\|\Qc\|_{\ell_k^2}\lesssim\|\widehat{w_1}\|_{\ell^2}\cdot\|\widetilde{H}\|_{\ell_{k_2k_3}^2}\cdot\|h^{\mathrm{b}}\|_{kk_2k_3\to k_1}\lesssim N^{-1/2+\varepsilon_1}\cdot N^{-1+\delta}\cdot R^{-\beta}\cdot NR\lesssim N^{-1/2+\varepsilon_1+\varepsilon_2/2}\end{equation} using Lemma \ref{counting}, noticing that $|k_1-k_2|\lesssim R$ and $|k_3|\leq N$. This completes the proof for term II.
\subsection{Term III}\label{sectionIII} Here we assume $w_3=\psi^N$ and $w_1,w_2\in\{\psi^N,\rho^N,\psi^{N/2},\rho^{N/2}\}$. We consider two possibilities, when $w_1,w_2\in\{\psi^N,\psi^{N/2}\}$, which we call term IV, and when $w_j\in\{\rho^N,\rho^{N/2}\}$ for some $j\in\{1,2\}$, which we call term V.
\subsubsection{Term IV}\label{sectionIV} Suppose $w_1,w_2\in \{\psi^N,\psi^{N/2}\}$. We may also decompose them into $\psi^{N_j,L_j}$ for $L_j\leq N_j/2$, and reduce to
\begin{equation}\label{termIV}
\begin{aligned}\mathrm{IV}_k(t)&=-i\Delta_N\int_0^t\sum_{k_1-k_2+k_3=k}e^{it'\Omega}\eta\bigg(\frac{k_1-k_2}{N^\varepsilon}\bigg)V_{k_1-k_2}\\&\times\sum_{k_1',k_2',k_3'}(h^{N_1,L_1})_{k_1k_1'}(t')\overline{(h^{N_2,L_2})_{k_2k_2'}(t')}(h^{N_3,L_3})_{k_3k_3'}(t')(F_{N_1})_{k_1'}\overline{(F_{N_2})_{k_2'}}(F_{N_3})_{k_3'}\,\mathrm{d}t',
\end{aligned}
\end{equation} where $N_1,N_2\in\{N,N/2\}$ and $N_3=N$. In (\ref{termIV}) we consider two cases, depending on whether there is a pairing $k_1'=k_2'$ or $k_2'=k_3'$, or not.

\emph{Case 1: no pairing}. Assume that $k_2'\not\in\{k_1',k_3'\}$, then we take the Fourier transform in the time variable $t$, and repeat the first part of the arguments in Section \ref{case2-1}, to reduce to estimating the $\ell_k^2$ norm of the quantity
\begin{multline}\label{quantity4}\Qc_k:=\sum_{\substack{k_1-k_2+k_3=k\\|k|^2-|k_1|^2+|k_2|^2-|k_3|^2=\Omega_0}}\eta\bigg(\frac{k_1-k_2}{N^\varepsilon}\bigg)V_{k_1-k_2}\\\times\sum_{k_1'}h_{k_1k_1'}^{(1)}(F_{N_1})_{k_1'}\cdot\sum_{k_2'}\overline{h_{k_2k_2'}^{(2)}(F_{N_2})_{k_2'}}\cdot\sum_{k_3'}h_{k_3k_3'}^{(3)}(F_{N_3})_{k_3'}.
\end{multline} In (\ref{quantity4}) we assume that $|k_j|\sim N$ and $|k_1-k_2|\sim R\lesssim N^\varepsilon$, and that the matrices $h^{(j)}$ is either identity or satisfies that
\[\|h^{(j)}\|_{k_j\to k_j'}\lesssim L_j^{-1/2+3\varepsilon_1},\quad \|h^{(j)}\|_{k_jk_j'}\lesssim N^{1+\delta}L_j^{-1/2+2\varepsilon_1},\] and moreover we may assume $h^{(j)}$ is supported in $|k_j-k_j'|\lesssim L_j N^{\delta}$ by inserting a cutoff exploiting (\ref{apriori1.5}). The $\ell_k^2$ norm for $\Qc_k$ can the be estimated using Proposition \ref{trim} in the same way as in Section \ref{case2-2}, either jointly in $(k_1',k_2',k_3')$ if each $N_j=N$ or first in those $k_j$ with $N_j=N$ and then in those $k_j$ with $N_j=N/2$, so that $N$-certainly we have (with the base tensor $h^{\mathrm{b}}$ defined as in Sections \ref{case2-1} and \ref{case2-2} above)
\[\|\Qc\|_{\ell^2}\lesssim N^\delta\cdot N^{-3}\|h^{\mathrm{b}}\|_{kk_1k_2k_3}\prod_{j=1}^3\|h^{(j)}\|_{k_j\to k_j'}\lesssim N^{-3+\delta}\cdot N^{3/2}RN\cdot R^{-\beta}\lesssim N^{-1/2+\varepsilon_1/2}\] using Lemma \ref{counting}, which is enough for (\ref{apriori1.8}).

\emph{Case 2: pairing}. We now consider the cases when $k_1'=k_2'$ or $k_2'=k_3'$. First, if $k_2'=k_3'$, then we can apply the reduction arguments as above and reduce to
\begin{equation}\label{quantity5}\Qc_k:=\sum_{\substack{k_1-k_2+k_3=k\\|k|^2-|k_1|^2+|k_2|^2-|k_3|^2=\Omega_0}}\eta\bigg(\frac{k_1-k_2}{N^\varepsilon}\bigg)V_{k_1-k_2}\cdot\sum_{k_1'}h_{k_1k_1'}^{(1)}(F_{N_1})_{k_1'}\cdot\widetilde{h}_{k_2k_3}
\end{equation} where 
\[(\widetilde{h})_{k_2k_3}=\sum_{k_2'}\frac{1}{\<k_2'\>^2}\overline{h_{k_2k_2'}^{(2)}}\cdot h_{k_3k_2'}^{(3)};\quad \|\widetilde{h}\|_{k_2k_3}\lesssim N^{-2}\min(\|h^{(2)}\|_{k_2\to k_2'}\|h^{(3)}\|_{k_3k_3'},\|h^{(2)}\|_{k_2k_2'}\|h^{(3)}\|_{k_3\to k_3'}).\] Note that $h^{(2)}$ and $h^{(3)}$ cannot both be identity as $k_2\neq k_3$. Now if $\max(L_2,L_3)\leq N/2$ then due to independence, applying similar arguments as before we can estimate $N$-certainly that
\[\|\Qc\|_{\ell^2}\lesssim N^\delta N^{-1}\cdot\|h^{(1)}\|_{k_1\to k_1'}\|h^{\mathrm{b}}\|_{kk_1\to k_2k_3}\|\widetilde{h}\|_{k_2k_3}\lesssim N^{-2+2\varepsilon+4\delta},\] using the constraint $|k_1-k_2|\lesssim N^\varepsilon$, which is enough for (\ref{apriori1.8}).; if $\max(L_2,L_3)=N$ then we can gain a negative power of this value and view $F_{N_1}$ simply as an $H^{-1/2-}$ function (without considering randomness) and bound 
\[\|\Qc\|_{\ell^2}\lesssim N^\delta N^{1/2+\delta}\cdot\|h^{(1)}\|_{k_1\to k_1'}\|h^{\mathrm{b}}\|_{kk_1\to k_2k_3}\|\widetilde{h}\|_{k_2k_3}\lesssim N^{-1+2\varepsilon+4\delta},\] which is also enough for (\ref{apriori1.8})..

Finally consider the case $k_1'=k_2'$, so in particular $N_1=N_2$. We will sum over $L_1$ and $L_2$ in order to exploit the cancellation (\ref{unitarity}) (as $k_1\neq k_2$); this leads to the expression
\[\sum_{k_1'}\frac{1}{\<k_1'\>^2}(H^{N_1})_{k_1k_1'}(t')\overline{(H^{N_1})_{k_2k_1'}(t')}\] where again we have replaced $|g_{k_1'}|^2$ by $1$ as before. Since $k_1\neq k_2$, by (\ref{unitarity}), we may replace $\<k_1'\>^{-2}$ in the above expression by $\<k_1'\>^{-2}-\<k_1\>^{-2}$. Then, decomposing in $L_1$ and $L_2$ again and taking Fourier transform in $t$ and repeating the reduction steps as before, we arrive at the quantity
\begin{equation}\label{quantity6}
\Qc_k:=\sum_{\substack{k_1-k_2+k_3=k\\|k|^2-|k_1|^2+|k_2|^2-|k_3|^2=\Omega_0}}\eta\bigg(\frac{k_1-k_2}{N^\varepsilon}\bigg)V_{k_1-k_2}\cdot\widetilde{h}_{k_1k_2}\cdot\sum_{k_3'}h_{k_3k_3'}^{(3)}(F_{N_3})_{k_3'}
\end{equation} where
\[\widetilde{h}_{k_1k_2}=\sum_{k_1'}\bigg(\frac{1}{\<k_1'\>^2}-\frac{1}{\<k_1\>^2}\bigg)h_{k_1k_1'}^{(1)}\overline{h_{k_2k_1'}^{(2)}}\] with $h^{(j)}$ as above. Note that may assume $|k_1-k_1'|\lesssim N^\delta\min(L_1,N^\varepsilon+L_2)\lesssim N^{\varepsilon+\delta}\min(L_1,L_2)$ in view of $|k_1-k_2|\lesssim N^\varepsilon$, it is easy to show, assuming $\min(L_1,L_2)=L$, that
\[\|\widetilde{h}\|_{k_1k_2}\lesssim N^{\varepsilon+\delta}\cdot LN^{-3}\|h^{(1)}\|_{k_1k_1'}\|h^{(2)}\|_{k_2\to k_2'}\lesssim N^{-2}N^{\varepsilon+\delta+4\varepsilon_1}.\] Since $\max(L_1,L_2)\leq N/2$, using independence and arguing as before, we can estimate that $N$-certainly,
\[\|\Qc\|_{\ell^2}\lesssim N^\delta\cdot N^{-1}\|h^{(3)}\|_{k_3\to k_3'}\cdot\|h^{\mathrm{b}}\|_{kk_3\to k_1k_2}\cdot \|\widetilde{h}\|_{k_1k_2}\lesssim N^{-1+\varepsilon+2\delta+4\varepsilon_1}\] which is enough for (\ref{apriori1.8}).. This completes the estimate for term IV.
\subsubsection{Properties of the matrix $M^N-H^N$} Before studying term V, we first establish some properties of the matrix $Q^N:=M^N-H^N=(Q^N)_kk'(t)$ such that
\begin{equation}\label{matrixQ}(\rho^N)_k(t)=\sum_{k'}(Q^N)_{kk'}(t)(F_N)_{k'}.\end{equation}
\begin{lem} \label{decomposeQ}Let $\varepsilon':=\sqrt{\varepsilon}$ so that $(\varepsilon_1\ll)\,\varepsilon\ll\varepsilon'\ll 1$. Then we have \begin{equation}\label{basicQ}\|Q^N\|_{Y^{1-b}(J)}+\sup_{t\in J}\|(Q^N)_{kk'}(t)\|_{k\to k'}\lesssim N^{-1/2+3\varepsilon_1},\quad \|Q^N\|_{Z^b(J)}\lesssim N^{1/2+2\varepsilon_1}.\end{equation} Moreover we can decompose $Q^N=Q^{N,\ll}+Q^{N,\mathrm{rem}}$ such that $\|Q^{N,\mathrm{rem}}\|_{Z^b(J)}\lesssim N^{1/2+2\varepsilon_1-\varepsilon'/4}$, and that
\begin{equation}\label{rhoNrem}\|\rho^{N,\mathrm{rem}}\|_{X^b(J)}\lesssim N^{-1/2+2\varepsilon_1-\varepsilon'/4},\quad\textrm{where}\quad (\rho^{N,\mathrm{rem}})_k(t)=\sum_{k'}(Q^{N,\mathrm{rem}})_{kk'}(t)(F_N)_{k'}.
\end{equation} Moreover $Q^{N,\ll}$ can be decomposed into at most $N^{C\varepsilon'}$ terms. For each term $Q$ there exist vectors $\ell^*,m^*$ such that $|\ell^*|,|m^*|\lesssim N^{\varepsilon'}$, and that $(\widehat{Q})_{kk'}(\lambda)$ is a linear combination (in the form of some integral\footnote{Strictly speaking this means $(\widehat{Q})_{kk'}(\lambda)=\int a(\mu)\mathbf{1}_{k'-k=\ell^*}\cdot\Yc_{\ell^*,m^*}(k,\lambda,\mu)\cdot\Rc_{\ell^*,m^*}(k,\mu)$ where the integration is taken over some Euclidean space, $a(\mu)\in L^1$, and the bounds for $\Rc$ and $\Yc$ are uniform in $\mu$.}), with summable coefficients, of expressions of form
\begin{equation}\label{lem5.1:form}
\mathbf{1}_{k'-k=\ell^*}\cdot\Yc_{\ell^*,m^*}(k,\lambda)\cdot\Rc_{\ell^*,m^*}(k)
\end{equation} where $\Yc$ is independent with $(F_N)_k$, and $|\Yc|\lesssim 1$ and $\Rc(k)$ \emph{depends only on $m^*\cdot k$}, moreover we have \begin{equation}\label{lem5.1:est1}\|\Rc\|_{\ell_k^2}\lesssim N^{1/2+2\varepsilon_1+C\varepsilon'},\quad \|\<\lambda\>^{b}\Yc\|_{L_\lambda^2\ell_k^\infty}\lesssim N^{C\varepsilon'}.\end{equation}
\end{lem}
\begin{rem}
Lemma \ref{decomposeQ} plays an important role in Section \ref{subsec:TermV} when estimating Term V. In particular,  we will exploit the one-dimensional extra independence of $\Rc(k)$ with $(F_N)_k$, since $\Rc(k)$ depends only on $m^*\cdot k$ instead of on $k$.
\end{rem}

\begin{proof} By definition of $\xi^N$ and $\psi^N$ in (\ref{defxin3}) and (\ref{eqnrhon})--(\ref{eqnrhon3}), as well as the associated matrices, we have the identity 
\[(Q^N)_{kk'}(t)=\sum_{k_1}\int_0^t (\Hs^N\Ms)_{kk_1}(t,t_1)(H^N-Q^N)_{k_1k'}(t_1)\,\mathrm{d}t_1\]
and hence we have
\begin{equation}\label{identity}(Q^N)_{kk'}(t)=\sum_{n=1}^\infty(-1)^{n-1}\sum_{k_1}\int_0^t\{(\Hs^N\Ms)^n\}_{kk_1}(t,t_1)(H^N)_{k_1k'}(t_1)\,\mathrm{d}t_1,\end{equation}where $\Hs^N=\Hs^{N,N/2}$ is defined in Section \ref{linearbound}, and $\Ms$ denotes the operator
\begin{equation}\label{defM}z\mapsto \sum_{\max(N_1,N_2)=N}\Delta_N\Mc^{\ll}(y_{N_1},y_{N_2},z).\end{equation} 

The bounds in (\ref{basicQ}) then follow from iterating like in Section \ref{linearbound} using the bounds (\ref{aux1})--(\ref{aux5}) (together with the $X^\alpha\to X^\alpha$ bounds) for the operators $\Hs^N$ and $\Ms$, where the bounds for $\Ms$ is proved in the same way as in Sections \ref{operbound} and \ref{linearbound}. Moreover, in (\ref{identity}) if we assume $n\geq 2$ or replace $\Hs^N$ by $\Hs^N-\Hs^{N,N^{\varepsilon'}}$ (or $H^N$ by $H^N-H^{N,N^{\varepsilon'}}$) then the corresponding bounds can be improved by $N^{-\varepsilon'/4}$, and the resulting terms can be put in\footnote{To prove (\ref{rhoNrem}), we may repeat the proofs above for terms I--IV, and then treat V in the same way as II. This leads to a loss of $N^{O(\varepsilon_1)}$, which will be negligible compared to the gain $N^{\varepsilon'/4}$.} $Q^{\mathrm{rem}}$. As for the remaining contribution, we can write
\[(Q^{N,\ll})_{kk'}(t)=\sum_{k_1,k_2}\int\Hs_{kk_1}^{N,N^{\varepsilon'}}(t,t_1)\Ms_{k_1k_2}(t_1,t_2)(H^{N,N^{\varepsilon'}})_{k_2k'}(t_2)\,\mathrm{d}t_1\mathrm{d}t_2,\] hence\begin{equation}\label{express}(\widehat{Q^{N,\ll}})_{kk'}(\lambda)=\sum_{k_1,k_2}\int (\Fc\Hs^{N,N^{\varepsilon'}})_{kk_1}(\lambda,\lambda_1)(\Fc\Ms)_{k_1k_2}(\lambda_1,\lambda_2)(\Fc H^{N,N^{\varepsilon'}})_{k_2k'}(\lambda_2)\,\mathrm{d}\lambda_1\mathrm{d}\lambda_2.\end{equation} We may assume $|k-k_1|\lesssim N^{\varepsilon'}$ and the same for $k_1-k_2$ (using the definition of $\Ms$) and $k_2-k'$, so at a loss of $N^{C\varepsilon'}$ we may fix the values of $k-k_1$, $k_1-k_2$ and $k_2-k'$. Note that the matrices $\Fc\Hs^{N,N^{\varepsilon'}}$ and $\Fc\Ms$ satisfy the bounds (\ref{aux3})--(\ref{aux5}); moreover in (\ref{aux5}) we may replace the unfavorable exponents $\<\lambda\>^{2(1-b)}\<\lambda'\>^{-2b}$ by the favorable ones $\<\lambda\>^{2b}\<\lambda'\>^{-2(1-b)}$, at the price of replacing the right hand side by a small positive power $N^{C\kappa^{-1}}$, by repeating the interpolation argument in Section \ref{linearbound}. Using these bounds, we then see that the integral (\ref{express}) provides the required linear combination. Here summability of coefficients follows from the estimate
\begin{multline}\label{lem5.1:lambda}\int_{\Rb^2}A(\lambda_1)B(\lambda_1,\lambda_2)C(\lambda_2)\,\mathrm{d}\lambda_1\mathrm{d}\lambda_2\\
\lesssim\|\<\lambda_1\>^{-(1-b)}A(\lambda_1)\|_{L^2}\cdot \|\<\lambda_1\>^{b}\<\lambda_2\>^{-(1-b)}B(\lambda_1,\lambda_2)\|_{L^2}\cdot\|\<\lambda_2\>^bC(\lambda_2)\|_{L^2}\end{multline} and the improved versions of (\ref{aux3})--(\ref{aux5}). 
Recall that  $k-k_1$, $k_1-k_2$ and $k_2-k'$ are all fixed. We set that $\ell^*:= (k_1-k)+(k_2-k_1)+(k'-k_2)=k'-k$ and $m^*:=k_1-k_2$.
Finally, for fixed $(\lambda_1,\lambda_2)$\footnote{In fact, to reach the heart of the matter easily, we don't show all details about the $\lambda_1$, $\lambda_2$ here but it could be seen by using (\ref{lem5.1:lambda}) and its above argument about $\lambda's$.},
we set $\Yc_{\ell^*,m^*}(k,\lambda):= (\Fc\Hs^{N,N^{\varepsilon'}})_{kk_1}(\lambda,\lambda_1)(\Fc H^{N,N^{\varepsilon'}})_{k_2k'}(\lambda_2)$  and $\Rc_{\ell^*,m^*}(k):= (\Fc\Ms)_{k_1k_2}(\lambda_1,\lambda_2)$.
The factors coming from $H^{N,N^{\varepsilon'}}$ and $\Hs^{N,N^{\varepsilon'}}$ are independent from $(F_N)_k$, while the factor coming from $\Ms$ \emph{depends on $k_1$ only via the quantity $|k_1|^2-|k_2|^2$} in view of the definition (\ref{defM}), hence the desired decomposition is valid because $|k_1|^2-|k_2|^2$ equals $m^*\cdot k$ plus a constant once the above-mentioned difference vectors are all fixed. Also the bounds of $\Rc$ and $\Yc$ in (\ref{lem5.1:est1}) can be easily proved by the above setting of $\Rc$ and $\Yc$ together with the bounds (\ref{aux1})--(\ref{aux5}) (together with the $X^\alpha\to X^\alpha$ bounds) for the operators $\Hs^N$ and $\Ms$.
\end{proof}
\subsubsection{Term V}\label{subsec:TermV} Now, let us consider term V as defined in the introduction of Section \ref{sectionIII}.
In the following proof of estimating term V, we will fully use the cancellation in (\ref{unitarity}) together with Lemma \ref{decomposeQ}.
 We may assume $N_1=N_2=N$, because if $N_1\neq N_2$, then in later expansions we must have $k_1'\neq k_2'$ (so the cancellation in (\ref{unitarity}) is not needed), and the proof will go in the same way; if $N_1=N_2=N/2$ then the same cancellation holds and again we have the same proof. Now, recall that $\rho^N=\xi^N-\psi^N$, and that
\begin{equation}\label{xinmatrix2}(\xi^N)_k(t)=\sum_{k'}(M^N)_{kk'}(t)(F_N)_{k'},\quad (\psi^N)_k(t)=\sum_{k'}(H^N)_{kk'}(t)(F_N)_{k'},\end{equation} as in (\ref{raoterm}) and (\ref{defxin3}) and that $M^N$ and $H^N$ both satisfy the equality (\ref{unitarity}). Using this cancellation (when $k_1'=k_2'$ in the expansion) in the same way as Section \ref{sectionIV}, and by repeating the reduction steps before we can reduce to estimating the quantity that is either
\begin{multline}\label{quantity7}\Qc_k:=\sum_{\substack{k_1-k_2+k_3=k\\|k|^2-|k_1|^2+|k_2|^2-|k_3|^2=\Omega_0}}\eta\bigg(\frac{k_1-k_2}{N^\varepsilon}\bigg)V_{k_1-k_2}\\\times\sum_{k_1'\neq k_2'}Q_{k_1k_1'}(F_{N_1})_{k_1'}\cdot\overline{P_{k_2k_2'}(F_{N_2})_{k_2'}}\cdot\sum_{k_3'}h_{k_3k_3'}^{(3)}(F_{N_3})_{k_3'},
\end{multline} or
\begin{equation}\label{quantity8}
\Qc_k:=\sum_{\substack{k_1-k_2+k_3=k\\|k|^2-|k_1|^2+|k_2|^2-|k_3|^2=\Omega_0}}\eta\bigg(\frac{k_1-k_2}{N^\varepsilon}\bigg)V_{k_1-k_2}\cdot\widetilde{h}_{k_1k_2}\cdot\sum_{k_3'}h_{k_3k_3'}^{(3)}(F_{N_3})_{k_3'}
\end{equation} where
\[\widetilde{h}_{k_1k_2}=\sum_{k_1'}\bigg(\frac{1}{\<k_1'\>^2}-\frac{1}{\<k_1\>^2}\bigg)Q_{k_1k_1'}\overline{P_{k_2k_1'}}.\] Here in (\ref{quantity7}) and (\ref{quantity8}) the matrix $Q$ is coming from $Q^N$ where $Q_{k_1k_1'}=(\widehat{Q^N})_{k_1k_1'}(\lambda)$ for some fixed $\lambda$; similarly $P$ is coming from either $Q^N$ or $h^{N,L_2}$, and $h^{(3)}$ is coming from $h^{N,L_3}$ in the same way.

First we consider (\ref{quantity8}). By losing a power $N^{C\varepsilon}$ we may fix the values of $k_1-k_2$ and $k-k_3$, then we will estimate $\Qc$ using $\|h^{\mathrm{b}}\|_{k_1k_2\to kk_3}\lesssim N^{2+C\varepsilon}$, and we have these bounds
\[\sup_{k_3}\bigg|\sum_{k_3'}h_{k_3k_3'}^{(3)}(F_{N_3})_{k_3'}\bigg|\lesssim N^{O(\varepsilon_1)}\cdot N^{-1}L_3^{-1/2};\quad \|\widetilde{h}\|_{k_1k_2}\lesssim N^{O(\varepsilon_1)}\cdot L_{2}N^{-3}N^{1/2}L_2^{-1/2}\] (with $L_2=N$ if $P$ is coming from $Q^N$), where the first bound above follows from Proposition \ref{trim} for each fixed $k_3$, and the second bound follows from estimating $\|\widetilde{h}\|_{k_1k_2}\lesssim L_2N^{-3}\|Q\|_{k_1k_1'}\|P\|_{k_2\to k_2'}$. This leads to
\[\|\Qc\|_{\ell^2}\lesssim N^{O(\varepsilon)}\cdot N^{2}N^{-1}\cdot L_{2}N^{-3}\cdot N^{1/2}L_2^{-1/2}\lesssim N^{-1+C\varepsilon}\] which is enough.

Now we consider (\ref{quantity7}). If $P$ is coming from $Q^N$, then in (\ref{quantity7}) we may remove the condition $k_1'\neq k_2'$, reducing essentially to the expression in (\ref{eqnrhon3}) with both $w_1$ and $w_2$ replaced by $\rho^N$, which is estimated in the same way as in Section \ref{case2-1}. On the other hand, the term when $k_1'=k_2'$ can be estimates in the same way as (\ref{quantity8}) above. The same argument applies if $P$ is coming from $h^{N,L_2}$ and $\max(L_2,L_3)\geq N^{\varepsilon'}$, where we can gain a power $N^{-\varepsilon'/4}$ from either $L_2$ or $L_3$, or if $Q$ is coming from $Q^{N,\mathrm{rem}}$, where we can gain extra powers $N^{-\varepsilon'/4}$ using Lemma \ref{decomposeQ}.

Finally, consider (\ref{quantity7}), assuming $\max(L_2,L_3)\leq N^{\varepsilon'}$, and that $Q$ comes from $Q^{N,\ll}$ in Lemma \ref{decomposeQ}. By losing at most $N^{C\varepsilon'}$ we may fix the values of $k_1-k_2$, $k-k_3$, $k_2-k_2'$, $k_3-k_3'$, and consider one single component of $Q^{N,\ll}$ described as in Lemma \ref{decomposeQ}. Then there are only two independent variables---namely $k$ and $k_1$---and we essentially reduce (\ref{quantity7}) to
\begin{equation}
\label{eqn:Qk}\Qc_k=\widetilde{g_k}\cdot\sum_{k_1:\ell\cdot(k+k_1)=\Omega_0}\Ac\cdot\mathbf{1}_{k'_1-k_1=\ell^*}\cdot\frac{1}{\<k_1'\>\<k_2'\>}\Yc(k_1)\Rc(k_1)\overline{P_{k_2k_2'}}\cdot g_{k_1'}\overline{g_{k_2'}}.
\end{equation} Here $|\Ac|\lesssim 1$ is a non-probabilistic factor, $|\ell|, |\ell^*| \lesssim N^{\varepsilon'}$ are fixed vectors, $\Yc=\Yc(k_1)$ and $\Rc=\Rc(k_1)$ are as in Lemma \ref{decomposeQ}, and $P=P_{k_2k_2'}$ is defined as above. Moreover we know that $\Yc$ and $P$ are independent from $g_{k_1'}$ and $g_{k_2'}$, that $\Rc(k_1)$ depends only on $m^*\cdot k_1$ for some fixed vector $|m^*|\lesssim N^{\varepsilon'}$, and that $|P|\lesssim N^{O(\varepsilon)}$, $|\Yc|\lesssim N^{O(\varepsilon)}$ and $\|\Rc\|_{\ell^2}\lesssim N^{1/2+O(\varepsilon)}$ (after fixing $\lambda$ as before). Finally $\widetilde{g_k}$ in \ref{eqn:Qk} is $\sum_{k_3'}h_{k_3k_3'}^{(3)}(F_{N_3})_{k_3'}$  bounded by $|\widetilde{g_k}|\lesssim N^{-1}$.

Since $\Rc(k_1)$ only depends on $m^*\cdot k_1$, if we fix the value of $m^*\cdot k_1$ in the above summation, then $\Rc(k_1)$ can be extracted as a common factor and for the rest of the sum we can apply independence (using Proposition \ref{trim}) and get
\[\begin{aligned}
|\Qc_k|&\lesssim |\widetilde{g_k}|\cdot\sum_a |\Rc(a)|\cdot\bigg(\sum_{k_1\in S_{a,k}} \big|\frac{1}{\<k_1'\>\<k_2'\>}\Yc(k_1)\Rc(k_1)\overline{P_{k_2k_2'}}\big|^2\bigg)^{1/2}
\\
& \lesssim N^{-3+O(\varepsilon)}\cdot \sum_a |\Rc(a)|\cdot |S_{a,k}|^{1/2},
\end{aligned}\]
where $\Rc(a) = \Rc(k_1)$ for any $k_1\cdot m^*=a$ and $S_{a,k}:= \{k_1\in \mathbb Z^3: \ell\cdot(k_1+k)=\Omega_0, k_1\cdot m^*=a\}$.
 Note that in the above estimate we are dividing the set of possible $k_1$'s into subsets $S_{a,k}$ where $\ell\cdot k_1$ equals some constant, and $m^*\cdot k_1$ equals another constant, and that $S_{a,k}$ is either empty or has cardinality $\geq N^{1-C\varepsilon'}$. 
 When $S_{a,k}=\varnothing$, $|\Qc_k|=0$. When $S_{a,k}\neq\varnothing$, we have $|S_{a,k}|\geq N^{1-C\varepsilon'}$ and hence
\[|\Qc_k| \lesssim N^{C\varepsilon'-7/2}\cdot \sum_a |\Rc(a)|\cdot |S_{a,k}|=N^{C\varepsilon'-7/2}\cdot \sum_{k_1:\ell\cdot(k+k_1)=\Omega_0}|\Rc(k_1)|.\]
 Then, using Schur's bound, we get that
\[\|\Qc\|_{\ell^2}\lesssim N^{C\varepsilon'-7/2}N^2\|\Rc\|_{\ell^2}\lesssim N^{-1+C\varepsilon'}\] which is enough for (\ref{apriori1.8}). This completes the proof for $\rho^N$.
\subsection{An extra improvement} For the purpose of Section \ref{remainder}, we need an improvement for the $\rho^N$ bound in (\ref{apriori1.8}), namely the following.
\begin{prop}\label{improve} Let $N=M$, $Y\in\Rb$ be any constant, and consider $\rho^*$ defined by
\[(\rho^*)_k(t)=(\rho^N)_k(t)\cdot\mathbf{1}_{Y\leq |k|^2\leq Y+N^{\varepsilon'}}\] then $N$-certainly we can improve (\ref{apriori1.8}) to $\|\rho^*\|_{X^b(J)}\leq N^{-1/2+\varepsilon_1/2}$. Note that this bound id better than the bound for $z_N$ in (\ref{apriori1.8}) (which is better than the bound for $\rho^N$ in (\ref{apriori1.8})).
\end{prop}
\begin{proof} We only need to examining the terms I$\sim$V in the above proof. For terms I and IV and V (hence also III), in the above proof we already obtain bounds better than $N^{-1/2+\varepsilon_1/2}$, so these terms are acceptable, and we need to study term II. Note that the definition of $\rho^*$ restricts $k$ to a set $E$ of cardinality $\leq N^{1+C\varepsilon'}$ by the standard divisor counting bound.

Let $h^{\mathrm{b}}=h_{kk_1k_2k_3}^{\mathrm{b}}$ be the base tensor, which is supported in $|k_j|\lesssim N_j\lesssim N$ and $|k_1-k_2|\sim R$, such that in the support of $h^{\mathrm{b}}$ we have $k-k_1+k_2-k_3=0$ and $|k|^2-|k_1|^2+|k_2|^2-|k_3|^2=\Omega_0$. There are three cases in term II that need consideration:

(1) The case in Section \ref{case2-1}. Here the bound (\ref{rhoest1}) suffices unless $\max(N_2,R)\leq N^{C\varepsilon'}$; if this happens, note that in the above proof, (\ref{rhoest1}) follows from the estimate

\[\max\big(\big\|h_{kk_1k_2k_3}^{\mathrm{b}}\big\|_{kk_2k_3\to k_1},\|h_{kk_1k_2k_3}^{\mathrm{b}}\big\|_{kk_2\to k_1k_3}\big)\lesssim N^{1+\delta}\] assuming $\max(N_2,R)\leq N^{C\varepsilon'}$. However if we further require $k\in E$, then the right hand side of the above bound can be improved to $|E|^{1/2}=N^{1/2+C\varepsilon'}$, which leads to the desired improvement of (\ref{apriori1.8}).

(2) The \emph{case 1} in Section \ref{case2-2}. Here the bound (\ref{rhoest2}) suffices unless $R\leq N^{C\varepsilon'}$; if this happens, note that (\ref{rhoest2}) follows from the estimate
\[\max\big(\|h^{\mathrm{b}}\|_{kk_2k_3\to k_1},\|h^{\mathrm{b}}\|_{kk_2\to k_1k_3},\|h^{\mathrm{b}}\|_{kk_3\to k_1k_2},\|h^{\mathrm{b}}\|_{k\to k_1k_2k_3}\big)\lesssim N^{1+\delta}N_2\] assuming $R\leq N^{C\varepsilon'}$. However if we further require $k\in E$, then the right hand side can be improved to $|E|^{1/2}N_2=N^{1/2+C\varepsilon'}N_2$, which allows for the improvement.

(3) The \emph{case 2} in Section \ref{case2-2}. Here (\ref{rhoest3}) follows from the estimate $\|h^{\mathrm{b}}\|_{kk_2k_3\to k_1}\lesssim R^{-\beta}NR$. However if we further require $k\in E$, then the right hand side can be improved to $R^{-\beta}|E|^{1/2}R=R^{-\beta}N^{1/2+C\varepsilon'}R$, which allows for the improvement. This finishes the proof.
\end{proof}

\section{The remainder terms} \label{remainder}
Now we will prove the $z_N$ part of the bound (\ref{apriori1.8}), assuming $N=M$. We will prove it by a continuity argument, so we may assume (\ref{apriori1.8}) and only need to improve it using the equation (\ref{eqnzn}); note that the smallness factor is automatic as long as we use (\ref{eqnzn}), as explained before. As such, we can assume that each input factor $w_j$s on the right hand side of (\ref{eqnzn}) has one of the following four types, where in all cases we have $N_j\leq N$:

(i) Type (G), where we define $L_j=1$, and \begin{equation}\label{input1}(\widehat{w_j})_{k_j}(\lambda_j)=\mathbf{1}_{N_j/2< \langle k_j\rangle\leq N_j}\frac{g_{k_j}(\omega)}{\langle k_j\rangle}\widehat{\chi}(\lambda_j).
\end{equation}

(ii) Type (C), where
\begin{equation}\label{input2}(\widehat{w_j})_{k_j}(\lambda_j)=\sum_{N_j/2<\langle k'_j\rangle\leq N_j}h_{k_jk'_j}^{(j)}(\lambda_j,\omega)\frac{g_{k'_j}(\omega)}{\langle k'_j\rangle},\end{equation} with $h_{k_jk'_j}^{(j)}(\lambda_j,\omega)$ supported in the set $\big\{\frac{N_j}{2}<\langle k_j\rangle\leq N_j,\frac{N_j}{2}<\langle k'_j\rangle\leq N_j\big\}$, $\mathcal{B}_{\leq L_j}$ measurable for some $L_j\leq N_j/2$, and satisfying the bounds (where in the first bound we first fix $\lambda_j$, take the operator norm, and the take the $L^2$ norm in $\lambda_j$)
\begin{equation}\label{input2+}\|\langle \lambda_j\rangle^{1-b}h_{k_jk'_j}^{(j)}(\lambda_j)\|_{L_{\lambda_j}^2(\ell_{k_j}^2\to \ell_{k_j'}^2)}\lesssim L_j^{-1/2+3\varepsilon_1},\quad \| \langle \lambda_j\rangle^{b}h_{k_jk'_j}^{(j)}(\lambda_j)\|_{\ell_{k_jk'_j}^2L_{\lambda_j}^2}\lesssim N_j^{1+\delta}L_j^{-1/2+2\varepsilon_1}.\end{equation} Moreover using (\ref{apriori1.5}) we may assume $h^{(j)}$ is supported in $|k_j-k_j'|\lesssim N^\delta L_j$. Note that if $w_j$ is of type (G), $(\widehat{w_j})_{k_j}(\lambda_j)$ can be also expressed in the same form as (\ref{input2}) but with $h_{k_jk'_j}^{(j)} =\mathbf{1}_{k_j=k'_j}\cdot \widehat{\chi}(\lambda_j)$, except the second equation in (\ref{input2+}) is not true in this case.

(iii) Type (L), where  $(\widehat{w_j})_{k_j}(\lambda_j)$ 
is supported in $\{|k_j|\sim N_j\}$, and satisfies
\begin{equation}\label{input3}\|\langle\lambda_j\rangle^b(\widehat{w_j})_{k_j}(\lambda_j)\|_{\ell_{k_j}^2L_{\lambda_j}^2}\lesssim N_j^{-1/2+\varepsilon_1+\varepsilon_2}.
\end{equation}
Also such $w_j$ is a solution to the equation (\ref{eqnrhon}).

(iv) Type (D), where $(\widehat{w_j})_{k_j}(\lambda_j)$ is supported in $\{|k_j|\lesssim N_j\}$, and satisfies
\begin{equation}\label{input4}\|\langle\lambda_j\rangle^b(\widehat{w_j})_{k_j}(\lambda_j)\|_{\ell_{k_j}^2L_{\lambda_j}^2}\lesssim N_j^{-1/2+\varepsilon_1}.
\end{equation} 

Now, let the multilinear forms $\mathcal{M}^{\circ}$, $\mathcal{M}^{<}$, $\mathcal{M}^{>}$ and $\mathcal{M}^{\ll}$ be as in (\ref{reducedeqn2}), (\ref{lownonlin}) and (\ref{lownonlin1}). The terms on the right hand side of (\ref{eqnzn}), apart from the first term in the second line of (\ref{eqnzn}) which is trivially bounded, are the followings:

(1) The term
\[\mathrm{I}=\mathcal{I}_\chi\Pi_N \mathcal{M}^{>}(w_1,w_2,w_3)\] where $w_j$ can be any type and $\max(N_1,N_2,N_3)=N$.

(2) The term
\[\mathrm{II}=\mathcal{I}_\chi\Pi_N (\Mc^{<}-\mathcal{M}^{\ll})(w_1,w_2,w_3)\] where $w_j$ can be any type and $\max(N_1,N_2)=N$.

(3) The term
\[\mathrm{III}=\mathcal{I}_\chi\Delta_N \Mc^\circ(w_1,w_2,w_3)\] where $w_j$ can be any type and $\max(N_1,N_2,N_3)\leq N/2$.

(4) The term
\[\mathrm{IV}=\mathcal{I}_\chi\Pi_{N/2} \Mc^<(w_1,w_2,w_3)\] where $w_j$ can be any type and $\max (N_1,N_2)\leq N/2$ and $N_3=N$.

(5) The term
\[\mathrm{V}=\mathcal{I}_\chi\Pi_{N/2} \Mc^\ll(w_1,w_2,w_3)\] where $w_j$ can be any type and $\max (N_1,N_2)=N_3=N$.

(6) The term 
\[\mathrm{VI}=\mathcal{I}_\chi\Delta_{N} \Mc^<(w_1,w_2,w_3)\] where $w_1$ and $w_2$ can be any type, $w_3$ has type (D) and $\max (N_1,N_2)\leq N/2$ and $N_3=N$.

(7) The term 
\[\mathrm{VII}=\mathcal{I}_\chi\Delta_{N} \Mc^\ll(w_1,w_2,w_3)\] where $w_1$ and $w_2$ can be any type, $w_3$ has type (D) and $\max (N_1,N_2)=N_3=N$.

\smallskip

(8) The term VIII, which is the last two lines of the right hans side of (\ref{eqnzn}).

\smallskip
Our goal is to recover the bound for $z_N$ in (\ref{apriori1.8}) for each of the terms I--VIII above. In doing so we will consider two cases. First is the \emph{no-pairing} case, where if $w_1$ and $w_2$ are of type (C) or (G) and hence expanded as in (\ref{input2}), then we assume $k_1'\neq k_2'$; similarly if $w_2$ and $w_3$ are of type (C) or (G) then we assume $k_2'\neq k_3'$. The second case is the \emph{pairing} case which is when $k_1'=k_2'$ or $k_2'=k_3'$ (the \emph{over-pairing} case where $k_1'=k_2'=k_3'$ is easy and we shall omit it). We will deal with the no-pairing case for terms I--VII in Sections \ref{nopairing1}--\ref{subsec:gamma}, the pairing case for these terms in Section \ref{pairing}, and term VIII in Section \ref{subsec:linear}.
\subsection{No-pairing case}\label{nopairing1} We start with the no-pairing case.
\subsubsection{Preparation of the proof}
We start with some general reductions in the no-pairing case. Recall as in Section \ref{rem} that we can always gain a smallness factor from the short time $\tau\ll 1$, and can always ignore losses of $(N_*)^{C\kappa^{-1}}$ provided we can gain a power $N^{-\varepsilon/10}$ (which will be clear in the proof). We will consider $\widehat{\Ic_\chi\Mc^{(\star)}}(w_1,w_2,w_3)_k(\lambda)$ where $\Mc^{(\star)}$ can be one of $\Pi \mathcal{M}^{\circ}$, $\Pi \mathcal{M}^{<}$, $\Pi \mathcal{M}^{>}$, $\Pi\mathcal{M}^{\ll}$ and $\Pi (\Mc^{<}-\Mc^{\ll})$ with $\Pi$ being a general notation for projections for $\Pi_N$, $\Pi_{N/2}$ and $\Delta_N$,
\begin{multline}\label{sec5:mult1}
\widehat{\Ic\Mc^{(\star)}}(w_1,w_2,w_3)_k(\lambda)=\sum^{(\star)}_{\substack{(k_1,k_2,k_3)\\k=k_1-k_2+k_3, \\k_2\notin \{k_1, k_3\}}}\int\mathrm{d}\lambda_1\mathrm{d}\lambda_2\mathrm{d}\lambda_3\, I(\lambda,\Omega+\lambda_1-\lambda_2+\lambda_3)\,\\
\times V_{k_1-k_2}\cdot (\widehat{w_1})_{k_1}(\lambda_1)\, \overline{(\widehat{w_2})_{k_2}(\lambda_2)}\,(\widehat{w_3})_{k_3}(\lambda_3),
\end{multline}
where $\Omega = |k|^2-|k_1|^2+|k_2|^2-|k_3|^2$
and $\sum^{(\star)}$ is directly defined based on the definitions of $\mathcal{M}^{\circ}$, $\mathcal{M}^{<}$, $\mathcal{M}^{>}$ and $\mathcal{M}^{\ll}$ and the selection of $\Pi$. For example, if $\Mc^{(\star)}$ is $\Pi_N \Mc^>$, then 
there will be two more restrictions $|k|\leq N$ and $\<k_1-k_2\> > N^{1-\delta}$ in the sum $\sum^{(\star)}$. The other $\sum^{(\star)}$ will defined in the similar ways.

Before going into the different estimates for I--VII, we first make a few remarks.
\begin{itemize}
\item If a position $w_j$ has type (L) or (D), then in most cases we only need to consider type (L) terms since (\ref{input4}) is stronger than (\ref{input3}); there are exceptions that will be treated separately later.
\item the $w_j$ of type (G) can be considered as a special case of type (C) when $h^{(j)}_{k_j k'_j} (\lambda_j) = \mathbf{1}_{N_j/2< \langle k_j\rangle\leq N_j} \cdot \mathbf{1}_{k_j=k'_j}\cdot \widehat{\chi}(\lambda_j)$; if we avoid using the $\ell_{k_jk_j'}^2$ norm in (\ref{input2+}), then we only need to consider type (C) terms.
\item Term I can be estimates in the same way as term II. In fact the definition of $\Mc^>$ implies $\max(N_1,N_2)\geq N^{1-\delta}$, so we are essentially in (special case of) term II up to a possible loss $N^{C\delta}$ which will be negligible compared to the gain. Moreover, term V can be estimated similarly as term IV, see Section \ref{subsec:gamma}.
\item Terms VI and VII are readily estimated using the $X^\alpha\to X^\alpha$ bounds for the linear operator (\ref{linearmapping}) proved in Sections \ref{operbound} and \ref{linearbound}.
\end{itemize}

Based on these remarks, from now on we will consider terms II--IV (and VIII at the end), where the possible cases for the types of $(w_1,w_2,w_3)$ are 
(a) (C, C, C),
(b) (C, C, L),
(c) (C, L, C),
(d) (L, C, C),
(e) (L, L, C),
(f) (C, L, L),
(g) (L, C, L), 
and (h) (L, L, L).

In Section \ref{subsec:HH} we will estimate term II, which can be understood as high-high interactions in view of $\max(N_1,N_2)=N$, and noticing that assuming $k$ is the high frequency, then either $k_3$ is also high frequency or $|k_1-k_2|$ must be large. In Section \ref{subsec:gamma}, we will estimate terms III and IV by using a counting technique in a special situation called \emph{$\Gamma$-condition} (see (\ref{eq:Gamma})). In Section \ref{pairing} we consider the pairing case.
\subsection{High-high interactions}\label{subsec:HH}
We will estimate term II in this subsection. First we can repeat the arguments for $\lambda$, $\lambda_j$ and the Duhamel operator $I$ in (\ref{sec5:mult1}) as in Section 4 and 5. Namely, we first restrict to $|\lambda_j|\leq N^{100}$ and $|\lambda|,\,|\mu|\leq N^{100}$ where $\mu=\lambda-(\Omega+\lambda_1-\lambda_2+\lambda_3)$, and replace the unfavorable exponents $(1-b$ or $b$ depending on the context) by the favorable ones ($b$ or $1-b$), then exploit the resulting integrability in $\lambda_j$ to fix the values of $\lambda$, $\lambda_j$ and $\lfloor\mu\rfloor$. Then we reduce to the following expression where $\Omega_0$ is a fixed integer:
\begin{equation}\label{sec5:mult2}
\Xc_k := \sum_{k_1-k_2+k_3=k} h_{kk_1k_2k_3}^{\mathrm{b}}(\widehat{w_1})_{k_1}(\lambda_1)\, \overline{(\widehat{w_2})_{k_2}(\lambda_2)}\,(\widehat{w_3})_{k_3}(\lambda_3),
\end{equation} where $h^{\mathrm{b}}$ is the base tensor which contains the factors
\[V_{k_1-k_2}\cdot\mathbf{1}_{k_1-k_2+k'=k}\cdot \mathbf{1}_{|k|^2-|k_1|^2+|k_2|^2-|k'|^2=\Omega_0}.\] We assume $h^{\mathrm{b}}$ is supported in the set where $|k_j|\leq N_j$ and $\<k_1-k_2\>\sim R$ where $R$ is a dyadic number. Moreover we assume that $R$ and the support of $h^{\mathrm{b}}$ satisfies the conditions associated with the definition of some $\Mc^{(\star)}$. In view of the factor $|V_{k_1-k_2}|\sim R^{-\beta}$ in $h^{\mathrm{b}}$, we also define $h^{R,(\star)}:=R^\beta\cdot h^{\mathrm{b}}$, which is essentially the characteristic function of the set
\begin{equation}\label{sec5:set}
S^{R} =\left\{
\begin{array}{lr}
(k,k_1, k_2, k_3)\in (\mathbb{Z}^3)^4,\quad k_2\notin \{k_1, k_3\}\\
k=k_1-k_2+k_3, \quad |k|\leq N\\
|k|^2-|k_1|^2+|k_2|^2-|k_3|^2=\Omega_0\\
|k_j|\leq N_j\, (j\in \{1,2,3\}),\quad \<k_1-k_2\>\sim R 
\end{array} 
\right\},
\end{equation}possibly with extra conditions determined by the definition of $\Mc^{(\star)}$. We also define $S_k^{R}$ to be the set of $(k, k_1,k_2,k_3)\in S^{R}$ with fixed $k$, and similarly define $S_{k_1k_2}^R$ etc. Noticing that when $w_j$ has type (G), (C) or (L), we can further assume that $|k_j|>N/2$ in the definition of $S^R$.

The goal now is to bound the norm $\|\Xc_k\|_\ell^2$ or abbreviated $\|\Xc_k\|_k$, assuming $w_j$ satisfy the bounds (\ref{input1})--(\ref{input4}) but without the $\lambda_j$ component, for example (\ref{input4}) becomes $\|w_j\|_{k_j}\lesssim N_j^{-1/2+\varepsilon_1}$.
\subsubsection{Case (a): (C, C, C)}\label{casea}
In this case we have
\begin{equation}\label{xkexp1}
\Xc_k =R^{-\beta}\sum_{(k_1,k_2,k_3)} h^{R, (\star)}_{kk_1k_2k_3}\,\cdot \sum_{\substack{(k'_1,k'_2,k'_3)\\N_j/2<|k'_j|\leq N_j \\j\in\{1,2,3\}}} h^{(1)}_{k_1k'_1}\overline{h^{(2)}_{k_2k'_2}}h^{(3)}_{k_3k'_3}\,\frac{g_{k'_1}\overline{g_{k'_2}}g_{k'_3}}{\<k'_1\>\<k'_2\>\<k'_3\>},
\end{equation}
where $h^{(j)}_{k_jk'_j} = h^{(j)}_{k_jk'_j} (\omega)$ satisfies (\ref{input2+}) with some $N_j$ and $L_j\leq N_j/2$ for $1\leq j\leq 3$ and $h^{R, (\star)}_{kk_1k_2k_3}$ is defined as above.

To estimate $\|\Xc\|_k$ we would like to apply Proposition \ref{trim} and then Proposition \ref{merge}. Like in Sections \ref{Lbound} and \ref{estrhoN}, the way we apply Proposition \ref{trim} depends on the relative sizes of $N_j\,(1\leq j\leq 3)$. For example, if $N_1=N_2=N_3$ we shall apply Proposition \ref{trim} jointly in the $(k_1',k_2',k_3')$ summation in (\ref{xkexp1}); if $N_1=N_3>N_2$ we will first apply Proposition \ref{trim} jointly in the $(k_1',k_3')$ summation, then apply it in the $k_2'$ summation, if $N_3>N_1>N_2$ we will apply first in the $k_3'$ summation, then in the $k_1'$ summation, and then in the $k_2'$ summation, etc. The results in the end will be the same in all cases, so for example we will consider the case $N_3>N_1>N_2$. Now we have
\begin{equation}\label{eq:(a6)1}
\left\|\Xc_k\right\|_k =R^{-\beta}\bigg\|
\sum_{k'_3} \big(\sum_{k_3} \widetilde{H}_{kk_3} h^{(3)}_{k_3k'_3}\big) \frac{g_{k'_3}}{\<k'_3\>}
\bigg\|_k,
\end{equation}
where
\begin{equation}\label{eq:(a6)2}
\widetilde{H}_{kk_3} := \sum_{k'_1}\bigg(\sum_{k_1}
\mathring{H}_{kk_1k_3}h^{(1)}_{k_1k'_1}
\bigg)
\frac{g_{k'_1}}{\<k'_1\>},
\quad
\mathring{H}_{kk_1k_3}:= \sum_{k'_2} \bigg(\sum_{k_2}h^{R, (\star)}_{kk_1k_2k_3} \overline{h^{(2)}_{k_2k'_2}}\bigg) \frac{\overline{g_{k'_2}}}{\<k'_2\>}.
\end{equation}
By the independence between $g_{k'_3}$ and $\widetilde{H}_{kk_3} h^{(3)}_{k_3k'_3}$ since $N_3>N_1>N_2$, we apply Proposition \ref{trim} and Proposition \ref{merge} and get $\tau^{-1}N_*$-certainly that
\begin{equation}\label{eq:(a6)3}
\begin{aligned}
\left\|\Xc_k\right\|_k &\leq R^{-\beta}N_3^{-1} \cdot \bigg\|\sum_{k_1}
\widetilde{H}_{kk_3}h^{(3)}_{k_3k'_3}\bigg\|_{kk'_3}
\\
& \lesssim  R^{-\beta}N_3^{-1} \cdot \big\|h^{(3)}_{k_3k'_3}\big\|_{k'_3\to k_3} \big\|\widetilde{H}_{kk_3}\big\|_{kk_3}
\end{aligned}
\end{equation}
Similarly, by the independence between $g_{k'_1}$ and $\mathring{H}_{kk_1k_3} h^{(1)}_{k_1k'_1}$ since $N_1>N_2$, and also by the independence between $g_{k'_2}$ and $h^{R, (\star)}_{kk_1k_2k_3} \,\overline{h^{(2)}_{k_2k'_2}}$,
once again we can apply Proposition \ref{trim} and Proposition \ref{merge} to $\|\widetilde{H}_{kk_3}\|_{kk_3}$ and then to $\|\mathring{H}\|_{kk_1k_3}$. As a consequence, we have $\tau^{-1}N_*$-certainly
\begin{equation}\label{eq:(a6)4}
\begin{aligned}
\left\|\Xc_k\right\|_k\lesssim R^{-\beta} (N_1N_2N_3)^{-1} \cdot \bigg(\prod_{j=1}^3\big\|h^{(j)}_{k_jk'_j}\big\|_{k'_j\to k_j} \bigg) \cdot \left\| h^{R, (\star)}_{kk_1k_2k_3} \right\|_{kk_1k_2k_3}.
\end{aligned}
\end{equation}
In the other cases we get the same bound. Without loss of generality we may assume $N_1=N$, then using Lemma \ref{counting} we can estimate
\[\| h^{R, (\star)} \|_{kk_1k_2k_3}\lesssim N^\delta\cdot N_3^{3/2}\cdot RN_2\] which implies that $\|\Xc_k\|_k\lesssim N^{-1+C\delta}N_3^{1/2}\lesssim N^{-1/2+C\delta}$, which is enough for (\ref{apriori1.8}).
\subsubsection{Case (b): (C, C, L)}
In this case we have
\begin{equation}\label{eq:(b)1}
\Xc_k =R^{-\beta}\sum_{(k_1,k_2,k_3)} h^{R, (\star)}_{kk_1k_2k_3}\,\cdot \sum_{\substack{(k'_1,k'_2)\\N_j/2<|k'_j|\leq N_j }} h^{(1)}_{k_1k'_1}\overline{h^{(2)}_{k_2k'_2}}\,\frac{g_{k'_1}\overline{g_{k'_2}}}{\<k'_1\>\<k'_2\>}(w_3)_{k_3},
\end{equation}
where $h^{(j)}_{k_jk'_j} = h^{(j)}_{k_jk'_j} (\omega)$ satisfies (\ref{input2+}) with some $N_j$ and $L_j\leq N_j/2$ for $1\leq j\leq 2$ and the base tensor $h^{R, (\star)}_{kk_1k_2k_3}$ is defined as before. Clearly $\|\Xc_k\|_k$ can be bounded by $N_3^{-1/2+\varepsilon_1+\varepsilon_2}$ times the norm
\[R^{-\beta}\bigg\|\sum_{(k_1,k_2,k_3)} h^{R, (\star)}_{kk_1k_2k_3}\,\cdot \sum_{\substack{(k_1',k_2')\\N_j/2<|k'_j|\leq N_j}} h^{(1)}_{k_1k'_1}\overline{h^{(2)}_{k_2k'_2}}\,\frac{g_{k'_1}\overline{g_{k'_2}}}{\<k'_1\>\<k'_2\>}\bigg\|_{k\to k_3}.\] By applying Propositions \ref{trim} and \ref{merge} again, in the same manner as (\ref{casea}), we get that the above norm is bounded by
\[R^{-\beta}\cdot (N_1N_2)^{-1}\max(\|h^{R, (\star)}\|_{k\to k_1k_2k_3},\|h^{R, (\star)}\|_{kk_1\to k_2k_3},\|h^{R, (\star)}\|_{kk_2\to k_1k_3},\|h^{R, (\star)}\|_{kk_1k_2\to k_3}).\] By Lemma \ref{counting} we can conclude that
\[\max(\|h^{R, (\star)}\|_{k\to k_1k_2k_3},\|h^{R, (\star)}\|_{kk_1\to k_2k_3},\|h^{R, (\star)}\|_{kk_2\to k_1k_3},\|h^{R, (\star)}\|_{kk_1k_2\to k_3})\lesssim R\cdot\min(N_1,N_2),\] hence we easily get $\|\Xc_k\|_k\lesssim N^{-1+C\varepsilon_1}$, which is enough for (\ref{apriori1.8}).
\subsubsection{Cases (c): (C,L,C) and (d): (L,C,C)}\label{casec}
The estimates of Case (c) and Case (d) are similar with Case (b), so we will state the estimates in Case (c) and Case (d) without proofs. In Case (c) we get
\[\|\Xc_k\|_k\lesssim N_2^{-1/2+\varepsilon_1+\varepsilon_2}R^{-\beta}(N_1N_3)^{-1}\max(\|h^{R, (\star)}\|_{k\to k_1k_2k_3},\|h^{R, (\star)}\|_{kk_1\to k_2k_3},\|h^{R, (\star)}\|_{kk_3\to k_1k_2},\|h^{R, (\star)}\|_{kk_1k_3\to k_2})\] and in case (d) we get a similar bound, but with the subindices $1$ and $2$ switched.

Now by Lemma \ref{counting} we can obtain that
\[\max(\|h^{R, (\star)}\|_{k\to k_1k_2k_3},\|h^{R, (\star)}\|_{kk_3\to k_1k_2})\lesssim N^{C\delta}\cdot N_3\cdot\min(N_1,N_2),\]
\[\|h^{R, (\star)}\|_{kk_1\to k_2k_3}\lesssim N^{C\delta}\cdot \min(R\cdot\min(N_1,N_2),N_1N_3),\]
\[\|h^{R, (\star)}\|_{kk_1k_3\to k_2}\lesssim\min(R,N_1)N_3.\] In the first case we directly get 
\[\|\Xc_k\|_k\lesssim N_1^{-1}N_2^{-1/2+\varepsilon_1+\varepsilon_2}R^{-\beta}\cdot\min(N_1,N_2)
\] which is enough for (\ref{apriori1.8}) as $\max(N_1,N_2)=N$ and $R\geq N^\varepsilon$ in view of the definition of $\Mc^<-\Mc^\ll$. In the second case we get
\[\|\Xc_k\|_k\lesssim\min(R^{1-\beta}N_1^{-1}N_2^{1/2+\varepsilon_1+\varepsilon_2},R^{-\beta}N_2^{-1/2+\varepsilon_1+\varepsilon_2})\] which is also enough for (\ref{apriori1.8}) as $\max(N_1,N_2)=N$ and $R\geq N^\varepsilon$. In the third case we get
\[\|\Xc_k\|_k\lesssim N_2^{-1/2+\varepsilon_1+\varepsilon_2}\max(R,N_1)^{-1}\] which is also enough for (\ref{apriori1.8}). By switching the indices $1$ and $2$ we also get the same estimates in case (d).
\subsubsection{Case (e): (L,L,C)}\label{casee}
In this case we have
\begin{equation}\label{eq:(e)1}
\Xc_k =  \sum_{k'_3}\sum_{(k_1,k_2,k_3)}h^{R, (\star)}_{kk_1k_2k_3}\cdot (w_1)_{k_1}\overline{(w_2)_{k_2}} {h^{(3)}_{k_3k'_3}}\,\frac{g_{k'_3}}{\<k'_3\>},
\end{equation}
where $h^{(3)}_{k_3k'_3} = h^{(3)}_{k_3k'_3} (\omega)$ satisfies (\ref{input2+}) with some $N_3$ and $L_3\leq N_3/2$ and the base tensor $h^{R, (\star)}_{kk_1k_2k_3}$ is defined as before. By symmetry we may assume $N_1\leq N_2$, then by the same argument as above, using Propositions \ref{merge} and \ref{trim} we can bound
\[\|\Xc_k\|_k\lesssim(N_1N_2)^{-1/2+\varepsilon_1+\varepsilon_2}N_3^{-1}R^{-\beta}\cdot\max(\|h^{R, (\star)}\|_{kk_1k_3\to k_2},\|h^{R, (\star)}\|_{kk_1\to k_2k_3}).\] By Lemma \ref{counting} both tensor norms are bounded by $\min(N_1,R)N_3$; as $N_1\leq N_2$ (and hence $N_2=N$) and $R\geq N^{\varepsilon}$, it is easy to check that this bound is enough for (\ref{apriori1.8}).
\subsubsection{Cases (f): (C,L,L) and (g): (L,C,L)}
The estimates of Case (f) and (g) are similar with Case (e), so we will state the estimates of Case (f) and (g) directly. Again the two cases here only differs by switching indices $1$ and $2$, so we only consider case (f). Like in case (e) we get two bounds:
\[\|\Xc_k\|_k\lesssim (N_2N_3)^{-1/2+\varepsilon_1+\varepsilon_2}N_1^{-1}R^{-\beta}\max(\|h^{R, (\star)}\|_{kk_1k_2\to k_3},\|h^{R, (\star)}\|_{kk_2\to k_1k_3})\] and 
\[\|\Xc_k\|_k\lesssim (N_2N_3)^{-1/2+\varepsilon_1+\varepsilon_2}N_1^{-1}R^{-\beta}\max(\|h^{R, (\star)}\|_{kk_1k_3\to k_2},\|h^{R, (\star)}\|_{kk_1\to k_2k_3})\] Now if $N_3\geq N^{\varepsilon^2}$ we will apply the first bound and use that \[\max(\|h^{R, (\star)}\|_{kk_1k_2\to k_3},\|h^{R, (\star)}\|_{kk_2\to k_1k_3})\lesssim R\min(N_1,N_2),\] so the factor $N_1^{-1}N_2^{-1/2+\varepsilon_1+\varepsilon_2}\min(N_1,N_2)$, together with $N_3^{-1/2+\varepsilon_1+\varepsilon_2}$ where $N_3\geq N^{\varepsilon^2}$, provides the bound that is enough for (\ref{apriori1.8}). Moreover, the same bound also works if $N_2\leq N^{1-\varepsilon^2}$ (since in this case $N_1=N$).

If $N_3\leq N^{\varepsilon^2}$ and $N_2\geq N^{1-\varepsilon^2}$ we will apply the second bound and use that
\[\max(\|h^{R, (\star)}\|_{kk_1k_3\to k_2},\|h^{R, (\star)}\|_{kk_1\to k_2k_3})\lesssim N^{C\varepsilon^2}N_1\] assuming $N_3\leq N^{\varepsilon^2}$. This is also enough for (\ref{apriori1.8}) assuming $N_2\geq N^{1-\varepsilon^2}$ and $R\geq N^\varepsilon$.
\subsubsection{Case (h): (L,L,L)}
In this case we have\begin{equation}\label{eq:(h)1}
\Xc_k :=\sum_{(k_1,k_2,k_3)} h^{R, (\star)}_{kk_1k_2k_3}\cdot (w_1)_{k_1}\overline{(w_2)_{k_2}} (w_3)_{k_3},
\end{equation}
where the base tensor $h^{R,  (\star)}_{kk_1k_2k_3}$ is defined as before.
Then simply using Proposition \ref{merge} we get
\begin{equation}\label{eq:(h)2}
\|\Xc_k\|_k \lesssim  R^{-\beta} \cdot  (N_1N_2N_3)^{-1/2+\varepsilon_1+\varepsilon_2} \cdot \big\|h^{R, (\star)}_{kk_1k_2k_3} \big\|_{kk_2\to k_1k_3}.
\end{equation} By Lemma \ref{counting} we have $\|h^{R, (\star)}_{kk_1k_2k_3} \|_{kk_2\to k_1k_3}\lesssim (R\min(N_1,N_2))^{1/2}$, which implies
\[\|\Xc_k\|_k\lesssim R^{-\beta+1/2}\max(N_1,N_2)^{-1/2+C\varepsilon_1}\] which is enough for (\ref{apriori1.8}) because $\max(N_1,N_2)=N$ and $R\geq N^\varepsilon$.
\subsection{The $\Gamma$ condition terms}\label{subsec:gamma} In this section we estimate terms III and IV. These two terms are actually similar, and the key property they satisfy is the so-called \emph{$\Gamma$ condition}. Namely, due to the projections and assumptions on the inputs in terms III and IV, we have that \begin{equation}\label{eq:Gamma}\begin{aligned}
& &|k|^2\geq \Gamma \geq |k_3|^2,& \quad \text{for all }(k,k_1,k_2,k_3)\in S\\   &\text{or}& \quad |k|^2\leq  \Gamma\leq |k_3|^2,&\quad \text{for all }(k,k_1,k_2,k_3)\in S\end{aligned}
\end{equation} for some real number $\Gamma$, where $S$ is the support of the base tensor $h^{\mathrm{b}}$ (note that in term IV we may assume $w_3$ is not of type (D) as otherwise the bound follows from what we have already done, so here we may choose $\Gamma=(N/2)^2-1$).

To proceed, we return to $\widehat{\Ic_\chi\Mc^{(\star)}}(w_1,w_2,w_3)_k(\lambda)$ in (\ref{sec5:mult1}) where $\Omega=|k|^2-|k_1|^2+|k_2|^2-|k_3|^2$ and suppose $\mu = \lambda-(\Omega+\lambda_1-\lambda_2+\lambda_3)$ and then we have $|I|\lesssim \langle\lambda \rangle^{-1}\langle\mu \rangle^{-1}$ by (\ref{duhamleker}). Following the same reduction steps as before, we can assume $|\lambda|, |\lambda_j| (j=1,2,3), |\mu| \leq N^{100}$ and may replace the unfavorable exponents by the favorable ones. Now, instead of fixing each $\lambda_j$ and $\lambda$ and $\lfloor\mu\rfloor$, we do the following.

Without loss of generality, we may assume $|\lambda_3|$ is the maximum of all the parameters $|\lambda_j|$ and $|\lambda|$ and $|\mu|$; the other cases are treated similarly. We may fix a dyadic number $K$ and assume $|\lambda_3|\sim K$. Then, we may fix $\lambda_j\,(j\neq 3)$ and $\lambda$ and $\lfloor\mu\rfloor$, again using integrability in these variables, and exploit the weight $\<\lambda_3\>^b$ in the weighted norms in which $w_3$ is bounded, and reduce to an expression
\begin{equation}\label{sec5:mult2'}
\Xc_k := R^{-\beta} K^{-b}
  \sum_{k_1-k_2+k_3=k}\int\,\mathrm{d}\lambda_3\cdot h_{kk_1k_2k_3}^{R,K,(\star)}(\lambda_3)\cdot(w_1)_{k_1}\, \overline{(w_2)_{k_2}}\cdot(\widetilde{w_3})_{k_3}(\lambda_3),
\end{equation}
where $(\widetilde{w_3})_{k_3}(\lambda_3)=K^{b}(\widehat{w_3})_{k_3}(\lambda_3)$ and $h_{kk_1k_2k_3}^{R,K,(\star)}(\lambda_3)$ is essentially the characteristic function of the set (with possibly more restrictions according to the definition of $\Mc^{(\star)}$)
\begin{equation}\label{sec5:set'}
S^{R, K} =\left\{
\begin{array}{lr}
(k,k_1, k_2, k_3,\lambda_3)\in (\mathbb{Z}^3)^4\times\Rb,\quad k_2\notin \{k_1, k_3\}\\
k=k_1-k_2+k_3, \quad |k|\leq N\\
 |k|^2-|k_1|^2+|k_2|^2-|k_3|^2=-\lambda_3 + \Omega_0+O(1), \quad |\lambda_3|\sim K\\
|k_j|\leq N_j\, (j\in \{1,2,3\}),\quad \<k_1-k_2\>\sim R 
\end{array} 
\right\},
\end{equation}
where $\Omega_0$ is a fixed number such that $|\Omega_0|\lesssim K$. We also define the sets $S_k^{R,M}$ to be the set of $(k_1,k_2,k_3,\lambda_3)$ such that $(k, k_1, k_2,k_3,\lambda_3)\in S^{R,M}$ for fixed $k$ etc.. Note that when $w_j$ is of type (C), (G) or (L), we can further assume $\frac{N_j}{2}<|k_j|\leq N_j$.

The idea in estimating (\ref{sec5:mult2'}) is to view $(k_3,\lambda_3)$ as a whole (say denote it by $\widetilde{k_3}$), which will allow us to gain using the $\Gamma$ condition in estimating the norms of the base tensor $h^{R,K,(\star)}$. Though our tensors here  involve the variable $\lambda_3\in\Rb$, it is clear that Propositions \ref{merge0} and \ref{merge} still hold for such tensors, and Proposition \ref{trim} can also be proved by using a meshing argument (see Section \ref{rem}, where the derivative bounds in $\lambda_3$ is easily proved as all the relevant functions are compactly supported in physical space). Moreover, by the induction hypothesis and the manipulation above (for example with $Y^{1-b}$ norm replaced by $Y^b$ norm) we can also deduce corresponding bounds for $w_3=(w_3)_{\widetilde{k_3}}$ and the corresponding matrices such that $\widetilde{h}_{\widetilde{k_3}k_3'}^{(3)}$, for example $\|\widetilde{h}_{\widetilde{k_3}k_3'}^{(3)}\|_{k_3'\to\widetilde{k_3}}\lesssim L_3^{-1/2+3\varepsilon_1}$. Because of this, in the proof below we will simply write $\sum_{\widetilde{k_3}}$, while we actually mean $\sum_{k_3}\int\mathrm{d}\lambda_3$, so the proof has the same format as the previous ones.

We now consider the input functions. In term III, clearly $\max(N_1,N_2,N_3)\gtrsim N$; if $N_3\ll N$, then we must have $\max(N_1,N_2)\gtrsim N$ and $|k_1-k_2|\gtrsim N$, hence this term can be treated in the same way as term II. Therefore we may assume $N_3\sim N$, and clearly the same happens for term IV. If $\max(N_1,N_2)\gtrsim N$, then again using term II estimate we only need to consider the case where $|k_1-k_2|\lesssim N^{\varepsilon}$. This term can be treated using similar arguments as below and is much easier due to the smallness of $|k_1-k_2|$, so we will only consider the case $\max(N_1,N_2)\ll N$. In the same way we will not consider term V here. Finally, if $w_3=z_{N_3}$ with $N_3\sim N$, then (\ref{apriori1.8}) directly follows from the linear estimate proved in Section \ref{Lbound}, and the $\Gamma$ condition is not needed.

There are two cases: when $w_3$ has type (L) or or $w_3$ has type (C) (or (G)). In the latter case there are four further cases for the types of $w_1$ and $w_2$, which we will discuss below.
\subsubsection{The type (L) case} Suppose $w_3$ has type (L). Clearly if $\max(N_1,N_2)\geq N^{100\varepsilon_2}$ then (\ref{apriori1.8}) also follows from the linear estimates in Section \ref{Lbound} (because the difference between the $\rho^N$ bound and the $z_N$ bound in (\ref{apriori1.8}) is at most $N^{\varepsilon_2}$), so we may assume $\max(N_1,N_2)\leq N^{100\varepsilon_2}$. Then in (\ref{sec5:mult2'}) we may further fix the values of $(k_1,k_2)$ at the price of $N^{C\varepsilon_2}$, hence we may write\[\Xc_k=R^{-\beta}K^{-b}\sum_{\widetilde{k_3}}h(k,\widetilde{k_3})\cdot(\widetilde{w_3})_{\widetilde{k_3}}\] and by definition it is easy to see that $\|h\|_{\widetilde{k_3}\to k}\lesssim 1$. Then, (\ref{apriori1.8}) follows, using the bound for $w_3$, if $K\geq N^{\varepsilon_1^2}$. Finally, if $K\leq N^{\varepsilon_1^2}$, then we have $|\Omega|\lesssim N^{\varepsilon_1^2}$ where $\Omega=|k|^2-|k_1|^2+|k_2|^2-|k_3|^2$. Using the $\Gamma$ condition \ref{eq:Gamma}, we conclude that $|k_3|^2$ belongs to an interval of length $N^{O(\varepsilon_1^2)}$, so we can apply Proposition \ref{improve} to gain a power $N^{-\varepsilon_1/2}$, which covers the loss $N^{O(\varepsilon_2+\varepsilon_1^2)}$ and is enough for (\ref{apriori1.8}).
\subsubsection{The type (C,C,C) case} Now suppose $w_1$, $w_2$ and $w_3$ has type (C,C,C). By symmetry we may assume $N_1\leq N_2$. Then by the same argument as in Section \ref{casea}, we obtain that
\[\|\Xc_k\|_k\lesssim R^{-\beta}K^{-b}(N_1N_2N)^{-1}\|h^{R,K,(\star)}\|_{kk_1k_2\widetilde{k_3}}\cdot\|h^{(1)}\|_{k_1\to k_1'}\|h^{(2)}\|_{k_2\to k_2}\|h^{(3)}\|_{k_3'\to\widetilde{k_3}}.\] The last three factors are easily bounded by $1$, so it suffices to bound the tensor $h^{R,K,(\star)}$.

By definition, this is equivalent to counting the number of lattice points $(k,k_1,k_2,k_3)$ such that $k_1-k_2+k_3=k$ (and also satisfying the inequalities listed above) and $|\Omega|\lesssim K$. Note that \[||k_1|^2-|k_2|^2|\lesssim R\cdot\max(N_1,R):=K_1,\] so when $K\leq K_1$, by the $\Gamma$ condition, $|k|^2$ has at most $K_1$ choices, hence $k$ has at most $K_1N$ choices. Once $K$ is fixed, the number of choices fo $(k_1,k_2,k_3)$ is at most $KN_1^2R^2$, which leads to the bound
\[\|h^{R,K,(\star)}\|_{kk_1k_2\widetilde{k_3}}^2\lesssim N^{C\delta}\cdot KK_1NN_1^2R^2.\] If instead $K\geq K_1$, then $k$ has at most $KN$ choices, and once $k$ is fixed the number of choices for $(k_1,k_2,k_3)$ is at most $N_1^3R^3$, so we get \[\|h^{R,K,(\star)}\|_{kk_1k_2\widetilde{k_3}}^2\lesssim N^{C\delta}\cdot KNN_1^3R^3.\] In either way we get
\[\|\Xc_k\|_k\lesssim N^{C\varepsilon_2}N^{-1/2}\cdot\max(R,R^{1/2}N_1^{1/2})N_2^{-1}\] which is enough for (\ref{apriori1.8}) as $\max(R,N_1)\lesssim N_2$.
\subsubsection{The type (L,L,C) case}\label{typeLLC} Now suppose $w_1$, $w_2$ and $w_3$ has type (L,L,C). First assume $N_1\leq N_2$. The same arguments in Section \ref{casee} yields
\[\|\Xc_k\|_k\lesssim (N_1N_2)^{-1/2+\varepsilon_1+\varepsilon_2}N^{-1}R^{-\beta}K^{-b}\cdot \max(\|h^{R,K, (\star)}\|_{kk_1\widetilde{k_3}\to k_2},\|h^{R,K, (\star)}\|_{kk_1\to k_2\widetilde{k_3}}).\] The second norm above is easily bounded by $K^{1/2}RN_1$ using Lemma \ref{counting}, which is clearly enough for (\ref{apriori1.8}); for the first norm there are two ways to estimate.

The first way is to use Lemma \ref{counting} directly, without using $\Gamma$ condition, to get 
\[\|h^{R,K, (\star)}\|_{kk_1\widetilde{k_3}\to k_2}\lesssim K^{1/2}\min(R,N_1)N.\] The second way is to use the $\Gamma$ condition and first fix the value of $|k|^2$ and hence $k$, then count $(k_1,\widetilde{k_3})$. This yields
\[\|h^{R,K, (\star)}\|_{kk_1\widetilde{k_3}\to k_2}\lesssim K^{1/2}N^{1/2}(R+R^{1/2}N_1^{1/2})\min(R,N_1)^{1/2}\] assuming $K\leq K_1$, and a better bound assuming $K\geq K_1$.  Now, plugging in the second bound yields
\[\|\Xc_k\|_k\lesssim (N_1N_2)^{-1/2+\varepsilon_1+\varepsilon_2}N^{-1}R^{-\beta}K^{-b}\cdot K^{1/2}N^{1/2}(R+R^{1/2}N_1^{1/2})\min(R,N_1)^{1/2},\] which can be shown to be $\lesssim N^{-1/2}$ using the fact $\max(R,N_1)\leq N_2$ and by considering whether $R\geq N_1$ or $R\leq N_1$. Moreover the same estimate can be checked to work if $N_1\leq N_1^{1.1}$. If $N_1\geq N_2^{1.1}$ we can switch the subscripts $1$ and $2$, in which case we have the weaker bound \[\|\Xc_k\|_k\lesssim (N_1N_2)^{-1/2+\varepsilon_1+\varepsilon_2}N^{-1}R^{-\beta}K^{-b}\cdot K^{1/2}N^{1/2}(R+R^{1/2}N_2^{1/2})\min(R,N_2),\]without the $1/2$ power in the last factor, however this is still $\lesssim N^{-1/2}$ provided $N_1\geq N_2^{1.1}$.
\subsubsection{The type (L,C,C) and (C,L,C) cases}\label{hardnorm} Now suppose $w_1$, $w_2$ and $w_3$ has type (L,C,C); the case (C,L,C) is treated similarly. Here the same arguments in Section \ref{casec} implies
\begin{multline}\|\Xc_k\|_k\lesssim N_1^{-1}N_2^{-1/2+\varepsilon_1+\varepsilon_2}N^{-1}R^{-\beta}K^{-b}\\\times\max(\|h^{R,K,(\star)}\|_{kk_1\widetilde{k_3}\to k_2},\|h^{R,K,(\star)}\|_{k\to k_1k_2\widetilde{k_3}},\|h^{R,K,(\star)}\|_{kk_1\to k_2\widetilde{k_3}},\|h^{R,K,(\star)}\|_{k\widetilde{k_3}\to k_1k_2}).\end{multline} The two norms $k\to k_1k_2\widetilde{k_3}$ and $kk_1\to k_2\widetilde{k_3}$ can be estimated by $K^{1/2}R\min(N_1,N_2)$, using Lemma \ref{counting} only and without , which is clearly enough for (\ref{apriori1.8}). For the $kk_1\widetilde{k_3}\to k_2$ norm we can use the estimates in Section \ref{typeLLC} and get
\[\|h^{R,K, (\star)}\|_{kk_1\widetilde{k_3}\to k_2}\lesssim K^{1/2}N^{1/2}(R+R^{1/2}N_1^{1/2})\min(R,N_1)\lesssim K^{1/2}N^{1/2}RN_1\] up to $N^{C\delta}$ losses, which yields 
\[\|\Xc_k\|_k\lesssim R^{1-\beta}N^{-1/2}N_2^{-1/2+\varepsilon_1+\varepsilon_2}\]and is also enough for (\ref{apriori1.8}). Finally we consider the $k\widetilde{k_3}\to k_1k_2$ norm. By Schur's bound and using the $\Gamma$ condition we can get
\[\|h^{R,K,(\star)}\|_{k\widetilde{k_3}\to k_1k_2}\lesssim\min(N_1,N_2)\cdot (R+R^{1/2}\min(N_1,N_2)^{1/2})N^{1/2}\] (note the absence of $K$ on the right hand side) if $K\leq K_1:=R^2+R\min(N_1,N_2)$, and
\[\|h^{R,K,(\star)}\|_{k\widetilde{k_3}\to k_1k_2}\lesssim \min(N_1,N_2)\cdot K^{1/2}N^{1/2}\] if $K\geq K_1$. The second bound is obviously enough for (\ref{apriori1.8}); by examining the relation between $N_1$ and $N_2$, we see that the first bound is also enough if $\max(K,R)\geq N^{\varepsilon}$.

Finally, suppose $K,R\leq N^{\varepsilon}$, then by losing $N^{C\varepsilon}$ we may fix the values of $k_1-k_2$ and $\Omega=|k|^2-|k_1|^2+|k_2|^2-|k_3|^2$. Here we will improve the bound on the $k\widetilde{k_3}\to k_1k_2$ norm. Namely, when $(k_1,k_2)$ is fixed, let $\ell=k_1-k_2$ with $0<|\ell|\leq N^{\varepsilon}$, then the value of $k\cdot \ell$ is also fixed. Moreover, by the $\Gamma$ condition we know that $|k|^2$ belongs to an interval of length $O(\min(N_1,N_2))$. Once $|k|^2$ and $k\cdot \ell$ are fixed, $k$ will be determined by a lattice point on a two-dimensional ellipse of radius $O(N)$, and the number of such points is at most $N^{2/3}$ by a classical geometric argument (see for example \cite{FOSW}, Lemma 4.1). This leads to the improved bound
\[\|h^{R,K, (\star)}\|_{k\widetilde{k_3}\to k_1k_2}\lesssim N^{C\varepsilon}\min(N_1,N_2)^{3/2}N^{1/3},\] which is then enough for (\ref{apriori1.8}).
\subsection{The pairing case}\label{pairing} Now we consider the pairing case where we may expand some $w_j$ as in (\ref{input2}) and assume either $k_1'=k_2'$ or $k_2'=k_3'$. In this case we will use the cancellation (\ref{unitarity}) as before. First consider term II; there are four different cases.
\subsubsection{Case (C,C,C): $k_1'=k_2'$} Suppose each $w_j$ has type (C) (or (G)) and assume $k_1'=k_2'$, so in particular $N_1=N_2$. Since we are considering term II, we must have $N_1=N_2=N$. Then exploiting (\ref{unitarity}) like before, we can reduce to estimating the quantity
\begin{equation}
\Xc_k =R^{-\beta}\sum_{(k_1,k_2,k_3)} h^{R, (\star)}_{kk_1k_2k_3}\,\cdot \sum_{\substack{(k'_1,k'_3)\\N_j/2<|k'_j|\leq N_j}} \bigg(\frac{1}{\langle k_1'\rangle^2}-\frac{1}{\langle k_1\rangle^2}\bigg) h^{(1)}_{k_1k'_1}\overline{h^{(2)}_{k_2k'_1}}h^{(3)}_{k_3k'_3}\,\frac{g_{k'_3}}{\<k'_3\>}.
\end{equation} Note that we may assume $|k_1-k_1'|\lesssim L_1N^{\delta}$ and $|k_2-k_1'|\lesssim L_2N^{\delta}$, and $|k_1-k_2|\sim R$, so at a loss of $N^{C\delta}$ we may assume $R\lesssim\max(L_1,L_2)$, and \[\bigg|\frac{1}{\langle k_1'\rangle^2}-\frac{1}{\langle k_1\rangle^2}\bigg|\lesssim N^{-3}(R+\min(L_1,L_2)).\] Therefore, the matrix
\begin{equation}\label{defnewh1}\widetilde{h}_{k_1k_2}=\sum_{k_1'}\bigg(\frac{1}{\langle k_1'\rangle^2}-\frac{1}{\langle k_1\rangle^2}\bigg) h^{(1)}_{k_1k'_1}\overline{h^{(2)}_{k_2k'_1}}\end{equation} is bounded (up to loss $N^{C\delta}$) by
\[\|\widetilde{h}\|_{k_1k_2}\lesssim N^{-2}(R+\min(L_1,L_2))\cdot(L_1L_2)^{-1/2+3\varepsilon_1}.\] Note that here $h^{(1)}$ and $h^{(2)}$ cannot both be identity, so we may always estimate the non-identity one in the Hilbert-Schmidt $\ell^2$ norm. Using Propositions \ref{merge} and \ref{trim}, we can estimate
\[\|\Xc_k\|_k\lesssim R^{-\beta}N_3^{-1}\|\widetilde{h}\|_{k_1k_2}\cdot\|h^{(3)}\|_{k_3\to k_3'}\cdot\max(\|h^{R, (\star)}\|_{k\to k_1k_2k_3},\|h^{R, (\star)}\|_{kk_3\to k_1k_2}).\] Both norms are bounded by $NN_3$, so
\[\|\Xc_k\|_k\lesssim N^{-1}R^{-\beta}(R+\min(L_1,L_2))\cdot(L_1L_2)^{-1/2+3\varepsilon_1}\] which is enough for (\ref{apriori1.8}).
\subsubsection{Case (C,C,C): $k_2'=k_3'$} Now suppose each $w_j$ has type (C) or (G), and assume $k_2'=k_3'$, then $N_2=N_3$ and $\max(N_1,N_2)=N$. In this case we do not need to use the cancellation (\ref{unitarity}). The same argument as above yields
\[\|\Xc_k\|_k\lesssim N_1^{-1}R^{-\beta}\|\widetilde{h}\|_{k_2k_3}\|h^{(1)}\|_{k_1\to k_1'}\cdot\max(\|h^{R, (\star)}\|_{k\to k_1k_2k_3},\|h^{R, (\star)}\|_{kk_1\to k_2k_3})\] where $\widetilde{h}_{k_2k_3}$ is the matrix
\[\widetilde{h}_{k_2k_3}=\sum_{k_2'}\frac{1}{\<k_2'\>^2}\overline{h_{k_2k_2'}^{(2)}}h_{k_3k_2'}^{(3)}\] and satisfies $\|\widetilde{h}\|_{k_2k_3}\lesssim N_2^{-1}$. As both norms of $h^{R,(\star)}$ are bounded by $R\min(N_1,N_2)$, we get that
\[\|\Xc_k\|_k\lesssim N^{C\delta}\cdot(N_1N_2)^{-1}\min(N_1,N_2)\] which is enough for (\ref{apriori1.8}) as $\max(N_1,N_2)=N$.
\subsubsection{Case (C,C,L): $k_1'=k_2'$} Here assume that $w_1$ and $w_2$ has type (C) or (G), $w_3$ has type (L) (or (D)), and $k_1'=k_2'$. Then we have
\begin{equation}
\Xc_k =R^{-\beta}\sum_{(k_1,k_2,k_3)} h^{R, (\star)}_{kk_1k_2k_3}\,\cdot \sum_{N_1/2<|k'_1|\leq N_1} \bigg(\frac{1}{\langle k_1'\rangle^2}-\frac{1}{\langle k_1\rangle^2}\bigg) h^{(1)}_{k_1k'_1}\overline{h^{(2)}_{k_2k'_1}}(w_3)_{k_3}.
\end{equation} Hence we can easily estimate
\[\|\Xc_k\|_k\lesssim R^{-\beta}\|\widetilde{h}\|_{k_1k_2}\|w_3\|_{k_3}\cdot\|h^{R,(\star)}\|_{k\to k_1k_2k_3}\] where $\widetilde{h}$ is defined as in (\ref{defnewh1}). This yields
\[\|\Xc_k\|_k\lesssim N^{-1}(R+\min(L_1,L_2))(L_1L_2)^{-1/2+3\varepsilon_1}R^{-\beta}N_3^{-1/2+\varepsilon_1+\varepsilon_2}\min(R,N_3)\] as $N_1=N_2=N$, using also Lemma \ref{counting}. Since $R\lesssim\max(L_1,L_2)$, by considering the relative sizes between $R$ and $\min(L_1,L_2)$ we can check that this term is always bounded by $N^{-1+C\varepsilon_1}$, which is enough for (\ref{apriori1.8}).
\subsubsection{Case (L,C,C): $k_2'=k_3'$} Here we assume that $w_2$ and $w_3$ has type (C) or (G), $w_1$ has type (L) (or (D)), and $k_2'=k_3'$, then $N_2=N_3$ and $\max(N_1,N_2)=N$. In this case we will need to use the cancellation (\ref{unitarity}). Like before we can reduce to estimating the quantity
\begin{equation}
\Xc_k =R^{-\beta}\sum_{(k_1,k_2,k_3)} h^{R, (\star)}_{kk_1k_2k_3}\,\cdot (w_1)_{k_1}\sum_{N_2/2<|k'_2|\leq N_2} \bigg(\frac{1}{\langle k_2'\rangle^2}-\frac{1}{\langle k_2\rangle^2}\bigg) h^{(2)}_{k_2k'_2}\overline{h^{(3)}_{k_3k'_2}}.
\end{equation} Denote
\[\widetilde{h}_{k_2k_3}=\sum_{N_2/2<|k'_2|\leq N_2} \bigg(\frac{1}{\langle k_2'\rangle^2}-\frac{1}{\langle k_2\rangle^2}\bigg) h^{(2)}_{k_2k'_2}\overline{h^{(3)}_{k_3k'_2}},\] then similarly we have
\[\|\widetilde{h}\|_{k_2k_3}\lesssim N_2^{-3}\max(L_2,L_3)\cdot N_2(L_2L_3)^{-1/2+3\varepsilon_1}\lesssim N_2^{-3/2+C\varepsilon_1},\] hence
\[\|\Xc_k\|_k\lesssim N_1^{-1/2+\varepsilon_1+\varepsilon_2}N_2^{-3/2+C\varepsilon_1}R^{-\beta}\|h^{R,(\star)}\|_{k\to k_1k_2k_3}.\] Using that
\[\|h^{R,(\star)}\|_{k\to k_1k_2k_3}\lesssim \min(N_2,R)\cdot\min(N_1,N_2)\] and that $R\geq N^\varepsilon$, by considering the relative size between $N_1$, $N_2$ and $R$, it is easy to check that this bound is enough for (\ref{apriori1.8}) when $\max(N_1,N_2)=N$.
\subsubsection{The Gamma condition term} Finally we consider terms III and IV with pairing. Note that as in Section \ref{subsec:gamma} we may assume $N_3\sim N$ and $\max(N_1,N_2)\ll N$, hence the only possibility of pairing is $k_1'=k_2'$ (so $N_1=N_2$). Moreover if $w_3$ has type (L) or (D) the proof can be done as in Section \ref{subsec:gamma} above, so we only need to consider the case of type (C,C,C) and $k_1'=k_2'$ (in particular $N_1=N_2$). Like in Section \ref{subsec:gamma} we can reduce to the quantity
\[\Xc_k=R^{-\beta}K^{-b}\sum_{(k_1,k_2,\widetilde{k_3})}h_{kk_1k_2\widetilde{k_3}}^{R,K,(\star)}\sum_{k_3'}\widetilde{h}_{k_1k_2}\widetilde{h}_{\widetilde{k_3}k_3'}^{(3)}(F_{N_3})_{k_3'}\] where $\widetilde{h}$ is defined as in (\ref{defnewh1}). Using Propositions \ref{merge} and \ref{trim} we get
\[\|\Xc_k\|_k\lesssim R^{-\beta}K^{-b}N^{-1}\|\widetilde{h}\|_{k_1k_2}\|\widetilde{h}^{(3)}\|_{k_3'\to \widetilde{k_3}}\cdot\max(\|h^{R,K,(\star)}\|_{k\to k_1k_2\widetilde{k_3}},\|h^{R,K,(\star)}\|_{k\widetilde{k_3}\to k_1k_2}).\] The $k\to k_1k_2\widetilde{k_3}$ norm can be bounded by $K^{1/2}RN_1$ which is clearly enough for (\ref{apriori1.8}); for the $k\widetilde{k_3}\to k_1k_2$ norm we use the bound obtained in Section \ref{hardnorm} to get
\[\|h^{R,K,(\star)}\|_{k\widetilde{k_3}\to k_1k_2}\lesssim R^{1/2}N_1^{3/2}N^{1/2}K^{1/2}\] hence
\[\|\Xc_k\|_k\lesssim R^{-\beta}K^{-b}N^{-1}\cdot N_1^{-2}(R+\min(L_1,L_2))(L_1L_2)^{-1/2+3\varepsilon_1}\cdot R^{1/2}N_1^{3/2}N^{1/2}K^{1/2}\] with a possible loss of $N^{C\delta}$, which is enough for (\ref{apriori1.8}) by considering the relative size between $R$ and $\min(L_1,L_2)$, using also that $R\lesssim\max(L_1,L_2)$.
\subsection{Term VIII}\label{subsec:linear}
Finally, we will estimate term VIII, which is the last two lines of (\ref{eqnzn}). This is an easier term and most part of this term can be estimated using similar arguments as the above proof, so we will not detail them out. In fact, this term can be decomposed into expressions in the form
\begin{equation}\label{termviii}(\mathrm{VIII})_k(t)=-i\int_0^t\sum_\ell V_{k-\ell}(w_1)_k(t')\overline{(w_2)_\ell(t')}(w_3)_\ell(t').\end{equation} Here, if $w_2$ and $w_3$ both have type (C) or (G), then we can expand them as in (\ref{input2}) and exploit the independence if $k_2'\neq k_3'$, and exploit the cancellation (\ref{unitarity}) if $k_2'=k_3'$ (note that here $k_2=k_3=\ell$, so the right hand side of (\ref{unitarity}) is in fact $1$ instead of $0$, but this cancels with the term $-1/\<\ell\>^2$ which is subtracted in (\ref{eqnzn}); this is also the reason why the renormalization term), and the rest of proof can go just like before.

The hardest term in VIII in fact is the term where $w_1$ has type (G), one of $w_2$ and $w_3$ has type (G), and the other has type (L) or (D) in (\ref{termviii}). For such terms, standard estimates will fall short by a power $N^{1-\beta}$ as $\beta<1$; however since $1-\beta\ll\varepsilon_2$ by our choice, this can be controlled if we gain a power $N^{\varepsilon_2/10}$ from elsewhere. If either $w_2$ or $w_3$ has type (L), then we can plug in the equation satisfies by $\rho^N$ and estimate like Section \ref{estrhoN} to gain this extra power\footnote{There is also a term which is essentially $\Hs^{N,N^{\varepsilon'}}$ applied to $z_N$, which can also be treated using the Hilbert-Schmidt bound for the matrices $\Hs^{N,N^{\varepsilon'}}$. We omit the details.}.

So the only bad term is when $w_1$ has type (G), and when one of $w_2$ and $w_3$ has type (G), and the other has type (D). Let this term be $z_N^*$, then $z_N-z_N^*$ satisfies (\ref{apriori1.8}); due to the symmetry between $w_2$ and $w_3$ in (\ref{termviii}), we see that $(z_N^*)_k\in(-ig_k)\cdot\Rb$. Then, if we replace the type (D) term (which is $z_{N'}$ for some $N'$) with $z_{N'}-z_{N'}^*$ the resulting contribution will satisfy (\ref{apriori1.8}), while if we replace this term by $z_{N'}^*$, the net contribution to term VIII, after exploiting symmetry between $w_2$ and $w_3$, will be
\[\mathrm{Re}(\overline{g_\ell}\cdot (z_N^*)_\ell(t'))=0.\] Therefore, in any case, we can control this term by (\ref{apriori1.8}). This finishes the proof of Proposition \ref{mainprop} and hence Theorem \ref{main}.

\end{document}